\documentclass[12pt,a4paper]{article}
\usepackage{amssymb,amsmath,amsthm}
\usepackage{latexsym}
\usepackage{color}

\newtheorem{theorem}{Theorem}
\newtheorem{lemma}{Lemma}
\newtheorem{proposition}{Proposition}
\newtheorem{corollary}{Corollary}
\newtheorem{remark}{Remark}
\numberwithin{equation}{section}
\numberwithin{theorem}{section}
\numberwithin{lemma}{section}
\numberwithin{proposition}{section}
\numberwithin{corollary}{section}
\numberwithin{remark}{section}

\begin{document}
\title{Stability of time-dependent motions for
fluid-rigid ball interaction}
\author{Toshiaki Hishida \\
Graduate School of Mathematics \\
Nagoya University \\
Nagoya 464-8602, Japan \\
\texttt{hishida@math.nagoya-u.ac.jp} \\
}
\date{}
\maketitle
\begin{abstract}
We aim at the stability of time-dependent motions, such as 
time-periodic ones, of a rigid body in a viscous fluid filling 
the exterior to it in 3D. The fluid motion obeys the incompressible
Navier-Stokes system, whereas the motion of the body is governed 
by the balance for linear and angular momentum.
Both motions are affected by each other at the boundary.
Assuming that the rigid body is a ball, we adopt a monolithic approach 
to deduce $L^q$-$L^r$ decay estimates of solutions to a non-autonomous
linearized system. 
We then apply those estimates to the full nonlinear initial value 
problem to find temporal decay properties of the disturbance. 
Although the shape of the body is not allowed to be arbitrary,
the present contribution is the first attempt at analysis of the 
large time behavior of solutions around nontrivial basic states, 
that can be time-dependent, for the fluid-structure interaction problem 
and provides us with a stability theorem which is indeed new even for 
steady motions under the self-propelling condition or with wake structure.


\end{abstract}

%% 1
\section{Introduction}
\label{intro}

We study the stability of nontrivial basic motions
for the fluid-rigid ball interaction in 3D within the framework
of $L^q$-theory.
This work is inspired by Ervedoza,
Maity and Tucsnak \cite{EMT}, who have successfully established the $L^q$-stability
of the rest state.
The novelty of our study is to develop analysis even for
time-dependent basic motions such as time-periodic (and almost periodic) ones.
We emphasize that what is done here is already new when the basic motion
is a nontrivial steady state.
On the other hand, the shape of a single rigid body
is not allowed to be arbitrary; in fact, it is assumed to be a ball
in the present paper.
Later, we will explain the reason why and mention the existing literature,
in particular, \cite{Ga-new,MT-new} by Galdi and by Maity and Tucsnak,
about the related issue for the case of arbitrary shape.
To fix the idea, in what follows, we are going to discuss
the self-propelled regime as a typical case, nevertheless,
the approach itself can be adapted to other situations
of the fluid-rigid ball interaction, in which physically relevant nontrivial
basic motions could be available, as long as the system \eqref{eq-perturb} below for disturbance is the same.

When the rigid body is a ball, it is reasonable to take the frame 
attached to the center of mass of the ball just by translation to reduce the problem to the one
in a time-independent domain.
Let $B$ be the open ball centered at the origin with radius 1, that is
identified with a rigid body whose density $\rho>0$ is assumed to be
a constant.
In the resultant problem after the change of variable,
a viscous incompressible fluid
occupies the exterior domain 
$\Omega=\mathbb R^3\setminus \overline{B}$ and the motion
obeys the Navier-Stokes system with an additional nonlinear term
(due to the change of variable),
where the fluid velocity and pressure are denoted by
$u(x,t)\in \mathbb R^3$ and $p(x,t)\in\mathbb R$,
whereas $\eta(t)\in\mathbb R^3$ and $\omega(t)\in\mathbb R^3$ respectively
stand for the translational and angular velocities of the rigid ball
and are governed by the balance for linear and angular momentum.
The fluid velocity $u$ meets the rigid motion $\eta+\omega\times x$
(with $\times$ being the vector product)
at the boundary $\partial\Omega$, where an extra velocity
$u_*(x,t)\in \mathbb R^3$ generated by the body is allowed to be involved.
Indeed, the velocity $u_*$ 
plays an important role within the self-propelled regime,
in which there are no external force and torque so that the body moves
due to an internal mechanism described by $u_*$ through interaction of
the fluid-rigid body, see Galdi \cite{Ga99, Ga02} for the details.
We assume that $u_*$ is tangential to the boundary
$\partial\Omega$, that is, $\nu\cdot u_*|_{\partial\Omega}=0$.
Here, 
$\nu$ stands for the unit normal to
$\partial\Omega$ directed toward the interior of the rigid body;
indeed, $\nu=-x$ at $\partial\Omega$ 
being the unit sphere.
Consequently, the unknowns $u,\,p,\,\eta$ and $\omega$ obey
(\cite{EMT, Ga02})
\begin{equation} 
\begin{split}
&\partial_tu+(u-\eta)\cdot\nabla u=\Delta u-\nabla p, \qquad 
\mbox{div $u$}=0 \quad\mbox{in $\Omega\times I$},  \\
&u|_{\partial\Omega}=\eta+\omega\times x+u_*, \qquad 
u\to 0 \quad\mbox{as $|x|\to\infty$},  \\
&m\frac{d\eta}{dt}+\int_{\partial\Omega}\mathbb S(u,p)\nu\,d\sigma=0, \\ 
&J\frac{d\omega}{dt}+\int_{\partial\Omega}x\times \mathbb S(u,p)\nu\,d\sigma=0,
\end{split}
\label{eq-fixed}
\end{equation}
where $I$ is the whole or half time axis in $\mathbb R$,
$\mathbb S(u,p)$ denotes the Cauchy stress tensor, that is,
\begin{equation}
\mathbb S(u,p)=2Du-p\,\mathbb I, \qquad Du=\frac{\nabla u+(\nabla u)^\top}{2}
\label{stress}
\end{equation}
with $\mathbb I$ being $3\times 3$ unity matrix.
Here and hereafter, $(\cdot)^\top$ stands for the transpose.
The mass and tensor of inertia of the rigid ball are given by
\begin{equation}
m=
\frac{4\pi\rho}{3}, \qquad 
J=
\frac{2m}{5}\,\mathbb I. 
\label{mJ}
\end{equation}
The model with several physical parameters can be reduced to the problem above in which both the kinematic viscosity of the fluid
and the radius of the rigid ball are normalized.
Then $\rho$ is regarded as the ratio of both densities of the fluid-structure,
see subsection \ref{nondim}.

In \cite{EMT} Ervedoza, Maity and Tucsnak considered the case when 
the extra velocity $u_*$ 
is absent and proved that
the initial value problem for \eqref{eq-fixed} admits a unique solution globally in time with decay properties such as
\begin{equation*}
\|u(t)\|_{\infty,\Omega}+|\eta(t)|+|\omega(t)|=O(t^{-1/2})
\end{equation*}
as $t\to\infty$ under
the smallness of the $L^3$-norm of the initial velocity as well as the initial rigid motion.
This is a desired development of the local well-posedness in the space $L^q$ due to Wang and Xin \cite{WX}
and may be also regarded as the stability of the rest state. 
The stability of that state in the 2D case was already studied by Ervedoza, Hillairet and Lacave \cite{EHL14}
when the rigid body is a disk.
The present paper is aiming at the study of stability criterion for nontrivial basic motions
$\{u_b,p_b,\eta_b,\omega_b\}$ to \eqref{eq-fixed}, 
that is, we intend to deduce the large time behavior
\begin{equation}
\|u(t)-u_b(t)\|_{\infty,\Omega}+\|\nabla u(t)-\nabla u_b(t)\|_{3,\Omega}
+|\eta(t)-\eta_b(t)|+|\omega(t)-\omega_b(t)|=o(t^{-1/2})
\label{result-sta}
\end{equation}
of the solution to the initial value problem for \eqref{eq-fixed}
as $t\to \infty$ provided that the initial disturbance is
small enough in the same sense as in \cite{EMT} and that 
the basic motion is also small in a sense uniformly in $t$ as well as
globally H\"older continuous in $t$. 
We will discuss briefly in subsection \ref{basic-mo} possible basic motions within the self-propelled regime
\eqref{eq-fixed}.
Let us also mention that the asymptotic rate of $\nabla u(t)$ in \eqref{result-sta} is an improvement of the
corresponding result due to \cite{EMT}, in which less rate is deduced when $u_b=0,\, \eta_b=\omega_b=0$.
Let $\{u_b,\,p_b,\,\eta_b,\,\omega_b\}$ be a solution to \eqref{eq-fixed} on the whole time axis $I=\mathbb R$.
We intend to find a solution to the initial value problem for \eqref{eq-fixed} in the form
\[
u=u_b+\widetilde u, \quad
p=p_b+\widetilde p, \quad
\eta=\eta_b+\widetilde\eta, \quad
\omega=\omega_b+\widetilde\omega.
\]
Omitting the tildes $\;\widetilde{(\cdot)}$, the disturbance 
should obey
\begin{equation}
\begin{split}
&\left.
\begin{array}{rl}
\partial_tu+(u-\eta)\cdot\nabla u+(u_b-\eta_b)\cdot\nabla u+(u-\eta)\cdot\nabla u_b
=&\Delta u
-\nabla p  \\ 
\mbox{div $u$}=&0
\end{array}
\right\}
\mbox{in $\Omega\times (s,\infty)$}, \\
&\quad u|_{\partial\Omega}=\eta+\omega\times x, \qquad 
u\to 0\quad\mbox{as $|x|\to\infty$},  \\
&\quad m\frac{d\eta}{dt}+\int_{\partial\Omega}\mathbb S(u,p)\nu\,d\sigma=0, \\
&\quad J\frac{d\omega}{dt}+\int_{\partial\Omega}x\times \mathbb S(u,p)\nu\,d\sigma=0,
\end{split}
\label{eq-perturb}
\end{equation}
endowed with the initial conditions at the initial time $s\in\mathbb R$.

Besides \cite{EMT}, 
we will mention the existing literature,
that is particularly related to this study,
about the original
problem in the inertial frame in 3D. 
Serre \cite{Se87} constructed a global weak solution to the initial value problem under the action of gravity.
This solution is of the Leray-Hopf class and the argument is based on the
energy relation as in the standard Navier-Stokes theory.
Silvestre \cite{Sil02-a} also studied a weak solution 
under the self-propelling condition and moreover,
in a specific case that the body is axisymmetric, she constructed
a global strong solution with some parity conditions 
and discussed the attainability 
of the translational self-propelled steady motion found by
Galdi \cite{Ga97},
answering to celebrated Finn's starting problem in her context.
When the external force and torque act on
the rigid body whose shape is arbitrary,
a strong solution locally in time is 
constructed first by Galdi and Silvestre \cite{GaSil02}.
All of those literature \cite{Se87, Sil02-a, GaSil02} developed the $L^2$-theory
and the problem was studied in a frame attached to the
body of arbitrary shape (except the latter half of \cite{Sil02-a}), where, unlike the derivation of \eqref{eq-fixed}, 
the change of variable must involve the rotation matrix
as well as translation unless the body is a ball, see 
Galdi \cite{Ga02} for the details.
Then the resultant equation reads
\begin{equation}
\partial_t u+(u-\eta-\omega\times x)\cdot\nabla u
+\omega\times u=\Delta u-\nabla p
\label{shape-arbit}
\end{equation}
in place of the first equation of \eqref{eq-fixed}.
The equations of the rigid body in \eqref{eq-fixed} are also replaced by
\begin{equation}
\begin{split}
&m\frac{d\eta}{dt}+\omega\times (m\eta)
+\int_{\partial\Omega}\mathbb S(u,p)\nu\,d\sigma=0,  \\
&J\frac{d\omega}{dt}+\omega\times (J\omega)
+\int_{\partial\Omega}x\times \mathbb S(u,p)\nu\,d\sigma=0.
\end{split}
\label{shape-arbit-s}
\end{equation}
If we considered the exterior problem with a prescribed rigid motion,
the difficulty caused by the drift term $(\omega\times x)\cdot\nabla u$
with spatially unbounded (possibly even time-dependent) coefficient would be more or less overcome due to 
efforts by several authors, see \cite{GaNe, Hi13, Hi18-h, Hi18, Hi20} and the references therein.
But this term is indeed a nonlinear term for the fluid-structure interaction 
under consideration and still prevents us from carrying out
analysis in a successful way.
This is 
why we are forced to restrict ourselves to the case of the ball
in the present paper as well as in 
\cite{EMT}.
The linear theory of large time behavior 
is well established for the case of arbitrary shape in \cite{EMT},
while we have already a difficulty at the level of linear analysis
in this paper, see Remark \ref{rem-arbit} in subsection \ref{adjoint-backward}. 
Recently, Galdi \cite{Ga-new} has succeeded in
continuation of the solution obtained in \cite{GaSil02} globally in time by means of energy estimates
and, moreover, he has proved the large time decay
\begin{equation}
\lim_{t\to\infty}\Big(\|u(t)\|_{6,\Omega}+\|\nabla u(t)\|_{2,\Omega}+|\eta(t)|+|\omega(t)|\Big)=0
\label{Ga-large}
\end{equation}
even if the shape of the body is arbitrary when the $H^1$-norm of the initial velocity, the initial rigid motion,
and the $L^2$-norm in $t\in (0,\infty)$ of the external force and torque acting on the body are small enough.
Although there is no definite decay rate 
in general, one can derive the rate $O(t^{-1/2})$ 
in \eqref{Ga-large} if in particular the body
is a ball, see \cite[Remark 4.1]{Ga-new}.

There is the other way of transformation to reduce the original problem 
to the one in a time-independent domain.
This is a local transformation that keeps the situation far from the
body as it is in order to avoid the term $(\omega\times x)\cdot\nabla u$
in \eqref{shape-arbit} even if the shape of the body is arbitrary.
Cumsille and Takahashi \cite{CT08} adopted this transformation to 
construct a solution locally in time and then
successfully derived a priori estimates in the inertial frame so that the global existence of a strong 
solution for small data was first established 
within the $L^2$-theory, however, without any information about the large time behavior.
By the same transformation, Geissert, G\"otze and Hieber \cite{GGH13}
developed the $L^q$-theory for the local well-posedness
in the maximal regularity class;
further, in the recent work \cite{MT-new}, Maity and Tucsnak have proved even large time decay with definite rate
as well as global well-posedness for small data by adapting the time-shifted method proposed by Shibata \cite{Shi22}
in the framework of maximal regularity with time-weighted norms.
Since there are several complicated terms arising from 
this latter transformation,
it seems difficult to take a reasonable nontrivial basic motion in the resultant
system and to discuss its stability, that is indeed the issue here, unless one considers the problem around the rest state.
Thus the latter transformation is not preferred in the 
present paper. 

Let us turn to our problem \eqref{eq-perturb} around a nontrivial motion.
Towards the large time behavior of the disturbance, 
the essential step is to develop the temporal decay estimates of solutions to the linearized system.
In fact, this was successfully done by 
\cite{EMT} when $u_b=0$ and $\eta_b=0$.
They adopted 
the monolithic approach which is traced back to Takahashi and Tucsnak \cite{TT04}, see also Silvestre \cite{Sil02-a},
and then derived the $L^q$-$L^r$ decay estimates of the semigroup,
that they call the fluid-structure semigroup (or Takahashi-Tucsnak semigroup).
Look at \eqref{eq-perturb}, then the term which is never subordinate to the fluid-structure semigroup 
is $\eta_b\cdot\nabla u$.
This term is well known as the Oseen term in studies of the exterior 
Navier-Stokes problem, see \cite{Ga-b, KShi}.
For the fluid-structure interaction problem, however,
fine temporal behavior of the (purely) Oseen operator 
$-\Delta u-\eta_b\cdot\nabla u+\nabla p$
is hopeless because the term 
$\eta_b\cdot\nabla u$ 
is no longer skew-symmetric 
on account of the boundary condition.
The idea is to combine the Oseen term with $u_b\cdot\nabla u$ since 
the skew-symmetry is recovered for 
$(u_b-\eta_b)\cdot\nabla u$, 
see \eqref{B-skew}--\eqref{B-om}.
The first step is therefore to derive some decay properties
of solutions to the linearized system with the term $(U_b-\eta_b)\cdot\nabla u$ in the whole space
without the rigid body, see subsection \ref{evo-whole},
where $U_b$ is the monolithic velocity \eqref{background}.
Here, the linear term $U_b\cdot\nabla u$ 
can be treated 
as perturbation from
the purely Oseen evolution operator
provided that $u_b$ 
is of sub-critical class such as $u_b\in L^q(\Omega)$
with some $q<3$ even though $u_b$ is time-dependent. 
In the scale-critical case such as $u_b\in L^{3,\infty}(\Omega)$ (weak-$L^3$ space), 
decay estimate \eqref{decay-wh} with $r<\infty$ is still available, however, 
it is difficult to deduce gradient estimate \eqref{decay-grad-wh} 
of the solution when $u_b$ is time-dependent.
Although there is a device to construct the Navier-Stokes flow even
in this situation as proposed in \cite[Section 5]{Hi18},
it does not work for the other nonlinear term $\eta\cdot\nabla u$.
Thus the critical case is out of reach in this paper, and let us concentrate
ourselves on the sub-critical case, which covers the self-propelled motion
and the motion with wake structure (due to translation).
In the subcritical case, the linear term $u\cdot\nabla u_b$
can be discussed together with the nonlinear term as perturbation from
the principal part of the linearized system.
Moreover, 
the other term $\eta\cdot\nabla u_b$ is comparable to
$u\cdot\nabla u_b$ since $\eta(t)$ behaves like $\|u(t)\|_{\infty,\Omega}$ 
as $t\to\infty$.
As for an alternative way in which the term $(u-\eta)\cdot\nabla u_b$ is also involved into the linearization, see 
Remark \ref{rem-perturb} in subsection \ref{adjoint-backward}.
The right choice of the principal part of the linearized system would be the following 
non-autonomous system, which we call the Oseen-structure system:
\begin{equation}
\begin{split}
&\partial_tu=\Delta u+(\eta_b-u_b)\cdot\nabla u-\nabla p, \qquad 
\mbox{div $u$}=0 \quad \mbox{in $\Omega\times (s,\infty)$},  \\
&u|_{\partial\Omega}=\eta+\omega\times x, \qquad 
u\to 0 \quad\mbox{as $|x|\to\infty$},   \\
&m\frac{d\eta}{dt}+\int_{\partial\Omega}\mathbb S(u,p)\nu\,d\sigma=0, \\
&J\frac{d\omega}{dt}+\int_{\partial\Omega}x\times \mathbb S(u,p)\nu\,d\sigma=0,  \\
&u(\cdot,s)=u_0, \quad \eta(s)=\eta_0, \quad \omega(s)=\omega_0
\end{split}
\label{eq-linear}
\end{equation}
with $s\in\mathbb R$ being the initial time.
In fact, in this paper, we develop the $L^q$-$L^r$ estimates 
of solutions to 
\eqref{eq-linear} and apply them to the nonlinear problem \eqref{eq-perturb}.
We do need the specific shape already in linear analysis, 
otherwise new term $(\omega\times x)\cdot\nabla u_b$ 
appears in \eqref{eq-perturb}, 
which arises from 
$(\omega\times x)\cdot\nabla u$ in \eqref{shape-arbit}, see Remark \ref{rem-arbit} for further discussion,
although the other term $(\eta_b+\omega_b\times x-u_b)\cdot\nabla u$ 
can be handled as in \cite{Hi18, Hi20} by the present author.
Since the difficult term above is absent around the rest state
even if the shape of the rigid body is arbitrary, 
our method developed in this paper works well
to provide an alternative proof of the $L^q$-$L^r$ estimates of the fluid-structure
semigroup established by \cite{EMT}. 

The difficulty to analyze \eqref{eq-linear} is the non-autonomous character.
For the linearized problem relating to the Navier-Stokes system in exterior domains,
the only results that study the temporal decay for the non-autonomous system seem to be due to the present author
\cite{Hi18, Hi20} on the $L^q$-$L^r$ estimates of the Oseen evolution operator 
arising from (prescribed) time-dependent rigid motions.
In this paper we adapt the method developed in those papers to the Oseen-structure system \eqref{eq-linear}, where the crucial step is to deduce
the uniformly boundedness of the evolution operator for large time,
see \eqref{unif}--\eqref{unif-adj} in subsection \ref{decay-0th},
with the aid of \eqref{key-0}--\eqref{key-1} which follow from
the energy relations \eqref{energy}--\eqref{adj-energy}.
In this stage one needs to discuss the issue above 
simultaneously with the one of the adjoint evolution operator,
that is the solution operator to the backward problem for the adjoint system,
see subsection \ref{adjoint-backward}.
We then proceed to the next stage in which gradient estimates of the 
evolution operator are derived near the rigid body 
(Proposition \ref{led-prop}) and near spatial
infinity (Proposition \ref{prop-near-inf}).
In this latter stage, asymptotic behavior (for $t-s\to 0$ and $t-s\to\infty$) of the
associated pressure together with the time derivative of
the evolution operator plays an important role.
We first deduce
the bahavior of the pressure near the initial time 
in subsection \ref{sm-pressure}, where analysis is developed with the aid of fractional powers of the Stokes-structure operator 
and very different from the study of the same issue in \cite{Hi20} for the case
of prescribed time-dependent rigid motions.
In doing so, we make use the representation 
of the Stokes-structure resolvent,
which is given in subsection \ref{stokes-structure}.
To employ the monolithic approach, in this paper unlike \cite{EMT}, we do not rely on the
decomposition shown by Wang and Xin \cite{WX} 
because the associated projection
is not symmetric with respect to the duality pairing which involves
the constant density $\rho$ of the rigid body, see \eqref{mJ} and \eqref{pairing},
unless $\rho=1$.
This situation is not consistent with our argument particularly when making use of the adjoint evolution operator.
For this reason, our analysis is based on the similar but different decomposion
\eqref{decompo} below, then the associated projection possesses the aforementioned desired
symmetry. 
The decomposition \eqref{decompo} was established by Silvestre 
\cite{Sil02-a} when $q=2$.
Since the same result for the $L^q$-space does not seem to be found in the
existing literature, the proof is given in subsection \ref{new-decompo}.

The paper consists of five sections.
In the next section we 
introduce the Oseen-structure operator as well as the Stokes-structure one and then give our main results:
Theorem \ref{evo-op} on $L^q$-$L^r$ estimates of the evolution operator and Theorem \ref{nonlinear}
on the nonlinear stability.
Section \ref{oseen-structure-evo} is concerned with some preparatory
results: 
the decomposition of the $L^q$-space mentioned above, reformulation of the Stokes-structure operator,
smoothing estimate as well as generation of the Oseen-structure evolution
operator, analysis of the pressure and, finally, the adjoint
evolution operator.
In section 4 we study in detail the large time decay of the evolution operator to complete the proof of Theorem \ref{evo-op}.
Final section is devoted to the proof of Theorem \ref{nonlinear} for the nonlinear problem \eqref{eq-perturb}.

%% 2
\section{Main results}
\label{result}

%% 2.1
\subsection{Nondimensional variables}
\label{nondim}

If all physical parameters are taken into account, the system \eqref{eq-fixed}
should read
\begin{equation}
\begin{split}
&\rho_{_L}\left\{
\partial_tu+(u-\eta)\cdot\nabla u\right\}=\mu\Delta u-\nabla p,  \qquad 
\mbox{div $u$}=0 \quad\mbox{in $\Omega^R\times I$},  \\
&u|_{\partial\Omega^R}=\eta+\omega\times x+u_*, \qquad 
u\to 0 \quad \mbox{as $|x|\to\infty$},  \\
&m_{_R}\frac{d\eta}{dt}+\int_{\partial\Omega^R}\mathbb S_\mu(u,p)\nu\,d\sigma=0, \\
&J_{_R}\frac{d\omega}{dt}+\int_{\partial\Omega^R}x\times \mathbb S_\mu(u,p)\nu\,d\sigma=0,
\end{split}
\label{eq-para}
\end{equation}
with
\begin{equation*}
\begin{split}
&\Omega^R=\mathbb R^3\setminus \overline{B_R}, \qquad
\mathbb S_\mu(u,p)=2\mu Du-p\,\mathbb I, \\ 
&m_{_R}=\int_{B_R}\rho_{_S}\,dx
=\frac{4\pi R^3\rho_{_S}}{3}, \qquad
J_{_R}=\int_{B_R}\big(|x|^2\,\mathbb I-x\otimes x\big)\rho_{_S}\,dx
=\frac{8\pi R^5\rho_{_S}}{15}\,\mathbb I,
\end{split}
\end{equation*}
where the rigid body is a ball $B_R$ centered at the origin with radius $R$,
$\rho_{_L}$ and $\rho_{_S}$ are respectively the densities of the liquid and solid,
and $\mu$ stands for the viscosity coefficient.
All of those parameters are assumed to be positive constants.
Setting
\begin{equation*}
\begin{split}
&\widetilde x=\frac{1}{R}\,x, \qquad
\widetilde t=\frac{\mu}{R^2\rho_{_L}}\,t, \\ 
&\widetilde u(\widetilde x,\widetilde t)=\frac{R\rho_{_L}}{\mu}\,u(x,t), \qquad
\widetilde p(\widetilde x,\widetilde t)=\frac{R^2\rho_{_L}}{\mu^2}\,p(x,t), \\ 
&\widetilde\eta(\widetilde t)=\frac{R\rho_{_L}}{\mu}\,\eta(t), \qquad
\widetilde\omega(\widetilde t)=\frac{R^2\rho_{_L}}{\mu}\,\omega(t), \qquad
\widetilde u_*(\widetilde x,\widetilde t)=\frac{R\rho_{_L}}{\mu}\,u_*(x,t)
\end{split}
\end{equation*}
and omitting the tildes $\;\widetilde{(\cdot)}$, we are led to \eqref{eq-fixed} with
\begin{equation}
\rho=\frac{\rho_{_S}}{\rho_{_L}}, \quad m=\int_{B_1}\rho\,dx=\frac{4\pi\rho}{3}, \quad
J=\int_{B_1}\big(|x|^2\mathbb I-x\otimes x\big)\rho\,dx=\frac{2m}{5}\,\mathbb I,
\label{new-para}
\end{equation}
where $J$ is called the tensor of inertia.
It is worth while mentioning
\begin{equation}
x\times (a\times x)=\big(|x|^2\mathbb I-x\otimes x\big)a 
\label{inertia}
\end{equation}
for all $a\in\mathbb R^3$, that implies
\begin{equation}
\int_{B_1}(a\times x)\cdot (b\times x)\rho\,dx=(Ja)\cdot b=a\cdot (Jb) 
\label{angular}
\end{equation}
for all $a,\, b\in\mathbb R^3$.
Even if the shape of the rigid body is not a ball, the symmetric matrix $J$ 
defined by the latter integral of \eqref{new-para} satisfies 
\eqref{angular} and thus $J$ is positive definite.

%% 2.2
\subsection{Stokes-structure operator}
\label{reduced}

Let us begin with introducing fundamental notation.
Throughout this paper, we set
\[
B:=B_1, \qquad \Omega:=\mathbb R^3\setminus \overline{B}, \qquad
\Omega_R:=\Omega\cap B_R\quad (R>1),
\]
where $B_R$ denotes the open ball centered at the origin with
radius $R$.
Given a domain $D\subset\mathbb R^3$, $q\in [1,\infty]$ and integer $k\geq 0$,
the standard Lebesgue and Sobolev spaces are denoted by
$L^q(D)$ and $W^{k,q}(D)$.
We abbreviate the norm $\|\cdot\|_{q,D}=\|\cdot\|_{L^q(D)}$.
By $\langle\cdot,\cdot\rangle_D$ 
(resp. $\langle \cdot,\cdot\rangle_{\partial D}$)
we denote standard duality pairings
over the domain $D$ (resp. $\partial D$ being the boundary of $D$) in each context.
As the space for the pressure, one needs also the homogeneous
Sobolev space
\[
\widehat W^{1,q}(D)=\{p\in L^q_{\rm loc}(\overline{D});\;\nabla p\in L^q(D)\}
\]
with seminorm $\|\nabla (\cdot)\|_{q,D}$ for $1<q<\infty$.
When $D=\Omega$, 
let us introduce
\[
\widehat W^{1,q}_{(0)}(\Omega)=\left\{p\in \widehat W^{1,q}(\Omega);\; \int_{\Omega_3}p\,dx=0\right\},
\]
that is a Banach space with norm $\|\nabla (\cdot)\|_{q,\Omega}$.
The class $C_0^\infty(D)$ consists of all $C^\infty$-functions with
compact support in $D$, then $W^{k,q}_0(D)$ denotes the completion
of $C_0^\infty(D)$ in $W^{k,q}(D)$, where $k>0$ is an integer.
Given $q\in [1,\infty]$, let $q^\prime\in [1,\infty]$ be the conjugate number
defined by $1/q^\prime+1/q=1$.
For $q\in (1,\infty)$, we define the Sobolev space of order $(-1)$ by
$W^{-1,q}(D)=W^{1,q^\prime}_0(D)^*$.
In what follows we adopt the same symbols for denoting scalar and vector (even tensor)
functions as long as no confusion occurs.
By $C^\infty_{0,\sigma}(D)$ we denote the class of all 
vector fields $u$
which are in $C_0^\infty(D)$ and satisfy $\mbox{div $u$}=0$ in $D$. 
Let $X$ be a Banach space.
Then ${\mathcal L}(X)$ stands for the Banach space consisting of
all bounded linear operators from $X$ into itself.
Finally, we denote several positive constants by $C$,
which may change from line to line.

Before stating our main results, let us also introduce
the underlying space and the Stokes-structure operator to formulate \eqref{eq-perturb} and \eqref{eq-linear}
within the monolithic framework as in \cite{EMT, TT04, Sil02-a}.
By ${\rm RM}$
we denote the space of all rigid motions, that is,
\[
{\rm RM}:=\{\eta+\omega\times x;\; \eta,\,\omega\in\mathbb R^3\}.
\]
For the resolvent problem, see subsection \ref{stokes-structure}, we have to
consider the complex rigid motions $\eta+\omega\times x$ with $\eta,\, \omega\in \mathbb C^3$. 
For $1<q<\infty$, we set
\begin{equation}
L^q_R(\mathbb R^3):=\{U\in L^q(\mathbb R^3)^3;\; U|_B\in {\rm RM}\}
\label{leb-rigid}
\end{equation}
that is closed in $L^q(\mathbb R^3)^3$.
The underlying space we adopt is
\begin{equation}
X_q(\mathbb R^3):=\{U\in L^q_R(\mathbb R^3);\; \mbox{div $U=0$\; in $\mathbb R^3$}\}. 
\label{sp-velo}
\end{equation}
It is convenient to define the map 
\[
i: L^q_R(\mathbb R^3)\to L^q(\Omega)^3\times \mathbb R^3\times \mathbb R^3
\]
or
\[
i: L^q_R(\mathbb R^3)\to L^q(\Omega)^3\times\mathbb C^3\times \mathbb C^3
\]
for the associated resolvent problem, where the scalar field of the Lebesgue space is $\mathbb C$ for the latter case, by
\begin{equation}
\begin{split}
&i: U\mapsto i(U):=(u,\eta,\omega) \;\;\mbox{with} \\
&u=U|_\Omega, \qquad \eta=\frac{1}{m}\int_B U(x)\rho\,dx, \qquad \omega=J^{-1}\int_B x\times U(x)\rho\,dx
\end{split}
\label{X1}
\end{equation}
with $\rho,\, m$ and $J$ being given by \eqref{new-para}, then
we see from \eqref{inertia} that
\[
U|_B=\eta+\omega\times x.
\]
If in particular $U\in X_q(\mathbb R^3)$, we find
\begin{equation}
\nu\cdot(u-\eta-\omega\times x)|_{\partial\Omega}=0
\label{hidden-bc1}
\end{equation}
with $\nu$ being the unit normal to $\partial\Omega$ directed toward $B$ 
(indeed, $\nu=-x$ since $\partial\Omega$ is the unit sphere) on account of
$\mbox{div $U$}=0$ in $\mathbb R^3$. 
Conversely,
if $u\in L^q(\Omega)$ satisfies $\mbox{div $u$}=0$ in $\Omega$ (so that the normal trace $\nu\cdot u|_{\partial\Omega}$ makes sense)
and
$\nu\cdot (u-\eta-\omega\times x)|_{\partial\Omega}=0$
for some pair of $\eta,\, \omega\in\mathbb R^3$,
then
\begin{equation}
U:=u\chi_\Omega+(\eta+\omega\times x)\chi_B\in X_q(\mathbb R^3)
\label{X2}
\end{equation}
with $i(U)=(u,\eta,\omega)$,
where $\chi_\Omega$ and $\chi_B$ denote the characteristic functions.
In this way, elements of $X_q(\mathbb R^3)$ are understood through
\eqref{X1}--\eqref{X2}.

The space $X_q(\mathbb R^3)$, $1<q<\infty$, is a Banach space endowed with norm
\begin{equation}
\|U\|_{X_q(\mathbb R^3)}
:=\Big(\|u\|_{q,\Omega}^q+\|(\eta_u+\omega_u\times x)\rho^{1/q}\|_{q,B}^q\Big)^{1/q}
\label{X-norm0}
\end{equation}
for $U\in X_q(\mathbb R^3)$ and we have the duality relation $X_q(\mathbb R^3)^*=X_{q^\prime}(\mathbb R^3)$,
see \eqref{dual-an} in Proposition \ref{prop-decom} below,
with the pairing
\begin{equation}
\begin{split}
\langle U,V\rangle_{\mathbb R^3,\rho}
&:=\int_\Omega u\cdot v\,dx
+\int_B (\eta_u+\omega_u\times x)\cdot (\eta_v+\omega_v\times x)\rho\,dx  \\
&=\langle u,v\rangle_\Omega+m\eta_u\cdot\eta_v+(J\omega_u)\cdot \omega_v
\end{split}
\label{pairing}
\end{equation}
(with obvious change if one should consider the complex rigid motions for the resolvent problem)
for $U\in X_q(\mathbb R^3)$ and $V\in X_{q^\prime}(\mathbb R^3)$,
where $i(U)=(u,\eta_u,\omega_u),\, i(V)=(v,\eta_v,\omega_v)$,
see \eqref{new-para}, \eqref{angular} and \eqref{X1}.
It is clear that the pairing \eqref{pairing} is defined for $U\in L^q_R(\mathbb R^3)$ and $V\in L^{q^\prime}_R(\mathbb R^3)$ as well.
Notice that the constant weight $\rho$ is involved in the integral over the rigid body $B$,
see \eqref{X-norm0}--\eqref{pairing}, nonetheless,
it is obvious from \eqref{X1} to see that the following three quantities are equivalent for $U\in X_q(\mathbb R^3),\; i(U)=(u,\eta_u,\omega_u)$:
\begin{equation}
\|U\|_{X_q(\mathbb R^3)}\sim
\|U\|_{q,\mathbb R^3}\sim \|u\|_{q,\Omega}+|\eta_u|+|\omega_u|,
\label{X-norm}
\end{equation}
where the symbol $\sim$ means that inequalities in both directions hold with some constants.
Thus the norm \eqref{X-norm0} does not play any role to discuss the asymptotic behavior of solutions to
\eqref{eq-perturb} and \eqref{eq-linear}.
Nevertheless, the reason for introducing \eqref{X-norm0} is to describe the energy
\begin{equation}
\|U\|_{X_2(\mathbb R^3)}^2=\langle U,U\rangle_{\mathbb R^3,\rho}
=\|u\|_{2,\Omega}^2+m|\eta_u|^2+(J\omega_u)\cdot \omega_u
\label{energy-rep}
\end{equation}
with 
$(J\omega_u)\cdot\omega_u=\frac{2m}{5}|\omega_u|^2$ when $B$ is a ball.
In fact, the energy \eqref{energy-rep} fulfills
the identity in the desired form along time-evolution of the linearized system \eqref{eq-linear},
see \eqref{energy}--\eqref{adj-energy} in subsection \ref{energy-relation}.
The other reason is to ensure several duality relations, see
\eqref{dual}, \eqref{stokes-sym} and \eqref{dual-evo-sense}.
Even for general $U,\, V\in L^q(\mathbb R^3)$, $1<q<\infty$,
it is sometimes (see subsectrion \ref{proof-1}) convenient to introduce the other norm and pairing
\begin{equation}
\|U\|_{q,(\mathbb R^3,\rho)}
:=\left(\int_{\mathbb R^3}|U(x)|^q\,(\rho\chi_B+\chi_\Omega)\,dx\right)^{1/q}, \quad
\langle U,V\rangle_{\mathbb R^3,\rho}
:=\int_{\mathbb R^3}(U\cdot V) (\rho\chi_B+\chi_\Omega)\,dx,
\label{other-pair}
\end{equation}
which are consistent with \eqref{X-norm0}--\eqref{pairing}, 
in order that  
$L^q_R(\mathbb R^3)$ is regarded as a subspace of $L^q(\mathbb R^3)^3$.
By \cite[Proposition 3.3]{EMT} it is shown that the class
\begin{equation}
\mathcal{E}(\mathbb R^3):=\{U\in C_{0,\sigma}^\infty(\mathbb R^3);\; DU=O\;\mbox{in $B$}\}
=C_0^\infty(\mathbb R^3)^3\cap X_q(\mathbb R^3)
\label{dense-sub}
\end{equation}
is dense in $X_q(\mathbb R^3)$, $1<q<\infty$.
An alternative proof will be given
in Proposition \ref{prop-decom} as well. 

Let $U\in X_q(\mathbb R^3)$ satisfy, in particular, $u\in W^{1,q}(\Omega)$, where $i(U)=(u,\eta,\omega)$, see \eqref{X1}.
Then, $U\in W^{1,q}(\mathbb R^3)$ if and only if
\begin{equation}
u|_{\partial\Omega}=\eta+\omega\times x.
\label{hidden-bc2}
\end{equation}
Moreover, under the condition \eqref{hidden-bc2} it holds that
\begin{equation}
\nabla U=(\nabla u)\chi_\Omega+M_\omega\chi_B, \qquad
M_\omega:=\left(
\begin{array}{ccc}
0 & -\omega_3 & \omega_2 \\
\omega_3 & 0 & -\omega_1 \\
-\omega_2 & \omega_1 & 0
\end{array}
\right).
\label{grad-form}
\end{equation}

Let $1<q<\infty$.
With the space $X_q(\mathbb R^3)$ at hand as the one in which we are going to look for the monolithic velocity \eqref{X2},
there are two ways to eliminate the pressure.
They are similar but slightly different.
One is the approach within the space
$L^q_R(\mathbb R^3)$ 
by Silvestre \cite{Sil02-a, Sil02-b}, who developed the case $q=2$.
In Proposition \ref{prop-decom} below we will establish the decomposition
\begin{equation}
L^q_R(\mathbb R^3)=X_q(\mathbb R^3)\oplus Z_q(\mathbb R^3)  
\label{decompo}
\end{equation}
with
\begin{equation}
\begin{split}
&Z_q(\mathbb R^3)=\Big\{V\in L^q_R(\mathbb R^3);\; V|_\Omega=\nabla p,\; 
p\in \widehat W^{1,q}(\Omega), \;
V|_B=\eta+\omega\times x,  \\
& \qquad\qquad\qquad\qquad\qquad
\eta=\frac{-1}{m}\int_{\partial\Omega}p\nu\,d\sigma,\;
\omega=-J^{-1}\int_{\partial\Omega}x\times (p\nu)\,d\sigma\Big\} \\
\end{split}
\label{sp-pressure}
\end{equation}
as well as
$X_q(\mathbb R^3)^\perp=Z_{q^\prime}(\mathbb R^3)$.
The latter relation means that 
$Z_{q^\prime}(\mathbb R^3)$ is the annihilator of $X_q(\mathbb R^3)$
with respect to $\langle\cdot,\cdot\rangle_{\mathbb R^3,\rho}$, see \eqref{pairing}. 
Note that, for the element of the space $Z_q(\mathbb R^3)$,
$p$ is determined uniquely up to constant which, however, does not change $\eta$ and $\omega$ since
$\int_{\partial\Omega}\nu\,d\sigma=\int_{\partial\Omega}x\times \nu\,d\sigma=0$. 
When $q=2$, 
\eqref{decompo} was already proved by Silvestre \cite{Sil02-a}, 
who studied her problem within, instead of $L^q_R(\mathbb R^3)$, the class $L^2(\Omega)+{\rm RM}$
in which the flow behaves like a rigid motion at infinity
in the exterior domain $\Omega$.
By
\begin{equation}
\mathbb P=\mathbb P_q: L^q_R(\mathbb R^3)\to X_q(\mathbb R^3)
\label{new-proj}
\end{equation}
we denote the bounded projection associated with \eqref{decompo}.
Then we have the relation $\mathbb P_q^*=\mathbb P_{q^\prime}$ 
with respect to $\langle\cdot,\cdot\rangle_{\mathbb R^3,\rho}$ in 
which the constant weight $\rho$ is involved, that is, in
the sense of \eqref{proj-sym-0}, see Proposition \ref{prop-decom}.

The other way is to use the following decomposition developed by Wang and Xin \cite[Theorem 2.2]{WX}, see also Dashti and Robinson \cite{DR} for the
case $q=2$:
\begin{equation} 
L^q(\mathbb R^3)=G_q^{(1)}(\mathbb R^3) 
\oplus L^q_\sigma(\mathbb R^3)  
=G_q^{(1)}(\mathbb R^3)
\oplus G_{q}^{(2)}(\mathbb R^3)\oplus X_q(\mathbb R^3)
\label{helm}
\end{equation}
where
\begin{equation*}
\begin{split} 
&L^q_\sigma(\mathbb R^3)=\{V\in L^q(\mathbb R^3);\;
\mbox{div $V$}=0\;\mbox{in $\mathbb R^3$}\}, \\ 
&G_q^{(1)}(\mathbb R^3)=\{\nabla p_1\in L^q(\mathbb R^3);\; p_1\in\widehat W^{1,q}(\mathbb R^3)\}, \\
&G_{q}^{(2)}(\mathbb R^3)=\Big\{ V\in L^q(\mathbb R^3);\;
\mbox{div $V$}=0\;\mbox{in $\mathbb R^3$},\; 
V|_\Omega=\nabla p_2,\;
p_2\in \widehat W^{1,q}(\Omega), \\  
& \qquad\qquad\qquad
\int_B V\,dx=-\int_{\partial\Omega}p_2\nu\,d\sigma,\;
\int_B x\times V\,dx=-\int_{\partial\Omega}x\times (p_2\nu)\,d\sigma\Big\}.
\end{split}
\end{equation*}
Here, 
the density $\rho$ is not involved in the conditions for 
$V\in G_{q}^{(2)}(\mathbb R^3)$, so that $\langle U,V\rangle_{\mathbb R^3}=0$ for all
$U\in X_q(\mathbb R^3)$ and $V\in G_{q^\prime}^{(2)}(\mathbb R^3)$.
One cannot avoid this situation because $L^q_\sigma(\mathbb R^3)^\perp=G_{q^\prime}^{(1)}(\mathbb R^3)$ holds with respect to
the standard pairing $\langle\cdot,\cdot\rangle_{\mathbb R^3}$.
As a consequence, the bounded projection
\[
\mathbb Q=\mathbb Q_q: L^q(\mathbb R^3)\to X_q(\mathbb R^3)
\]
associated with \eqref{helm} fulfills the relation 
$\mathbb Q_q^*=\mathbb Q_{q^\prime}$ in the sense that
\begin{equation*}
\langle \mathbb Q_qU,V\rangle_{\mathbb R^3}=\langle U,\mathbb Q_{q^\prime}V\rangle_{\mathbb R^3}
\end{equation*}
for all  $U\in L^q(\mathbb R^3)$ and $V\in L^{q^\prime}(\mathbb R^3)$
with respect to the standard pairing 
$\langle\cdot,\cdot\rangle_{\mathbb R^3}$,
which should be compared with \eqref{proj-sym-0} in Proposition \ref{prop-decom}.
The fact that $\mathbb Q$ is not symmetric with respect to $\langle\cdot,\cdot\rangle_{\mathbb R^3,\rho}$ unless $\rho=1$
is not consistent with 
duality arguments especially in subsections \ref{sm-pressure}, \ref{adjoint-backward}, \ref{decay-0th} and \ref{proof-1}.
For this reason, the latter way is not convenient for us
and thus the decomposition \eqref{decompo} is preferred in this paper.

The classical Fujita-Kato projection
\begin{equation}
\mathbb P_0=\mathbb P_{0,q}: L^q(\mathbb R^3)\to L^q_\sigma(\mathbb R^3)
\label{FK-proj}
\end{equation}
is well-known and it is described in terms of the Riesz transform
${\mathcal R}=(-\Delta)^{-1/2}\nabla$ as
$\mathbb P_0={\mathcal I}+{\mathcal R}\otimes {\mathcal R}$ with ${\mathcal I}$
being the identity operator.
Notice the relation
$\mathbb P_{0,q}^*=\mathbb P_{0,q^\prime}$
with respect to $\langle\cdot,\cdot\rangle_{\mathbb R^3}$.
The projection $\mathbb P_0$ is used in subsections
\ref{evo-whole} and \ref{near-infinity}.

Let us call \eqref{eq-linear} with $\{u_b,\eta_b\}=\{0,0\}$
the Stokes-structure system (although it is called the fluid-structure system in the existing literature); that is, it is written as
\begin{equation}
\begin{split}
&\partial_tu=\Delta u-\nabla p, \qquad 
\mbox{div $u$}=0 \quad \mbox{in $\Omega\times (0,\infty)$},  \\
&u|_{\partial\Omega}=\eta+\omega\times x, \qquad 
u\to 0 \quad \mbox{as $|x|\to\infty$}, \\
&m\frac{d\eta}{dt}+\int_{\partial\Omega}\mathbb S(u,p)\nu\,d\sigma=0, \\
&J\frac{d\omega}{dt}+\int_{\partial\Omega}x\times \mathbb S(u,p)\nu\,d\sigma=0,
\end{split}
\label{stokes-str}
\end{equation}
subject to the initial conditions at $s=0$ (since \eqref{stokes-str} is autonomous).
This problem can be formulated as the evolution equation of the form
\begin{equation}
\frac{dU}{dt}+AU=0
\label{abst-evo-eq}
\end{equation}
in the space $X_q(\mathbb R^3)$, where the velocities of the fluid and the rigid body are unified as a velocity $U$ in the whole space 
$\mathbb R^3$ through \eqref{X2}.
Here, the operator $A$, to which we refer as the Stokes-structure operator
in this paper, is defined by
\begin{equation}
\begin{split}
D_q(A)
&=\big\{U\in X_q(\mathbb R^3)\cap W^{1,q}(\mathbb R^3);\;
u=U|_\Omega\in W^{2,q}(\Omega)\big\},  \\
AU
&=\mathbb P{\mathcal A}U, \\
{\mathcal A}U
&=
\left\{
\begin{aligned}
&-\mbox{div $(2Du)$}=-\Delta u, \qquad x\in\Omega,  \\ 
&\frac{1}{m}\int_{\partial\Omega}(2Du)\nu\,d\sigma+
\left(J^{-1}\int_{\partial\Omega}y\times (2Du)\nu\,d\sigma_y\right)\times x,  \qquad x\in B, 
\end{aligned}
\right.  \\
\end{split}
\label{stokes-op}
\end{equation}
where $\mathbb P$ is the projection \eqref{new-proj}.
Notice that the domain $D_q(A)$ is dense since so is ${\mathcal E}(\mathbb R^3)$ and that
the boundary condition \eqref{hidden-bc2} 
is hidden for $U\in D_q(A)$,
where $(u,\eta,\omega)=i(U)$, see \eqref{X1}.

In \cite{EMT} the other operator $\widetilde AU=\mathbb Q{\mathcal A}U$ 
with the same domain $D_q(\widetilde A)=D_q(A)$ acting
on the same space $X_q(\mathbb R^3)$ is 
defined by use of the other projection $\mathbb Q$ associated with \eqref{helm}. Due to \cite[Proposition 3.4]{EMT}, given $F\in X_q(\mathbb R^3)$, 
the resolvent problem 
\begin{equation}
(\lambda +\widetilde A)U=F \qquad \mbox{in $X_q(\mathbb R^3)$}
\label{resol11}
\end{equation}
is equivalent to the Stokes-structure resolvent system
\begin{equation}
\begin{split} 
&\lambda u-\Delta u+\nabla p=f, \qquad 
\mbox{div $u$}=0 \quad\mbox{in $\Omega$}, \\ 
&u|_{\partial\Omega}=\eta+\omega\times x,  \qquad 
u\to 0\quad \mbox{as $|x|\to\infty$},  \\ 
&\lambda \eta+\frac{1}{m}\int_{\partial\Omega}\mathbb S(u,p)\nu\,d\sigma=\kappa, \\ 
&\lambda \omega+J^{-1}\int_{\partial\Omega}x\times \mathbb S(u,p)\nu\,d\sigma=\mu,  
\end{split} 
\label{st-str-resol}
\end{equation}
where $(f,\kappa,\mu)=i(F)$, see \eqref{X1}, 
and the associated pressure $p$ is appropriately determined.
Likewise, as we will show in Proposition \ref{prop-equi},
our resolvent problem 
\begin{equation}
(\lambda+A)U=F \qquad\mbox{in $X_q(\mathbb R^3)$}
\label{resol1}
\end{equation}
is also equivalent to \eqref{st-str-resol}.
Moreover, via uniqueness of solutions to the problem
\eqref{st-str-resol} with $\lambda =1$,
we find that $A=\widetilde A$ in Proposition \ref{two-st-str}
and, thereby, that several results for $\widetilde A$
established by \cite{EMT} continue to hold for $A$ as well.
We note that, as observed first by Takahashi and Tucsnak \cite{TT04},
the evolution equation \eqref{abst-evo-eq} is also shown to be
equivalent to the system \eqref{stokes-str} in the similar
fashion to Proposition \ref{prop-equi}.

If we denote by $A_q$ the operator $A$ with $D_q(A)$ acting on 
$X_q(\mathbb R^3)$, we have the duality relation
\begin{equation}
A_q^*=A_{q^\prime}
\label{op-sym}
\end{equation}
with respect to $\langle\cdot,\cdot\rangle_{\mathbb R^3,\rho}$ 
for every $q\in (1,\infty)$ and thus $A$ is closed.
The duality \eqref{op-sym} is observed from \eqref{stokes-sym} below combined with the fact 
that $\lambda +A$ with 
$\lambda >0$ is surjective.
This surjectivity follows from analysis of the resolvent mentioned just below.
In particular, it is a positive self-adjoint 
operator on $X_2(\mathbb R^3)$ as shown essentially by Silvestre \cite[Theorem 4.1]{Sil02-b} and by Takahashi and Tucsnak
\cite[Proposition 4.2]{TT04}. 
Given $F\in X_q(\mathbb R^3)$, 
consider the resolvent problem \eqref{resol1}.
Then, for every $\varepsilon\in (0,\pi/2)$, there is a constant $C_\varepsilon>0$ such that
$\mathbb C\setminus (-\infty,0]\subset \rho(-A)$
being the resolvent set of the operator $-A$
subject to
\begin{equation}
\|(\lambda +A)^{-1}F\|_{q,\mathbb R^3}
\leq\frac{C_\varepsilon}{|\lambda|}\|F\|_{q,\mathbb R^3}
\label{para-resol}
\end{equation}
for all $\lambda\in\Sigma_\varepsilon$ 
and $F\in X_q(\mathbb R^3)$, 
which is due to Ervedoza, Maity and Tucsnak \cite[Theorem 6.1]{EMT} and 
implies that the operator $-A$ generates a bounded analytic semigroup
(which is called the fluid-structure semigroup in the existing literature)
$\{e^{-tA};\, t\geq 0\}$ on $X_q(\mathbb R^3)$ for every $q\in (1,\infty)$,
that is an improvement of Wang and Xin \cite{WX},
where
\begin{equation}
\Sigma_\varepsilon:=
\{\lambda\in\mathbb C\setminus\{0\};\; |\arg \lambda|\leq \pi-\varepsilon\}
\label{para-sector}.
\end{equation}
Hence, the fractional powers $A^\alpha$ with $\alpha>0$ are well-defined as closed operators on $X_q(\mathbb R^3)$.
Estimate \eqref{para-resol} especially for large $|\lambda|$ will be revisited in subsection \ref{stokes-structure}.

\subsection{Main results}
\label{main}

As the basic motion, we fix a solution $\{u_b,\eta_b,\omega_b\}$ to
the problem \eqref{eq-fixed} on the whole time axis $\mathbb R$
together with the associated pressure $p_b$ and set
\begin{equation}
U_b(x,t)=u_b(x,t)\chi_\Omega(x)+\big(\eta_b(t)+\omega_b(t)\times x\big)\chi_B(x). 
\label{background}
\end{equation}
Let us assume that
\begin{equation}
\left\{
\begin{array}{l}
u_b\in L^\infty(\mathbb R;\, L^{q_0}(\Omega)\cap L^\infty(\Omega))\; 
\mbox{with some}\; q_0\in (1,3), \quad 
\mbox{div $u_b$}=0\;\mbox{in $\Omega$},  \\
\eta_b,\,\omega_b\in L^\infty(\mathbb R;\, \mathbb R^3), \\
\nu\cdot (u_b-\eta_b-\omega_b\times x)|_{\partial\Omega}=0, \quad 
\end{array}
\right.
\label{ass-1}
\end{equation}
then we see that $U_b\in L^\infty(\mathbb R;\, X_q(\mathbb R^3))$ for every $q\in [q_0,\infty)$. 
Let us also mention that
the assumption on $\omega_b$ is used merely in the proof of Proposition \ref{est-wh} 
for the whole space problem without body.
We further make the assumption
\begin{equation}
u_b\in C^\theta(\mathbb R;\,L^\infty(\Omega)), \quad
\eta_b\in C^\theta(\mathbb R;\, \mathbb R^3)\;\;
\mbox{with some}\; \theta\in (0,1).
\label{ass-2}
\end{equation}
Set
\begin{equation}
\begin{split}
&\|U_b\|:=\sup_{t\in\mathbb R}\big(\|u_b(t)\|_{q_0,\Omega}+\|u_b(t)\|_{\infty,\Omega}+|\eta_b(t)|+|\omega_b(t)|\big), \\
&[U_b]_\theta:=\sup_{t>s}\frac{\|u_b(t)-u_b(s)\|_{\infty,\Omega}+|\eta_b(t)-\eta_b(s)|}{(t-s)^\theta}.  
\end{split}
\label{quan}
\end{equation}
\begin{remark}
The condition at the boundary $\partial\Omega$ in \eqref{ass-1} may be rewritten as $\nu\cdot (u_b-\eta_b)|_{\partial\Omega}=0$
since $\nu\cdot (\omega_b\times x)=0$ 
at the sphere always holds.
In view of the boundary condition in \eqref{eq-fixed}, one can deal with the case of tangential velocity $u_*$, that is,
$\nu\cdot u_*|_{\partial\Omega}=0$.
\label{rem-ass}
\end{remark}

Examples of the basic motion for which the assumptions \eqref{ass-1}--\eqref{ass-2} could be met will be discussed briefly
in subsection \ref{basic-mo}.

To study the Oseen-structure system \eqref{eq-linear}, let us introduce the family 
$\{L_\pm(t);\, t\in\mathbb R\}$ of the Oseen-structure operators 
on $X_q(\mathbb R^3)$, $1<q<\infty$, by
\begin{equation}
D_q(L_\pm(t))=D_q(A), \qquad
L_\pm(t)U=AU\pm B(t)U,
\label{oseen-op}
\end{equation}
where $A$ is the Stokes-structure operator \eqref{stokes-op} and
\begin{equation}
B(t)U
=\mathbb P\big[\{(u_b(t)-\eta_b(t))\cdot\nabla u\}\chi_\Omega\big]
\label{oseen-term}
\end{equation}
for $u=U|_\Omega$ with $\mathbb P$ being the projection
given by \eqref{new-proj}.
As mentioned in the preceding subsection and as we will show in 
Proposition \ref{two-st-str}, our Stokes-structure operator $A$ 
coincides with the operator $\widetilde A$ studied
in \cite{EMT} by Ervedoza, Maity and Tucsnak.
With the aid of the elliptic estimate due to \cite[Proposition 7.3]{EMT}
\begin{equation}
\|u\|_{W^{2,q}(\Omega)}\leq C\big(\|AU\|_{q,\Omega}+\|U\|_{q,\mathbb R^3}\big), \qquad U\in D_q(A),\; u=U|_\Omega,
\label{ADN-0}
\end{equation}
we find
\begin{equation}
\begin{split}
\|\nabla u\|_{q,\Omega}
&\leq C\|u\|_{W^{2,q}(\Omega)}^{1/2}\|u\|_{q,\Omega}^{1/2}  \\
&\leq C\big(\|AU\|_{q,\Omega}+\|U\|_{q,\mathbb R^3}\big)^{1/2}\|u\|_{q,\Omega}^{1/2}  \\
&\leq C
\left(\|AU\|_{q,\mathbb R^3}^{1/2}\|U\|_{q,\mathbb R^3}^{1/2}+\|U\|_{q,\mathbb R^3}\right)
\end{split}
\label{grad}
\end{equation}
which together with \eqref{ass-1}--\eqref{quan} leads to
\begin{equation}
\|B(t)U\|_{q,\mathbb R^3}\leq C\|(u_b(t)-\eta_b(t))\cdot\nabla u\|_{q,\Omega}
\leq C\|U_b\|
\left(\|AU\|_{q,\mathbb R^3}^{1/2}\|U\|_{q,\mathbb R^3}^{1/2}+\|U\|_{q,\mathbb R^3}\right)
\label{perturb-est}
\end{equation}
for all $U\in D_q(A)$.
Thus, one justifies the relation $D_q(L_\pm(t))=D_q(A)$, see \eqref{oseen-op},
on which we have
\begin{equation}
\begin{split}
\|L_\pm(t)U\|_{q,\mathbb R^3}
&\leq (1+C\|U_b\|)\|AU\|_{q,\mathbb R^3}+C\|U_b\|\|U\|_{q,\mathbb R^3},  \\
\|AU\|_{q,\mathbb R^3}
&\leq 2\|L_\pm(t)U\|_{q,\mathbb R^3}+C(\|U_b\|^2+\|U_b\|)\|U\|_{q,\mathbb R^3}.
\end{split}
\label{equi-LA}
\end{equation}
Moreover, 
it follows immediately from \eqref{stokes-op} that
\begin{equation*}
\|AU\|_{q,\mathbb R^3}
\leq C\|u\|_{W^{2,q}(\Omega)}, \qquad  
U\in D_q(A),\, u=U|_\Omega,
\end{equation*}
which combined with \eqref{ADN-0} tells us that 
$\|\cdot\|_{W^{2,q}(\Omega)}+\|\cdot\|_{q,\mathbb R^3}$ is equivalent
to the graph norm of $A$ and, therefore, also to that of $L_\pm(t)$
uniformly in $t\in\mathbb R$ on account of \eqref{equi-LA}; in particular,
\begin{equation}
\|u\|_{W^{2,q}(\Omega)}
\leq C\big(\|L_\pm(t)U\|_{q,\mathbb R^3}+\|U\|_{q,\mathbb R^3}\big), \qquad
u=U|_\Omega,
\label{ADN-1}
\end{equation}
for all $U\in D(L_\pm(t))$ with some $C>0$ (involving $\|U_b\|$) independent
of $t\in\mathbb R$.
 
The initial value problems for \eqref{eq-perturb} as well as \eqref{eq-linear} are formulated as
\begin{equation}
\frac{dU}{dt}+L_+(t)U=0,\quad t\in (s,\infty); \qquad U(s)=U_0
\label{linearized}
\end{equation}
and
\begin{equation}
\frac{dU}{dt}+L_+(t)U=H(U), \quad t\in (s,\infty); \qquad U(s)=U_0
\label{perturbed}
\end{equation}
where
\begin{equation}
U_0=u_0\chi_\Omega+(\eta_0+\omega_0\times x)\chi_B
\label{IC-mono}
\end{equation}
and
\begin{equation}
H(U)=
\mathbb P\left[
\big\{(\eta-u)\cdot\nabla(u_b+u)\big\}\chi_\Omega\right]
\label{rhs}
\end{equation}
with $(u,\eta,\omega)=i(U)$, see\eqref{X1}.
Recall that the assumption $U_0\in X_q(\mathbb R^3)$ involves
the compatibility condition
$\nu\cdot (u_0-\eta_0-\omega_0\times x)|_{\partial\Omega}=0$.
The first main result of this paper is the following theorem on $L^q$-$L^r$ estimates of the evolution operator
generated by $L_+(t)$.
\begin{theorem}
Suppose \eqref{ass-1} and \eqref{ass-2}, then
the operator family $\{L_+(t);\, t\in\mathbb R\}$ generates an evolution operator $\{T(t,s);\, -\infty<s\leq t<\infty\}$
on $X_q(\mathbb R^3)$ for every $q\in (1,\infty)$ with the following properties:
\begin{equation}
T(t,\tau)T(\tau,s)=T(t,s)\quad (s\leq \tau\leq t), \quad
T(t,t)={\mathcal I} \quad
\mbox{in ${\mathcal L}(X_q(\mathbb R^3))$},
\label{semi}
\end{equation}
\begin{equation}
(t,s)\mapsto T(t,s)F\in X_q(\mathbb R^3)\;
\mbox{is continuous for $F\in X_q(\mathbb R^3)$},
\label{evo-conti}
\end{equation}
\begin{equation}
\left\{
\begin{split}
&T(\cdot,s)F\in C^1((s,\infty);\, X_q(\mathbb R^3))
\cap C((s,\infty);\,D_q(A)),  \\
&\partial_tT(t,s)F+L_+(t)T(t,s)F=0
\quad\mbox{for $F\in X_q(\mathbb R^3)$},\,t\in (s,\infty),
\end{split}
\right.
\label{evo-eq1}
\end{equation}
\begin{equation}
\left\{
\begin{split}
&T(t,\cdot)F\in C^1((-\infty,t);\, X_q(\mathbb R^3)), \\
&\partial_sT(t,s)F=T(t,s)L_+(s)F \quad 
\mbox{for $F\in D_q(A)$},\, s\in (-\infty,t).
\end{split}
\right.
\label{evo-eq2}
\end{equation}
Furthermore, 
if $\|U_b\|\leq \alpha_j$ 
being small enough, 
to be precise, see below about how small it is for each item $j=1,2,3,4$,
then the evolution operator
$T(t,s)$ enjoys the following estimates with some constant 
$C=C(q,r,\alpha_j,\beta_0,\theta)>0$
whenever $[U_b]_\theta\leq \beta_0$, where $\|U_b\|$ and $[U_b]_\theta$ are given by \eqref{quan} and
$\beta_0>0$ is arbitrary.

\begin{enumerate}
\item
%(i)
Let $q\in (1,\infty)$ and $r\in [q,\infty]$, then
\begin{equation}
\|T(t,s)F\|_{r,\mathbb R^3} \leq C(t-s)^{-(3/q-3/r)/2}\|F\|_{q,\mathbb R^3}
\label{decay-1}
\end{equation}
for all $(t,s)$ with $t>s$ and $F\in X_q(\mathbb R^3)$.
To be precise, there is a constant
$\alpha_1=\alpha_1(q_0)>0$ such that if
$\|U_b\|\leq\alpha_1$, then 
the assertion above holds
for every $q\in (1,\infty)$ and $r\in [q,\infty]$.

Estimate \eqref{decay-1} holds true for the adjoint evolution operator $T(t,s)^*$ as well under
the same smallness of $\|U_b\|$ as above. 

\item
Let $1<q\leq r<\infty$, then
\begin{equation}
\|\nabla T(t,s)F\|_{r,\mathbb R^3}\leq
C(t-s)^{-1/2-(3/q-3/r)/2}(1+t-s)^{\max\{(1-3/r)/2,\, 0\}}
\|F\|_{q,\mathbb R^3}
\label{grad-new-form}
\end{equation}
for all $(t,s)$ with $t>s$ and $F\in X_q(\mathbb R^3)$.
To be precise, given $r_1\in (1,4/3]$,
there is a constant $\alpha_2=\alpha_2(r_1,q_0)\in (0,\alpha_1]$ such that if
$\|U_b\|\leq \alpha_2$, then 
the assertion above holds for every $r\in [r_1,\infty)$ and
$q\in (1,r]$, where $\alpha_1=\alpha_1(q_0)$ is the constant given in
the previous item.

Estimate \eqref{grad-new-form} holds true for the adjoint evolution
operator $T(t,s)^*$ as well under the same smallness of $\|U_b\|$ as above.

\item
Let $q\in (1,\infty)$ and $r\in [q,\infty]$, then
\begin{equation}
\|T(t,s)\mathbb P\mbox{\rm div $F$}\|_{r,\mathbb R^3}
\leq C(t-s)^{-1/2-(3/q-3/r)/2}(1+t-s)^{\max\{(3/q-2)/2,\, 0\}}
\|F\|_{q,\mathbb R^3}
\label{compo}
\end{equation}
for all $(t,s)$ with $t>s$ and $F\in L^q(\mathbb R^3)^{3\times 3}$ with
$(F\nu)|_{\partial\Omega}=0$ as well as
$\mbox{\rm div $F$}\in L^p_R(\mathbb R^3)$ with some $p\in (1,\infty)$.
To be precise, given $r_0\in [4,\infty)$,
if $\|U_b\|\leq \alpha_3$ with
$\alpha_3(r_0,q_0):=\alpha_2(r_0^\prime,q_0)$,
where $\alpha_2$ is the constant given in the previous item, then
the assertion above holds for every $q\in (1,r_0]$ and
$r\in [q,\infty]$.

\item
Let $1< q\leq r<\infty$, then
\begin{equation}
\begin{split}
&\quad \|\nabla T(t,s)\mathbb P\mbox{\rm div $F$}\|_{r,\mathbb R^3}  \\
&\leq C(t-s)^{-1-(3/q-3/r)/2}(1+t-s)^{\max\{(1-3/r)/2,\, 0\}+\max\{(3/q-2)/2,\, 0\}}
\|F\|_{q,\mathbb R^3}
\end{split}
\label{compo-grad}
\end{equation}
for all $(t,s)$ with $t>s$ and $F\in L^q(\mathbb R^3)^{3\times 3}$ with $(F\nu)|_{\partial\Omega}=0$ as well as
$\mbox{\rm div $F$}\in L^p_R(\mathbb R^3)$ with some $p\in (1,\infty)$.
To be  precise, given $r_0\in [4,\infty)$ and $r_1\in (1,4/3]$,
if $\|U_b\|\leq \alpha_4$ with
$\alpha_4(r_0,r_1,q_0):=\alpha_2\big(\min\{r_0^\prime,r_1\},q_0\big)$,
where $\alpha_2$ is the constant given in the item 2, then
the assertion above holds for $1<q\leq r<\infty$
with $q\in (1,r_0]$ as well as $r\in [r_1,\infty)$.

\end{enumerate}
\label{evo-op}
\end{theorem}
\begin{remark}
%(i)
According to
\eqref{grad-new-form}, the rate of decay is given by
\begin{equation}
\|\nabla T(t,s)F\|_{r,\mathbb R^3}\leq
\left\{
\begin{array}{ll}
C(t-s)^{-1/2-(3/q-3/r)/2}\|F\|_{q,\mathbb R^3} \quad & (r\leq 3),  \\
C(t-s)^{-3/2q}\|F\|_{q,\mathbb R^3} & (r>3),
\end{array}
\right.
\label{decay-2}
\end{equation}
for all $(t,s)$ with $t-s>2$ and $F\in X_q(\mathbb R^3)$.
This rate is the same as the one for the Stokes and Oseen semigroups in exterior domains due to Iwashita \cite{I},
Kobayashi and Shibata \cite{KShi},
Maremonti and Solonnikov \cite{MaSol}, see also \cite{Hi20} for the generalized Oseen evolution operator even with rotating effect
in which the time-dependent motion of a rigid body is prescribed.
It is known (\cite{MaSol, Hi11}) that the rate \eqref{decay-2} is sharp for the 
Stokes semigroup in exterior domains, whereas the optimality is not clear
for the problem under consideration.
In fact, one of the points of their argument 
is that the steady Stokes flow in exterior domains does not possess
fine summability such as $L^3(\Omega)$ at infinity even for the external force $f\in C_0^\infty(\Omega)$ unless
$\int_{\partial\Omega}\mathbb S(u,p)\nu\,d\sigma=0$.
Compared with that, if
$f\chi_\Omega+(\kappa+\mu\times x)\chi_B=\mbox{\rm div $F$}$ with $F$ satisfying the conditions 
in the item 3 of Theorem \ref{evo-op} for $T(t,s)\mathbb P\mbox{\rm div}$ 
and if $U$ is the steady Stokes-structure solution
to $AU=\mathbb P\mbox{\rm div $F$}$, then we have
\[
\int_{\partial\Omega}\mathbb S(u,p)\nu\,d\sigma=m\kappa=\int_B (\kappa+\mu\times x)\rho\,dx
=-\int_{\partial\Omega}(F\nu)\rho\,dx=0,
\]
yielding better decay of $u=U|_\Omega$, where $p$ is the associated pressure, and thus the desired rate
$\mbox{\eqref{decay-2}}_1$ even with $r>3$ for $e^{-(t-s)A}$ (case $U_b=0$) does not lead us to any contradiction
unlike the Stokes semigroup in exterior domains.
\label{rem-grad1}
\end{remark}
\begin{remark}
Since we know from \eqref{semi}, \eqref{decay-1} and \eqref{evo-eq1} that
$T(t,s)F\in D_r(A)\subset W^{1,r}(\mathbb R^3)$ 
for all $(t,s)$
with $t>s$ and $F\in X_q(\mathbb R^3)$, estimate of
$\|\nabla T(t,s)F\|_{r,\mathbb R^3}$ makes sense.
Of course, $\nabla T(t,s)F$ of the form \eqref{grad-form} never belongs to $X_r(\mathbb R^3)$. 
By \eqref{X1} with $U=T(t,s)F$, we have
\begin{equation*}
\begin{split}
\|\nabla T(t,s)F\|_{r,B}
&\leq C|\omega(t)|
\leq C\|T(t,s)F\|_{1,B}  \\
&\leq\left\{
\begin{array}{ll}
C\|F\|_{q,\mathbb R^3} & (t-s\leq 2), \\
C(t-s)^{-3/2q}\|F\|_{q,\mathbb R^3}\quad & (t-s>2),
\end{array}
\right.
\end{split}
\end{equation*}
for all $F\in X_q(\mathbb R^3)$ on account of \eqref{decay-1}.
Hence, estimate of $\|\nabla T(t,s)F\|_{r,\Omega}$ over the fluid region
is dominant in the sense that it determines \eqref{grad-new-form}.
\label{rem-grad2}
\end{remark}
\begin{remark}
The adjoint evolution operator $T(t,s)^*$ is studied
in subsection \ref{adjoint-backward}.
The duality argument shows 
\eqref{compo} of the operator $T(t,s)\mathbb P\mbox{\rm div}$, 
that plays a role for the proof of Theorem \ref{nonlinear} below and that 
corresponds to Lemma 8.1 of \cite{EMT} in which 
the additional condition $F|_B=0$ is imposed.
The reason why this condition is removed in Theorem \ref{evo-op}
is that the same estimate of $\nabla T(t,s)^*$ as in
\eqref{grad-new-form} is deduced over the whole space
$\mathbb R^3$, as mentioned above, rather than the exterior domain $\Omega$ solely.
Nevertheless, as pointed out by 
Ervedoza, Hillairet and Lacave \cite[remark after Corollary 3.10]{EHL14} as well as \cite{EMT},
$\nabla T(t,s)^*$ is not exactly adjoint of the operator $T(t,s)\mathbb P\mbox{\rm div}$ (unless $\rho =1$).
This is because several duality relations hold true with respect to the pairing \eqref{pairing} involving
the constant weight $\rho$, and this is why one needs the vanishing normal trace $(F\nu)|_{\partial\Omega}=0$
for \eqref{compo} in order that the boundary integral disappears by integration by parts even though $\rho\neq 1$.
Notice that $(F\nu)|_{\partial\Omega}$ from both directions coincide with each other since $\mbox{\rm div $F$}\in L^p(\mathbb R^3)$
for some $p\in (1,\infty)$.
\label{rem-grad3}
\end{remark}

%% NS

To proceed to the nonlinear problem \eqref{eq-perturb},
we do need the latter estimates \eqref{compo}--\eqref{compo-grad} in Theorem \ref{evo-op}
to deal with the nonlinear term $\eta\cdot\nabla u$ among four terms in \eqref{rhs}
as in \cite{EHL14, EMT}.
For the other terms one can discuss them by use of \eqref{decay-1} and \eqref{grad-new-form}
under a bit more assumptions on $\nabla u_b$ than \eqref{ass-3} below, however, if one applies \eqref{compo}--\eqref{compo-grad} partly
to the linear term $(\eta-u)\cdot\nabla u_b$ in \eqref{rhs},
then one needs less further assumption:
\begin{equation}
\nabla u_b
\in L^\infty(\mathbb R;\, L^3(\Omega))
\cap C^{\tilde\theta}_{\rm loc}(\mathbb R;\, L^3(\Omega)) \quad
\mbox{with some}\; 
\tilde\theta\in (0,1). 
\label{ass-3}
\end{equation}
Accordingly, in addition to \eqref{quan}, let us introduce the quantity
\begin{equation}
\|U_b\|^\prime:=\|U_b\|+\sup_{t\in\mathbb R}\|\nabla u_b(t)\|_{3,\Omega}
\label{quan2}
\end{equation}
but the H\"older seminorm of $\nabla u_b(t)$ is not needed since the 
H\"older condition in \eqref{ass-3} is used merely to show that a mild solution (solution to \eqref{int-NS} below)
becomes a strong solution to the initial value problem \eqref{perturbed}.
The second main result reads
\begin{theorem}
Suppose \eqref{ass-1}--\eqref{ass-2} and \eqref{ass-3}.  
There exists a constant $\alpha=\alpha(q_0,\beta_0,\theta)>0$
such that if $\|U_b\|^\prime\leq \alpha$ as well as $[U_b]_\theta\leq\beta_0$, then the following statement holds,
where 
$\|U_b\|^\prime$ is given by \eqref{quan2} and
$\beta_0>0$ is arbitrary: 
There is a constant $\delta=\delta(\alpha,\beta_0,\theta)>0$ 
such that if $U_0\in X_3(\mathbb R^3)$
satisfies $\|U_0\|_{3,\mathbb R^3}<\delta$, where $U_0$ is given by \eqref{IC-mono},
then problem \eqref{perturbed} admits a unique strong solution
\[
U\in C([s,\infty);\,X_3(\mathbb R^3))\cap C((s,\infty);\,D_3(A))\cap C^1((s,\infty);\,X_3(\mathbb R^3))
\]
which enjoys
\begin{equation}
\begin{split}
&\|u(t)\|_{q,\Omega}=o\big((t-s)^{-1/2+3/2q}\big), \\
&\|\nabla u(t)\|_{r,\Omega}
+|\eta(t)|+|\omega(t)|=o\big((t-s)^{-1/2}\big)
\end{split}
\label{main-decay}
\end{equation}
as $(t-s)\to\infty$ for every $q\in [3,\infty]$ and $r\in [3,\infty)$,
where $(u,\eta,\omega)=i(U)$, see \eqref{X1}. 
\label{nonlinear}
\end{theorem}

\begin{remark}
The decay rate $(t-s)^{-1/2}$ for $\|\nabla u(t)\|_{3,\Omega}$ in \eqref{main-decay} is new even when $U_b=0$,
see \cite{EMT} in which less rate is deduced.
This improvement is due to \eqref{nonlinear-est}.
\label{rem-non1}
\end{remark}
\begin{remark}
In view of the proof, we see that the
large time decay \eqref{main-decay} with $r\in [3,\infty)$ replaced by $r\in [3,\sigma_{0*})$
of the mild solution can be obtained even if 
$\nabla u_b\in L^\infty(\mathbb R;\,L^{\sigma_0}(\Omega))$
with some $\sigma_0\in (3/2,3]$ instead of \eqref{ass-3}, where
$1/\sigma_{0*}=1/\sigma_0-1/3$.
This solution becomes a strong one with values in $D_3(A)+D_{\sigma_0}(A)$
under the additional condition
$\nabla u_b\in C^{\tilde\theta}_{\rm loc}(\mathbb R;\, L^{\sigma_0}(\Omega))$ with some
$\tilde\theta\in (3/2\sigma_0-1/2,1)$ as well as
$\theta\in (3/2\sigma_0-1/2,1)$, where $\theta$ is given in \eqref{ass-2}.
\label{rem-non2}
\end{remark}

\subsection{Basic motions}
\label{basic-mo}

In this subsection we briefly discuss the basic motions, specifically the self-propelled motions, in the literature.
Those motions are stable as long as they are small enough when applying Theorem \ref{nonlinear}, but
%The others would be also stable by the same theorem if further analysis could be done in a successful way.
the details should be discussed elsewhere.
%our main results in the preceding subsection. 

By following the compactness argument due to Galdi \cite[Theorem 5.1]{Ga99} who studied %proved the existence of 
the steady problem %motion $u_b$ with $\nabla u_b\in L^2(\Omega)$ (Leray class) 
attached to the body of arbitrary shape,
it is possible to show the existence
of a solution of the Leray class $\nabla u_b\in L^2(\Omega)$ %that is, \eqref{ass-3} with $\sigma_0=2$,
to the steady problem associated with
\eqref{eq-fixed} when $u_*$ is independent of $t$ and small
in $H^{1/2}(\partial\Omega)$ as well as vanishing flux condition.
If we assume further $u_*\in H^{3/2}(\partial\Omega)$, then we have 
$\nabla u_b\in H^1(\Omega)\subset L^3(\Omega)$, see \eqref{ass-3},
$u_b\in W^{1,6}(\Omega)\subset L^\infty(\Omega)\cap C^{1/2}(\overline{\Omega})$.
The uniqueness of the solution in the small is also available when
the body is a ball, while this issue remains still open for the case
of arbitrary shape.
If, for instance, the steady rigid motion $\eta_b+\omega_b\times x$ obtained in this way fulfills %either (i) 
$\eta_b\cdot\omega_b\neq 0$, % or (ii) $\eta_b\neq 0,\, \omega_b=0$,
then the result due to Galdi and Kyed \cite{GaKy} 
concludes that the solution $u_b$ %of the Leray class 
enjoys a wake structure, %so that $u_b\in L^q(\Omega)$ for every $q>2$, 
yielding \eqref{ass-1} with $q_0\in (2,3)$ provided
$u_*$ is tangential to $\partial\Omega$.
Actually, we are able to deduce even more; indeed,
according to Theorems 5.1 and 5.2 of the same paper \cite{GaKy}, we have
pointwise estimates
\begin{equation}
|u_b(x)|\leq C|x|^{-1}, \qquad |\nabla u_b(x)|\leq C|x|^{-3/2}
\label{pointwise-NS}
\end{equation}
for large $|x|$ with further wake behavior so that $u_b$ and $\nabla u_b$ decays even faster outside the wake region.
With $\mbox{\eqref{pointwise-NS}}_1$ for $u_b$ at hand, one can use Theorem 1.2 of Silvestre, Takahashi and the 
present author \cite{HST17} to find even better summability \eqref{ass-1} with $q_0\in (4/3,3)$ on account of the 
self-propelling condition, that is related to the asymptotic structure of the Navier-Stokes flow in exterior domains,
see \cite{GaNe, Hi18-h}.

Another %but related 
observation is that the steady solution, denoted by $\{u_s,\eta_s,\omega_s\}$, constructed
by Galdi \cite[Theorem 5.1]{Ga99} %mentioned above 
to the problem \eqref{eq-fixed} in which the equations of motions are replaced by \eqref{shape-arbit}--\eqref{shape-arbit-s} %obtained 
%by a different change of variable in which rotation is also taken into account 
becomes a time-periodic solution to \eqref{eq-fixed} by the change of variable
\[
u_b(x,t)=Q(t)^\top u_s(Q(t)x), \quad
\eta_b(t)=Q(t)^\top\eta_s, \quad
\omega_b(t)=Q(t)^\top\omega_s
\]
with a suitable orthogonal matrix $Q(t)$ %determined by $\omega_s$
when the body is a ball.
This solution $u_b$ enjoys the desired summability properties mentioned above if $u_*\in H^{3/2}(\partial\Omega)$,
$\nu\cdot u_*|_{\partial\Omega}=0$ and $\eta_s\cdot\omega_s\neq 0$. 
Moreover, it is seen that
\[
u_b(x,t)-u_b(x,s)%=\int_s^t\partial_\tau u_b(x,\tau)\,d\tau
=\int_s^t %O(\tau)^\top 
\big[(\omega_b(\tau)\times x)\cdot\nabla u_b(x,\tau)-\omega_b(\tau)\times u_b(x,\tau)\big]\,d\tau.
\]
Then the pointwise decay $\mbox{\eqref{pointwise-NS}}_2$ for $\nabla u_s$ implies that
$u_b$ is globally Lipschitz continuous with values in $L^\infty(\Omega)$,
which combined with $u_b\in L^\infty(\mathbb R;\,L^\infty(\Omega))$ leads to
\eqref{ass-2} for every $\theta\in (0,1)$.
One can also verify \eqref{ass-3} from
the globally Lipschitz continuity of $\nabla u_b$ with values in $L^3(\Omega)$ if
$|x|\nabla^2u_s\in L^3(\Omega)$ (that needs further discussion about the steady motion $u_s$ when $\eta_s\cdot\omega_s\neq 0$).
Even though the H\"older condition in \eqref{ass-3} is not available,
the proof of Theorem \ref{nonlinear} tells us that we have still the asymptotic decay \eqref{main-decay}
for the mild solution, that is, solution to the integral equation \eqref{int-NS}.

Galdi \cite{Ga99} discussed the direct problem, while
a steady control problem %(inverse problem) in the steady state regime 
was studied in \cite{HST17, HST20}
by Silvestre, Takahashi and the present author; in fact,
we found, among others,
a tangential control $u_*$ (together with $u_b$ and $p_b$) which attains
a target rigid motion $\eta_b+\omega_b\times x$ 
to be small, where the shape of the body is arbitrary.
The solution obtained there also becomes a time-periodic one to \eqref{eq-fixed} with \eqref{ass-2} 
by the same reasoning as described in the previous paragraph.
Fine decay structure caused by the self-propelling condition was already deduced in \cite[Theorem 1.2]{HST17},
which implies \eqref{ass-1} with $q_0\in (3/2,3)$, no matter what $\{\eta_b,\omega_b\}$ would be.

%%  3
\section{Oseen-structure evolution operator} 
\label{oseen-structure-evo}

Before analyzing the Oseen-structure operator \eqref{oseen-op}--\eqref{oseen-term}, 
we begin with some preparatory results in the first three subsections:
the decomposition \eqref{decompo} to eliminate the pressure, 
justification of the monolithic formulation (of the resolvent problem),
and the other formulation of the Stokes-structure operator \eqref{stokes-op} to derive the resolvent estimate \eqref{para-resol}.
In those subsections, the shape of a rigid body is allowed to be arbitrary.
Let $B$ be a bounded domain with 
connected boundary of class $C^{1,1}$, that is assumed to be 
\[
B\subset B_1, \qquad
\int_B x\,dx=0,
\]
and set
\[
\Omega=\mathbb R^3\setminus \overline{B}, \qquad
m=\int_B\rho\,dx, \qquad
J=\int_B\big(|x|^2\mathbb I-x\otimes x\big)\rho\,dx.
\]

In subsection \ref{generation} we show that the Oseen-structure operator generates an evolution operator on the space
$X_q(\mathbb R^3)$ and, successively in subsection \ref{smoothing-estimate}, we deduce smoothing rates of the evolution operator
near the initial time.
Such rates of the associated pressure is studied in
subsection \ref{sm-pressure}.
Analysis of the pressure is rather technical part of this paper.
Subsection \ref{adjoint-backward} is devoted to investigation of the backward problem for the adjoint system.

%% 3.1
\subsection{Decomposition of $L^q_R(\mathbb R^3)$} %\eqref{decompo}}
\label{new-decompo}

In this subsection 
we establish the decomposition \eqref{decompo}.
To this end, let us start with the following preparatory lemma on the space $L^q_R(\mathbb R^3)$, see \eqref{leb-rigid}.
\begin{lemma}
Let $1<q<\infty$.
Then the Banach space $L^q_R(\mathbb R^3)$ is reflexive.
%
%\begin{equation}
Furthermore, the class
\begin{equation}
{\mathcal E}_R(\mathbb R^3):=\{U\in C_0^\infty(\mathbb R^3)^3;\; DU=O\;\mbox{\rm in $B$}\}
\label{dense-B}
\end{equation}
is dense in $L^q_R(\mathbb R^3)$.
\label{lem-reflex}
\end{lemma}

\begin{proof}
Set 
\[
J^q(\mathbb R^3):=\Big\{U\in L^q(\mathbb R^3)^3;\; U|_\Omega=0,\;\int_B U\rho\,dx=0,\; \int_B x\times U\rho\,dx=0\Big\}.
\]
We first show that
\begin{equation}
L^q(\mathbb R^3)^3=L^q_R(\mathbb R^3)\oplus J^q(\mathbb R^3).
\label{auxi-decompo}
\end{equation}
It is readily seen that $L^q_R(\mathbb R^3)\cap J^q(\mathbb R^3)=\{0\}$.
Given $U\in L^q(\mathbb R^3)$, we set
\[
V=U\chi_\Omega+(\eta+\omega\times x)\chi_B
\]
with
\[
\eta=\frac{1}{m}\int_BU\rho\,dx, \qquad \omega=J^{-1}\int_B x\times U\rho\,dx.
\]
Then we observe
\[
V\in L^q_R(\mathbb R^3), \qquad U-V\in J^q(\mathbb R^3),
\]
%We thus obtain 
yielding the decomposition \eqref{auxi-decompo}.

We next show that $J^q(\mathbb R^3)^\perp\subset L^{q^\prime}_R(\mathbb R^3)$,
where the annihilator is considered with respect to $\langle\cdot,\cdot\rangle_{\mathbb R^3,\rho}$ (but the constant weight $\rho$
does not play any role here).
Suppose that $U\in L^{q^\prime}(\mathbb R^3)^3$ satisfies
\begin{equation}
\langle U,\Psi\rangle_{\mathbb R^3,\rho}=0, \qquad \forall\Psi\in J^q(\mathbb R^3).
\label{auxi-an}
\end{equation}
Let $\Phi\in C_0^\infty(B)^{3\times 3}$, then we find
\[
\int_B\mbox{div $(\Phi+\Phi^\top)$}\,dx=0, \qquad
\int_B x\times \mbox{div $(\Phi+\Phi^\top)$}\,dx=0,
\]
so that $\mbox{div $(\Phi+\Phi^\top)$}\in J^q(\mathbb R^3)$ by setting zero outside $B$.
By \eqref{auxi-an} we are led to
\[
2\rho\langle DU,\Phi\rangle_B
%\int_B 2(Du):\Phi\,dy
=-\int_B U\cdot\mbox{div $(\Phi+\Phi^\top)$}\rho\,dx
=-\langle U, \mbox{div $(\Phi+\Phi^\top)$}\rangle_{\mathbb R^3,\rho}=0
\]
for all $\Phi\in C_0^\infty(B)^{3\times 3}$, yielding $U|_B\in {\rm RM}$. %$u\in L^{q^\prime}_B(\mathbb R^3)$
We thus obtain the desired inclusion relation.
Since the opposite inclusion is obvious, we infer 
$J^q(\mathbb R^3)^\perp=L^{q^\prime}_R(\mathbb R^3)$.
This combined with \eqref{auxi-decompo} leads us to
\[
L^q_R(\mathbb R^3)^*=\big[L^q(\mathbb R^3)^3/J^q(\mathbb R^3)\big]^*=J^q(\mathbb R^3)^\perp=L^{q^\prime}_R(\mathbb R^3),
\]
which implies that $L^q_R(\mathbb R^3)$ is reflexive.

Finally, let us show that the class \eqref{dense-B} is dense in $L^q_R(\mathbb R^3)$.
Given $U\in L^q_R(\mathbb R^3)$, we set $\eta+\omega\times x=U|_B$. %and fix $U_0\in C_0^\infty(\mathbb R^3)$, %as in \eqref{lift} below, 
Let us take the following lift of the rigid motion:
\begin{equation}
\begin{split}
\ell(\eta,\omega)(x)&:=
\frac{1}{2}\,\mbox{rot $\Big(\phi(x)\big(\eta\times x-|x|^2\omega\big)\Big)$}
 \\
&=\phi(x)(\eta+\omega\times x)+\nabla\phi(x)\times (\eta\times x-|x|^2\omega),
\end{split}
\label{lift0}
\end{equation}
where $\phi$ is a cut-off function satisfying
\begin{equation}
\phi\in C_0^\infty(B_3), \qquad
0\leq \phi\leq 1, \qquad 
\phi=1\;\; \mbox{on $B_2$}.
\label{cut0}
\end{equation}
Then we have
\[
\ell(\eta,\omega)\in C_0^\infty(B_3), \qquad
\mbox{div $\ell(\eta,\omega)$}=0, \qquad
\ell(\eta,\omega)|_{\overline{B}}=\eta+\omega\times x. 
\]
Since the lift \eqref{lift0} will be also used later,
it is convenient to introduce the lifting operator
\begin{equation}
\ell: (\eta,\omega)\mapsto \ell(\eta,\omega).
\label{lift-op0}
\end{equation}
Set $U_0=\ell(\eta,\omega)$.
One can take $u_j\in C_0^\infty(\Omega)$ satisfying
$\|u_j-(U-U_0)\|_{q,\Omega}\to 0$ as $j\to\infty$, and let us denote by $u_j$ again by setting zero outside $\Omega$.
We then put $V_j:=u_j+U_0\in C_0^\infty(\mathbb R^3)$ that attains $\eta+\omega\times x$ on $B$ 
(thus $V_j\in {\mathcal E}_R(\mathbb R^3)$) and 
satisfies $\|V_j-U\|_{q,\mathbb R^3}\to 0$ as $j\to\infty$, leading to the desired denseness.
The proof is complete.
\end{proof}
The decomposition \eqref{decompo-0} below was proved by Silvestre \cite[Theorem 3.2]{Sil02-a} when $q=2$.
Proposition \ref{prop-decom} may be regarded as generalization of her result by following the idea due to Wang and Xin \cite[Theorem 2.2]{WX}, who established the other decomposition \eqref{helm} for general $q\in (1,\infty)$.
It should be emphasized that the projection $\mathbb P$ associated with the decomposition \eqref{decompo-0} is symmetric
with respect to $\langle\cdot,\cdot\rangle_{\mathbb R^3,\rho}$ with constant weight $\rho$ 
as in \eqref{proj-sym-0} below unlike the projection $\mathbb Q$
associated with the other one \eqref{helm}.
In fact, this is the reason why we do need the following proposition in the present study.
\begin{proposition}
Let $1<q<\infty$.
Let $X_q(\mathbb R^3)$ and $Z_q(\mathbb R^3)$ be the spaces given respectively
by \eqref{sp-velo} and \eqref{sp-pressure}.
Then the class ${\mathcal E}(\mathbb R^n)$, see \eqref{dense-sub}, is dense in $X_q(\mathbb R^n)$ and 
\begin{equation}
L^q_R(\mathbb R^3)=X_q(\mathbb R^3)\oplus Z_q(\mathbb R^3),
\label{decompo-0}
\end{equation}
\begin{equation}
X_q(\mathbb R^3)^*=Z_q(\mathbb R^3)^\perp=X_{q^\prime}(\mathbb R^3), \qquad
Z_q(\mathbb R^3)^*=X_q(\mathbb R^3)^\perp=Z_{q^\prime}(\mathbb R^3),
\label{dual-an}
\end{equation}
where $Z_q(\mathbb R^3)^\perp$ and $X_q(\mathbb R^3)^\perp$
stand for the annihilators with respect to $\langle\cdot,\cdot\rangle_{\mathbb R^3,\rho}$, see \eqref{pairing}.
%Moreover, the class ${\mathcal E}(\mathbb R^3)$ given by \eqref{dense-sub} is dense in $X_q(\mathbb R^3)$.

Let $\mathbb P=\mathbb P_q$ be the projection from $L^q_R(\mathbb R^3)$ onto $X_q(\mathbb R^3)$, then 
\begin{equation}
\|\mathbb PU\|_{q,\mathbb R^3}\leq C\|U\|_{q,\mathbb R^3}
\label{proj-bdd}
\end{equation}
for all $U\in L^q_R(\mathbb R^3)$ with some constant $C>0$
as well as the relation $\mathbb P_q^*=\mathbb P_{q^\prime}$ in the sense that
\begin{equation}
\langle \mathbb P_qU,V\rangle_{\mathbb R^3,\rho}=\langle U,\mathbb P_{q^\prime}V\rangle_{\mathbb R^3,\rho}
\label{proj-sym-0}
\end{equation}
for all $U\in L^q_R(\mathbb R^3)$ and $V\in L^{q^\prime}_R(\mathbb R^3)$.
%Moreover, the class ${\mathcal E}(\mathbb R^3)$ given by \eqref{dense-sub} is dense in $X_q(\mathbb R^3)$.
If in particular $U\in L^q_R(\mathbb R^3)$ satisfies $u=U|_\Omega\in W^{1,q}(\Omega)$, then we have
$(\mathbb PU)|_\Omega\in W^{1,q}(\Omega)$ and
\begin{equation}
\|\nabla \mathbb PU\|_{q,\Omega}\leq C\big(\|U\|_{q,\mathbb R^3}+\|\nabla u\|_{q,\Omega}\big)
\label{grad-proj-bdd}
\end{equation} 
with some constant $C>0$ independent of $U$.
\label{prop-decom}
\end{proposition}

\begin{proof}
%Let us show the decomposition \eqref{decompo}.
Step 1.
We first verify that $X_q(\mathbb R^3)\cap Z_q(\mathbb R^3)=\{0\}$.
Suppose $U\in X_q(\mathbb R^3)\cap Z_q(\mathbb R^3)$ and set
\[
U|_\Omega=\nabla p, \qquad
U|_B=\eta+\omega\times x,
\]
then we have
\begin{equation}
\Delta p=0\quad \mbox{in $\Omega$}, \qquad
\partial_\nu p=\nu\cdot (\eta+\omega\times x)\quad \mbox{on $\partial\Omega$}
\label{laplace-p}
\end{equation}
with
\begin{equation}
\eta=\frac{-1}{m}\int_{\partial\Omega}p\nu\,d\sigma, \qquad
\omega=-J^{-1}\int_{\partial\Omega}x\times (p\nu)\,d\sigma.
\label{eta-om}
\end{equation}
It is observed that %$p$ is smooth up to $\partial\Omega$ and
$\nabla p\in L^2_{\rm loc}(\overline{\Omega})$ %(by bootstrap argument)
even though $q$ is close to $1$ and 
that $p-p_\infty$ with some constant $p_\infty\in\mathbb R$
(resp. $\nabla p$) behaves like the fundamental solution $\frac{-1}{4\pi |x|}$ (resp. its gradient) 
of the Laplacian 
%(up to constant)
at inifinity since $\nabla p\in L^q(\Omega)$; in fact, they go to zero even faster because
$\int_{\partial\Omega}\partial_\nu p\,d\sigma\, (=0)$ is the coefficient of the leading term of the asymptotic expansion at infinity.
This %yields $\nabla p\in L^2(\Omega)$ and 
together with \eqref{laplace-p}--\eqref{eta-om} justifies the equality \eqref{p-energy} below
%by computing $\int_\Omega (\Delta p)(p-p_\infty)\phi_R\,dx\, (=0)$ 
when multiplying the Laplace equation in \eqref{laplace-p} by $(p-p_\infty)\phi_R$
%and then letting $R\to\infty$, where 
with
\begin{equation}
\phi_R(x):=\phi(x/R), 
\label{para-cut}
\end{equation}
where $\phi$ is fixed as in \eqref{cut0},
and then letting $R\to\infty$.
Indeed, since
\[
\lim_{R\to\infty}\int_{2R<|x|<3R}|\nabla p||p-p_\infty||\nabla \phi_R|\,dx=0,
\]
we obtain
\begin{equation}
\int_\Omega |\nabla p|^2\,dx
=\int_{\partial\Omega}(\partial_\nu p)(p-p_\infty)\,d\sigma
%=\int_{\partial\Omega}(\partial_\nu p)p\,d\sigma
=\int_{\partial\Omega}\nu\cdot (\eta+\omega\times x)p\,d\sigma
=-m|\eta|^2-(J\omega)\cdot\omega.
\label{p-energy}
\end{equation}
This implies that $\nabla p=0$ as well as $\eta=\omega=0$, leading to $U=0$.

Step 2.
Let us show that the class ${\mathcal E}(\mathbb R^3)$ is dense in $X_q(\mathbb R^3)$.
Given $U\in X_q(\mathbb R^3)$, we set $\eta+\omega\times x=U|_B$ and
proceed as in the latter half of Lemma \ref{lem-reflex}.
Using the same lifting function $U_0=\ell (\eta,\omega)\in C^\infty_{0,\sigma}(B_3)$ given by \eqref{lift0},
%of the rigid motion $\eta+\omega\times x$, 
we see that $(U-U_0)|_\Omega$ belongs to the space
\begin{equation}
L^q_\sigma(\Omega)=\{u\in L^q(\Omega);\; \mbox{div $u$}=0\;\mbox{in $\Omega$},\;\nu\cdot u|_{\partial\Omega}=0\}.
\label{ext-underly}
\end{equation}
Since $C^\infty_{0,\sigma}(\Omega)$ is dense in $L^q_\sigma(\Omega)$ (\cite{Ga-b, Mi82, SiSo}),
the proof of the desired denseness is complete in the similar manner to 
descriptions in the last paragraph of the proof of Lemma \ref{lem-reflex}.
This was already shown in \cite[Appendix A.2]{EMT} within the context of the other decomposition \eqref{helm}.

We next prove that 
${\mathcal E}(\mathbb R^3)^\perp=X_q(\mathbb R^3)^\perp=Z_{q^\prime}(\mathbb R^3)$
with respect to $\langle\cdot,\cdot\rangle_{\mathbb R^3,\rho}$
by following the argument in \cite{Sil02-a} (in which
the case $q=2$ is discussed).
Let $U\in {\mathcal E}(\mathbb R^3)$ and $V\in Z_{q^\prime}(\mathbb R^3)$, with
$i(U)=(u,\eta_u,\omega_u)$ and $i(V)=(\nabla p,\eta,\omega)$, where $\eta$ and $\omega$ are specified as \eqref{eta-om}.
Then it is readily seen that
\[
\langle U,V\rangle_{\mathbb R^3,\rho}
=\int_\Omega\mbox{div $(up)$}\,dx
-\eta_u\cdot\int_{\partial\Omega}p\nu\,d\sigma-\omega_u\cdot\int_{\partial\Omega} x\times (p\nu)\,d\sigma=0.
\]
This relation holds even for all $U\in X_q(\mathbb R^3)$ and $V\in Z_{q^\prime}(\mathbb R^3)$
thanks to the denseness observed above.
We are thus led to
$Z_{q^\prime}(\mathbb R^3)\subset X_q(\mathbb R^3)^\perp$.

Conversely, let $V\in L^{q^\prime}_R(\mathbb R^3)$ %, with $i(V)=(v,\eta,\omega)$, 
satisfy
$\langle\Phi,V\rangle_{\mathbb R^3,\rho}=0$ for all $\Phi\in {\mathcal E}(\mathbb R^3)$.
Set $(v,\eta,\omega)=i(V)$ and $(\phi,\eta_\phi,\omega_\phi)=i(\Phi)$.
If, in particular,  %we take
\begin{equation}
%i(\Phi)=(\phi,0,0), \qquad
\eta_\phi=\omega_\phi=0, \qquad
\phi\in C_{0,\sigma}^\infty(\Omega), \; %\qquad \mbox{supp\Phi\subset\Omega},
\label{spe-test}
\end{equation}
%$u\in C_{0,\sigma}^\infty(\Omega)$ together with $\eta_U=\omega_U=0$, where $(u,\eta_U,\omega_U)=i(U)$,
then it follows that $v=\nabla p$ for some $p\in L^{q^\prime}_{\rm loc}(\overline{\Omega})$.
This implies that 
\begin{equation*}
\begin{split}
m\eta_\phi\cdot\eta+\omega_\phi\cdot (J\omega)
&=-\int_{\Omega}\mbox{div $(\phi p)$}\,dx  \\
&=-\eta_\phi\cdot\int_{\partial\Omega}p\nu\,d\sigma
-\omega_\phi\cdot\int_{\partial\Omega}x\times (p\nu)\,d\sigma
\end{split}
\end{equation*}
for all $\Phi\in {\mathcal E}(\mathbb R^3)$.
Since $\eta_\phi$ and $\omega_\phi$ are arbitrary, we conclude
\[
\eta=\frac{-1}{m}\int_{\partial\Omega}p\nu\,d\sigma, \qquad
\omega=-J^{-1}\int_{\partial\Omega}x\times (p\nu)\,d\sigma,
\]
that is, $V\in Z_{q^\prime}(\mathbb R^3)$.
This proves $\mathcal{E}(\mathbb R^3)^\perp\subset Z_{q^\prime}(\mathbb R^3)$ and, therefore,
\begin{equation}
{\mathcal E}(\mathbb R^3)^\perp=
X_q(\mathbb R^3)^\perp=Z_{q^\prime}(\mathbb R^3), \qquad
Z_{q^\prime}(\mathbb R^3)^\perp=X_q(\mathbb R^3).  %=\overline{{\mathcal E}(\mathbb R^3)}.
\label{an}
\end{equation}
In fact, the latter follows from the former with $X_q(\mathbb R^3)^{\perp\perp}=X_q(\mathbb R^3)$
since $X_q(\mathbb R^3)$ is closed in the reflexive space $L^q_R(\mathbb R^3)$, see Lemma \ref{lem-reflex}.
Now, \eqref{an} immediately leads to \eqref{decompo-0} when $q=2$ (\cite[Theorem 3.2]{Sil02-a}).

Step 3.
To complete the proof of \eqref{decompo-0} for the case $q\neq 2$, 
we use it %\eqref{decompo-0} 
for the case $q=2$. %due to Silvestre \cite[Theorem 3.2]{Sil02-a} (it is also verified in the preceding step).
Given $U\in {\mathcal E}_R(\mathbb R^3)$, see \eqref{dense-B},
there is a unique pair of $V\in X_2(\mathbb R^3)$ and $W\in Z_2(\mathbb R^3)$ such that $U=V+W$.
By following the argument developed by Wang and Xin \cite[Theorem 2.2]{WX},
we will prove that $V\in X_q(\mathbb R^3)$ and $W\in Z_q(\mathbb R^3)$ along with estimate \eqref{decom-est} below,
where $(\nabla p,\eta,\omega)=i(W)$, see \eqref{X1}.
Set also
$(u,\eta_u,\omega_u)=i(U)$ and
$(v,\eta_v,\omega_v)=i(V)$.
Since 
$\nu\cdot v|_{\partial\Omega}=\nu\cdot(\eta_v+\omega_v\times x)$, see \eqref{hidden-bc1},
we have
\begin{equation*}
\begin{split}
&\Delta p=\mbox{div $u$}\quad\mbox{in $\Omega$}, \\
%&\partial_\nu p=\nu\cdot (u-\eta_v-\omega_v\times x)=\nu\cdot (\eta+\omega\times x)\quad\mbox{on $\partial\Omega$}.
&\nu\cdot (\nabla p-u)=-\nu\cdot (\eta_v+\omega_v\times x)\quad\mbox{on $\partial\Omega$}.
\end{split}
%\label{ext-neumann}
\end{equation*}
Let $\ell$ be the lifting operator given by \eqref{lift0} and
\eqref{lift-op0}.
Then the problem above is rewritten as
\begin{equation}
\begin{split}
&\Delta p=\mbox{div $\big(u-\ell(\eta_v,\omega_v)\big)$}
\quad\mbox{in $\Omega$}, \\
&%\partial_\nu p=\nu\cdot \big(u-{\mathcal L}(\eta_v,\omega_v)\big)
\nu\cdot (\nabla p-u)=-\nu\cdot \ell(\eta_v,\omega_v)
\quad\mbox{on $\partial\Omega$}.
\end{split}
\label{ext-neumann}
\end{equation}
One is then able to apply the theory of the Neumann problem
in exterior domains, see \cite{Mi82, SiSo}, to \eqref{ext-neumann} to infer
\begin{equation}
\|\nabla p\|_{q,\Omega}\leq C\big(\|u\|_{q,\Omega}+|\eta_v|+|\omega_v|\big)
\leq C\big(\|U\|_{q,\mathbb R^3}+|\eta|+|\omega|\big),
\label{siso-est}
\end{equation}
where \eqref{X-norm} is also taken into account.
Let us %fix $R>1$ and 
single out a solution $p$ in such a way that $\int_{\Omega_3}p\,dx=0$ to obtain
\begin{equation}
\|p\|_{q,\Omega_3}\leq C\|\nabla p\|_{q,\Omega}.
\label{mean-single}
\end{equation}

We show that
\begin{equation}
|\eta|+|\omega|\leq C\|U\|_{q,\mathbb R^3}
\label{auxi-neumann}
\end{equation}
for all $U\in {\mathcal E}_R(\mathbb R^3)$.
Suppose the contrary, then one can take a sequence $\{U_k\}\subset{\mathcal E}_R(\mathbb R^3)$ and the corresponding 
%rigid motions $\eta_k+\omega_k\times x$
%
\begin{equation}
\begin{split}
&(\nabla p_k)\chi_\Omega+(\eta_k+\omega_k\times x)\chi_B\in Z_2(\mathbb R^3), \qquad
\int_{\Omega_3} p_k\,dx=0, \\
&\eta_k=\frac{-1}{m}\int_{\partial\Omega}p_k\nu\,d\sigma, \qquad
\omega_k=-J^{-1}\int_{\partial\Omega}x\times (p_k\nu)\,d\sigma, \\
&v_k\chi_\Omega+(\eta_{v_k}+\omega_{v_k}\times x)\chi_B\in X_2(\mathbb R^3),
\end{split}
\label{seq-corres}
\end{equation}
such that
\begin{equation}
\lim_{k\to\infty}\|U_k\|_{q,\mathbb R^3}=0, \qquad |\eta_k|+|\omega_k|=1.
\label{contra-seq}
\end{equation}
By virtue of \eqref{siso-est}--\eqref{mean-single} together with \eqref{contra-seq},
there are $\eta,\,\omega\in\mathbb R^3$ and 
$p\in \widehat W^{1,q}_{(0)}(\Omega)$
as well as
a subsequence, still denoted by the same symbol, such that, along the subsequence,
\begin{equation*}
\begin{split}
&\mbox{w-$\displaystyle{\lim_{k\to\infty}\nabla p_k=\nabla p}$}\quad \mbox{in $L^q(\Omega)$}, \qquad 
\lim_{k\to\infty}\|p_k-p\|_{q,\Omega_3}=0, \\ %\quad\forall R>1, \\
&\lim_{k\to\infty}\eta_k=\eta, \qquad
\lim_{k\to\infty}\omega_k=\omega.
\end{split}
\end{equation*}
%and moreover
%\[
%\langle\partial_\nu p_k
%\]
By \eqref{seq-corres} and by the trace estimate
\[
\|p_k-p\|_{q,\partial\Omega}
\leq C\|\nabla p_k-\nabla p\|_{q,\Omega_3}^{1/q}\|p_k-p\|_{q,\Omega_3}^{1-1/q}+C\|p_k-p\|_{q,\Omega_3},
\]
we obtain
\[
\eta=\frac{-1}{m}\int_{\partial\Omega}p\nu\,d\sigma, \qquad
\omega=-J^{-1}\int_{\partial\Omega}x\times (p\nu)\,d\sigma.
\]
Since $p_k$ %$\eta_k$ and $\omega_k$ 
obeys \eqref{ext-neumann} with $u_k=U_k|_\Omega$,
we find $\Delta p=0$ in $\Omega$, so that $\partial_\nu p|_{\partial\Omega}$ makes sense, and
\begin{equation*}
\begin{split}
&\langle \nu\cdot (\nabla p_k-u_k-\nabla p), \psi\rangle_{\partial\Omega}
=\langle\nabla p_k-u_k-\nabla p, \nabla\psi\rangle_{\Omega_3} 
\to 0, \\
&-\nu\cdot (\eta_{v_k}+\omega_{v_k}\times x)=\nu\cdot (\eta_k+\omega_k\times x-\eta_{u_k}-\omega_{u_k}\times x)
\to \nu\cdot (\eta+\omega\times x),
\end{split}
\end{equation*}
as $k\to\infty$ for all $\psi\in C_0^\infty(B_3)$ % where $R>1$ is fixed.
on account of \eqref{contra-seq}, where $\eta_{u_k}+\omega_{u_k}\times x=U_k|_B$.
As a consequence, $p,\,\eta$ and $\omega$ solve \eqref{laplace-p}--\eqref{eta-om}, leading to $\eta=\omega=0$ as explained
in Step 1.
This contradicts $|\eta|=|\omega|=1$ and thus concludes \eqref{auxi-neumann}, which combined with \eqref{siso-est} proves
\begin{equation}
\|V\|_{q,\mathbb R^3}+\|W\|_{q,\mathbb R^3}\sim %\leq 
\|V\|_{q,\mathbb R^3}+%C\big(
\|\nabla p\|_{q,\Omega}+|\eta|+|\omega|%\big)
\leq C\|U\|_{q,\mathbb R^3}
\label{decom-est}
\end{equation}
for all $U\in {\mathcal E}_R(\mathbb R^3)$.
With \eqref{decom-est} at hand,
given $U\in L^q_R(\mathbb R^3)$, one can construct $V\in X_q(\mathbb R^3)$ and $W\in Z_q(\mathbb R^3)$ such that
$U=V+W$ along with the same estimate \eqref{decom-est} since ${\mathcal E}_R(\mathbb R^3)$ is dense in $L^q_R(\mathbb R^3)$
by Lemma \ref{lem-reflex}.
This completes the proof of
\eqref{decompo-0} and \eqref{proj-bdd}.
From \eqref{decompo-0} we find
\begin{equation*}
X_q(\mathbb R^3)^*=\big[L^q_R(\mathbb R^3)/Z_q(\mathbb R^3)\big]^*=Z_q(\mathbb R^3)^\perp,  \qquad
Z_q(\mathbb R^3)^*=\big[L^q_R(\mathbb R^3)/X_q(\mathbb R^3)\big]^*=X_q(\mathbb R^3)^\perp,
\end{equation*}
which along with \eqref{an} implies \eqref{dual-an} and \eqref{proj-sym-0} as well.

Step 4.
Finally, let us show \eqref{grad-proj-bdd} for $U\in L^q_R(\mathbb R^3)$ with $(u,\eta_u,\omega_u)=i(U)$, see \eqref{X1},
when additionally assuming $u\in W^{1,q}(\Omega)$.
Since $\mathbb PU=U-W$ with
\[
W=(\nabla p)\chi_\Omega+(\eta+\omega\times x)\chi_B\in Z_q(\mathbb R^3),
\]
it suffices to prove that
\begin{equation}
\|\nabla^2p\|_{q,\Omega}\leq C\big(\|U\|_{q,\mathbb R^3}+\|\nabla u\|_{q,\Omega}\big),
\label{2nd-p}
\end{equation}
where $p$ should obey
\begin{equation*}
\begin{split}
&\Delta p=\mbox{div $u$}\quad \mbox{in $\Omega$}, \\
&\partial_\nu p=\nu\cdot (u-\eta_u-\omega_u\times x+\eta+\omega\times x)\quad\mbox{on $\partial\Omega$},
\end{split}
\end{equation*}
with $\eta$ and $\omega$ given by \eqref{eta-om}.
As in \eqref{ext-neumann}, we utilize the lift 
\[
U_0:=\ell(\eta_u-\eta, \omega_u-\omega)
\]
to rewrite the problem above as
%Note that the compatibility condition is fulfilled for this Neumann problem.
%
\begin{equation*}
\begin{split}
&\Delta p=\mbox{div $\big(u-U_0\big)$} \quad \mbox{in $\Omega$}, \\
&\partial_\nu p=\nu\cdot \big(u-U_0\big) \quad\mbox{on $\partial\Omega$}.
\end{split}
\end{equation*}
By \cite[Lemma 2.3]{GiSo89} we have
\begin{equation*}
\begin{split}
\|\nabla^2p\|_{q,\Omega}
&\leq C\|u-U_0\|_{W^{1,q}(\Omega)}  \\
&\leq C\big(\|u\|_{W^{1,q}(\Omega)}+|\eta_u|+|\omega_u|+|\eta|+|\omega|\big)  \\
&\leq C\big(\|U\|_{q,\mathbb R^3}+\|\nabla u\|_{q,\Omega}+|\eta|+|\omega|\big) %\|p\|_{W^{1,q}(\Omega_2)}\big)
\end{split}
\end{equation*}
which combined with \eqref{decom-est}
concludes \eqref{2nd-p}.
The proof is complete.
\end{proof}

%% 3.2
\subsection{\eqref{resol1} is equivalent to \eqref{st-str-resol}}
\label{resol-equi}

This subsection claims that the resolvent equation \eqref{resol1} in $X_q(\mathbb R^3)$ 
with the Stokes-structure operator $A$ defined by \eqref{stokes-op}
is equivalent to the boundary value problem \eqref{st-str-resol}. % as in \cite[Proposition 3.4]{EMT}.
This can be proved in the same way as in \cite[Proposition 3.4]{EMT} on the same issue for \eqref{resol11}
with the other operator $\widetilde A$, see also \cite{TT04}, however, one has to prove the %following proposition 
equivalence independently of the existing literature about the operator $\widetilde A$.
In fact, with the following proposition at hand, we will then verify $A=\widetilde A$ afterwards by use of uniqueness of solutions
to the boundary value problem \eqref{st-str-resol}.
\begin{proposition}
Let $1<q<\infty$ and $\lambda\in\mathbb C$.
Suppose that $F\in X_q(\mathbb R^3)$ and $(f,\kappa,\mu)=i(F)$ through \eqref{X1}.
If
\begin{equation}
u\in W^{2,q}(\Omega), \quad p\in \widehat W^{1,q}(\Omega), \;
%\mbox{\color{blue}(introduce ...)}  
\quad
(\eta,\omega)\in\mathbb C^3\times \mathbb C^3
\label{sol-cl-resol}
\end{equation}
fulfill \eqref{st-str-resol}, then \eqref{resol1} holds with
$U=u\chi_\Omega+(\eta+\omega\times x)\chi_B$.
Conversely, assume that $U\in D_q(A)$ satisfies \eqref{resol1}.
Then there exists a pressure $p\in \widehat W^{1,q}(\Omega)$ such that $(u,\eta,\omega)=i(U)$ together with $p$ enjoys 
\eqref{st-str-resol}. % as well as $u\in W^{2,q}(\Omega)$.
\label{prop-equi}
\end{proposition}

\begin{proof}
Concerning the first half, it is obvious that
\begin{equation*}
\begin{split}
&U=u\chi_\Omega+(\eta+\omega\times x)\chi_B\in D_q(A),  \\
&(\nabla p)\chi_\Omega
+\left[\frac{-1}{m}\int_{\partial\Omega}p\nu\,d\sigma+\left(-J^{-1}\int_{\partial\Omega}y\times (p\nu)\,d\sigma_y\right)\times x\right]\chi_B
\in Z_q(\mathbb R^3)
\end{split}
\end{equation*}
and that,
in view of \eqref{decompo}--\eqref{sp-pressure}, applying the projection $\mathbb P$ to \eqref{st-str-resol} yields \eqref{resol1}.

To show the second half,
%let $F\in X_q(\mathbb R^3)$ and $(f,\kappa,\mu)=i(F)$.
suppose that $U\in D_q(A)$ fulfills \eqref{resol1}, %with $F\in X_q(\mathbb R^3)$ and set $(u,\eta,\omega)=i(U)$.
then $u$ possesses the desired regularity together with the boundary condition $u|_{\partial\Omega}=\eta+\omega\times x$.
%Let $\Phi\in {\mathcal E}(\mathbb R^3)$ be a test function, then
We also find
\[
\langle (\lambda+A)U,\Phi\rangle_{\mathbb R^3,\rho}=\langle F,\Phi\rangle_{\mathbb R^3,\rho}
\]
which is rewritten as
\begin{equation}
\begin{split}
&\quad \lambda\langle u,\phi\rangle_\Omega+\lambda m\eta\cdot\eta_\phi+\lambda (J\omega)\cdot\omega_\phi   \\
&\quad -\langle \Delta u,\phi\rangle_\Omega+\int_{\partial\Omega}(2Du)\nu\,d\sigma\cdot\eta_\phi
+\int_{\partial\Omega}y\times (2Du)\nu\,d\sigma_y\cdot\omega_\phi   \\
&=\langle f,\phi\rangle_\Omega+m\kappa\cdot\eta_\phi+(J\mu)\cdot\omega_\phi
\end{split}
\label{weak-resol}
\end{equation}
for all $\Phi\in {\mathcal E}(\mathbb R^3)$ with $(\phi,\eta_\phi,\omega_\phi)=i(\Phi)$
on account of \eqref{pairing} and \eqref{proj-sym-0}.
Let us, in particular, choose \eqref{spe-test},
then we get
\[
\langle\lambda u-\Delta u-f, \phi\rangle_\Omega=0
\]
for all $\phi\in C_{0,\sigma}^\infty(\Omega)$.
%{\color{blue}( ... not introduced yet ... )}
Since $\lambda u-\Delta u-f\in L^q(\Omega)$, there is a function $p\in L^q_{\rm loc}(\overline{\Omega})$ with
$\nabla p\in L^q(\Omega)$ such that
\[
\lambda u-\Delta u-f=-\nabla p
\]
in $\Omega$.
By taking into account this equation in \eqref{weak-resol}, we are led to
\begin{equation*}
\begin{split}
&\left(\lambda m\eta+\int_{\partial\Omega}(2Du)\nu\,d\sigma-\int_{\partial\Omega}p\nu\,d\sigma-m\kappa\right)\cdot\eta_\phi  \\
&+\left(\lambda J\omega+\int_{\partial\Omega}y\times (2Du)\nu\,d\sigma_y-\int_{\partial\Omega}y\times (p\nu)\,d\sigma_y-J\mu\right)\cdot\omega_\phi=0
\end{split}
\end{equation*}
for all $\eta_\phi,\,\omega_\phi\in\mathbb C^3$, which yields the equations for the rigid body in \eqref{st-str-resol}.
The proof is complete.
\end{proof}
\begin{proposition}
The operator $A$ defined by \eqref{stokes-op} coincides with $\widetilde A$, where the latter operator is defined by
\eqref{stokes-op} in which $\mathbb P$ is replaced by the other projection $\mathbb Q: L^q(\mathbb R^3)\to X_q(\mathbb R^3)$ associated
with \eqref{helm}.
\label{two-st-str}
\end{proposition}

\begin{proof}
By Proposition \ref{prop-equi} and by \cite[Proposition 3.4]{EMT},
both equations \eqref{resol11} and \eqref{resol1} are equivalent to the boundary value problem \eqref{st-str-resol}.
Hence, it suffices to show that the only solution to \eqref{st-str-resol} with $\lambda =1$ and $F=0$
within the class \eqref{sol-cl-resol} is the trivial one.
In fact, given $U\in D_q(A)$, one can see from %the knowledge about $\widetilde A$ due to 
\cite[Theorem 6.1]{EMT} that there is a unique
$\widetilde U\in D_q(\widetilde A)=D_q(A)$ satisfying $(1+\widetilde A)\widetilde U=(1+A)U$ in $X_q(\mathbb R^3)$.
Then the uniqueness for \eqref{st-str-resol} (which we will show below) implies that 
$\widetilde U=U$ and, thereby, $\widetilde AU=\widetilde A\widetilde U=AU$.

Consider the problem \eqref{st-str-resol} with $\lambda=1,\, f=0$ and $\kappa=\mu=0$.
Then we have $u\in W^{1,2}_{\rm loc}(\overline{\Omega})$ even if $q$ is close to $1$.
Moreover, from \eqref{sol-cl-resol} it follows that $\{u,p-p_\infty\}$ with some constant $p_\infty\in\mathbb R$ and $\nabla u$ 
behave like the fundamental solution and its gradient to the Stokes resolvent system
even though $q$ is large.
Let $\phi_R$ be the same cut-off function as in \eqref{para-cut}. Computing
\[
0=\int_\Omega \big(u-\Delta u+\nabla (p-p_\infty)\big)\cdot u\phi_R\, dx
\]
and then %observing 
letting $R\to\infty$, where %leads to
\[
\lim_{R\to\infty}\int_{2R<|x|<3R}|\mathbb S(u,p-p_\infty)||u||\nabla\phi_R|\,dx=0,
\]
we are led to %lead us to
\[
\|u\|_{2,\Omega}^2+m|\eta|^2+(J\omega)\cdot\omega+2\|Du\|_{2,\Omega}^2=0,
\]
which concludes that $u=0$ and $\eta=\omega=0$.
The proof is complete.
\end{proof}

% 3.3
\subsection{Stokes-structure resolvent}
\label{stokes-structure}

Estimate \eqref{para-resol} for large $|\lambda|$ is established in \cite[Proposition 5.1]{EMT}, where the authors of \cite{EMT}
however omit the proof since it is a slight variation of the proof given by Maity and Tucsnak \cite[Theorem 3.1]{MT} %who discussed
on the same issue for the Stokes-structure system in a bounded container.
In fact, their idea based on the reformulation below rather than \eqref{resol1} or \eqref{resol11} is fine, but the variation needs a couple of
nontrivial modifications since $\Omega$ is unbounded.  %for instance, the choice of a lift of the rigid motion.
%and thus seems more than "slight".
For completeness,
this subsection is devoted to the details of reconstruction of the proof of \eqref{para-resol} for large $|\lambda|$.
Once we have that, a contradiction argument performed in \cite[Section 6]{EMT} leads to \eqref{para-resol}
even for $\lambda\in\Sigma_\varepsilon$ close to the origin $\lambda=0$.
The reason why the details are provided here is that we do need a representation of the resolvent, see \eqref{A-series} together with
\eqref{A0-inv} below,
to deduce a useful estimate of the associated pressure near the initial time in subsection \ref{sm-pressure}.

The underlying space of the other formulation of the resolvent problem \eqref{st-str-resol} is
\[
Y_q:=L^q_\sigma(\Omega)\times \mathbb C^3\times\mathbb C^3,
\]
where $L^q_\sigma(\Omega)$ given by \eqref{ext-underly}
is the standard underlying space when considering
the Stokes resolvent in the exterior domain $\Omega$.
By $P_\Omega: L^q(\Omega)\to L^q_\sigma(\Omega)$ we denote the classical Fujita-Kato projection associated with
the Helmholtz decomposition \cite{Mi82, SiSo}
\[
L^q(\Omega)=L^q_\sigma(\Omega)\oplus \{\nabla p\in L^q(\Omega);\; p\in\widehat W^{1,q}(\Omega)\}.% L^q_{\rm loc}(\overline{\Omega})\}.
\]
Then the Stokes operator $A_\Omega$ in exterior domains is defined by
\[
D_q(A_\Omega)=L^q_\sigma(\Omega)\cap W^{1,q}_0(\Omega)\cap W^{2,q}(\Omega), \qquad
A_\Omega=-P_\Omega\Delta.
\]

Let us reformulate the resolvent system \eqref{st-str-resol} by following the procedure due to Maity and Tucsnak \cite{MT}.
Let $F\in X_q(\mathbb R^3),\, 1<q<\infty$, and consider \eqref{st-str-resol} with $(f,\kappa,\mu)=i(F)$, see \eqref{X1}.
Given $(\eta,\omega)\in\mathbb C^3\times \mathbb C^3$, we take
a lifting function 
\[
U_0=\ell(\eta,\omega)
\]
of the rigid motion $\eta+\omega\times x$, where $\ell$ is given by \eqref{lift0} and
\eqref{lift-op0}.
Obviously, we have
\begin{equation}
\|\ell(\eta,\omega)\|_{W^{2,q}(\Omega)}\leq C|(\eta,\omega)|.
\label{lift-est}
\end{equation}
The fluid part of \eqref{st-str-resol} is reduced to finding 
$\widetilde u:=u-U_0\in D(A_\Omega)$ %with the pressure $p$ 
that obeys
\begin{equation}
%\begin{split}
\lambda (\widetilde u+P_\Omega U_0)+A_\Omega\widetilde u-P_\Omega\Delta U_0=P_\Omega f
%\end{split}
\label{resol-fluid}
\end{equation}
in $L^q_\sigma(\Omega)$, while the associated pressure consists of
\begin{equation}
\begin{split}
&p=p_1+p_2-\lambda p_3+p_4 \quad\mbox{with} \\
&p_1=N(\Delta\widetilde u), \quad 
p_2=N(\Delta U_0), \quad
p_3=N(U_0), \quad
p_4=N(\ell(\kappa,\mu))
\end{split}
\label{resol-pressure}
\end{equation}
in terms of the Neumann operator
$N: E_q(\Omega)\ni h\mapsto N(h):=\psi\in \widehat W^{1,q}_{(0)}(\Omega)$
which singles out a unique solution to
\begin{equation*}
\Delta\psi=\mbox{div $h$}\quad \mbox{in $\Omega$}, \qquad
\partial_\nu \psi=\nu\cdot h\quad \mbox{on $\partial\Omega$}, \qquad
\int_{\Omega_3}\psi\,dx=0,
%\label{neumann-op}
\end{equation*}
where
$E_q(\Omega):=\{h\in L^q(\Omega)^3;\; \mbox{div $h$}\in L^q(\Omega)\}$.
Then we have
\begin{equation}
\|N(h)\|_{q,\Omega_3}\leq C\|\nabla N(h)\|_{q,\Omega_3}
\leq C\|\nabla N(h)\|_{q,\Omega}\leq C\|h\|_{q,\Omega}.
\label{neumann-est}
\end{equation}
Note that
\[
\nu\cdot f=\nu\cdot (\kappa+\mu\times x)=\nu\cdot \ell(\kappa,\mu) \qquad
\mbox{on $\partial\Omega$},
\]
where
the function $\ell(\kappa,\mu)$ is introduced for description of
$p_4$ in \eqref{resol-pressure} since $\kappa+\mu\times x\notin E_q(\Omega)$. 
In \eqref{resol-fluid} one can not write $A_\Omega U_0$ because 
$U_0$ never belongs to $D(A_\Omega)$ on account of 
$U_0|_{\partial\Omega}=\eta+\omega\times x$.  %on $\partial\Omega$.

We rewrite
%by following the procedure due to Maity and Tucsnak \cite{MT}.
the equation of balance for linear momentum
\[
\lambda m\eta+\int_{\partial\Omega}\mathbb S(\widetilde u+U_0,p_1+p_2-\lambda p_3+p_4)\nu\,d\sigma=m\kappa
\]
in \eqref{st-str-resol} as
\begin{equation}
\lambda\left(m\eta+\int_{\partial\Omega}p_3\nu\,d\sigma\right)
+\int_{\partial\Omega}\mathbb S(\widetilde u,p_1)\nu\,d\sigma
+\int_{\partial\Omega}\mathbb S(U_0,p_2)\nu\,d\sigma
=m\kappa+\int_{\partial\Omega}p_4\nu\,d\sigma.
\label{eq-lm}
\end{equation}
Likewise, the equation of balance for angular momentum is described as
\begin{equation}
\begin{split}
\lambda\left(J\omega+\int_{\partial\Omega}x\times (p_3\nu)\,d\sigma\right)
+\int_{\partial\Omega}x\times \mathbb S(\widetilde u,p_1)\nu\,d\sigma
&+\int_{\partial\Omega}x\times \mathbb S(U_0,p_2)\nu\,d\sigma   \\
&=J\mu+\int_{\partial\Omega}x\times (p_4\nu)\,d\sigma.
\end{split}
\label{eq-am}
\end{equation}
It is convenient to introduce
\begin{equation}  
K  
\left( 
\begin{array}{c}    
\eta \\   
\omega    
\end{array}       
\right)    
=\left(  
\begin{split}   
&m\eta+\int_{\partial\Omega}N(\ell(\eta,\omega))\nu\,d\sigma  \\  
&J\omega+\int_{\partial\Omega}x\times N(\ell(\eta,\omega))\nu\,d\sigma
\end{split}
\right) 
\label{K-mat}
\end{equation}
\begin{equation}  
Q_1\widetilde u=\left(  
\begin{split}    
&\int_{\partial\Omega}\mathbb S(\widetilde u,N(\Delta\widetilde u))\nu\,d\sigma  \\ 
&\int_{\partial\Omega}x\times \mathbb S(\widetilde u,N(\Delta\widetilde u))\nu\,d\sigma
\end{split} 
\right)   
\label{Q1}  
\end{equation}
\begin{equation}  
Q_2(\eta,\omega)=\left(  
\begin{split}  
&\int_{\partial\Omega}\mathbb S(\ell(\eta,\omega),N(\Delta\ell(\eta,\omega))\nu\,d\sigma \\
&\int_{\partial\Omega}x\times \mathbb S(\ell(\eta,\omega),N(\Delta\ell(\eta,\omega))\nu\,d\sigma
\end{split}  
\right).      
\label{Q2}   
\end{equation}
Note that
$K\in\mathbb C^{6\times 6}$ is independent of
the choice of lift of the rigid motion $\eta+\omega\times x$ since $\psi=N(\ell(\eta,\omega))$
solves
\[
\Delta\psi=0\quad\mbox{in $\Omega$}, \qquad
\partial_\nu\psi=\nu\cdot (\eta+\omega\times x)\quad\mbox{on $\partial\Omega$},
\qquad \int_{\Omega_3}\psi\,dx=0,
\]
and that $K$ is invertible, see \cite[Lemma 3.8]{MT}.
Hence, the equations \eqref{eq-lm}--\eqref{eq-am} read
\begin{equation}
\lambda\left(
\begin{array}{c}
\eta \\
\omega
\end{array}
\right)+K^{-1}Q_1\widetilde u+K^{-1}Q_2(\eta,\omega)
=\left(
\begin{array}{c}
\kappa \\
\mu
\end{array}
\right).
\label{resol-solid}
\end{equation}
By \eqref{resol-solid}
one can describe the term $\lambda P_\Omega U_0$ 
of \eqref{resol-fluid} as
\[
\lambda P_\Omega U_0=
P_\Omega\ell(\lambda\eta,\lambda\omega)=P_\Omega\ell
\left((\kappa,\mu)-K^{-1}Q_1\widetilde u-K^{-1}Q_2(\eta,\omega)\right),
\]
from which \eqref{resol-fluid} is reduced to
\begin{equation}
\begin{split}
\lambda\widetilde u+A_\Omega\widetilde u-P_\Omega\Delta U_0
-P_\Omega\ell\left(K^{-1}Q_1\widetilde u+K^{-1}Q_2(\eta,\omega)\right)
&=P_\Omega f-P_\Omega\ell(\kappa,\mu)  \\
&=f-\ell(\kappa,\mu),
\end{split}
\label{resol-fluid2}
\end{equation}
where the last equality above follows from 
$f-\ell(\kappa,\mu)\in L^q_\sigma(\Omega)$.

%Taking into account \eqref{resol-solid}--\eqref{resol-fluid2}, we
Let us introduce the other Stokes-structure operator $\mathbb A$ acting on $Y_q$ by
\begin{equation}
\begin{split}
&D_q(\mathbb A)=D_q(A_\Omega)\times \mathbb C^3\times \mathbb C^3, \\ %\qquad
&\mathbb A
\left(
\begin{array}{c}
\widetilde u \\
(\eta,\omega)
\end{array}
\right)
=
\left(
\begin{array}{cc}
1 & -P_\Omega \ell \\
0 & 1
\end{array}
\right)
\left(
\begin{array}{cc}
A_\Omega & -P_\Omega\Delta\ell \\
K^{-1}Q_1 & K^{-1}Q_2
\end{array}
\right)
\left(
\begin{array}{c}
\widetilde u \\
(\eta,\omega)
\end{array}
\right),
\end{split}
\label{other-A}
\end{equation}
then, in view of \eqref{resol-solid}--\eqref{resol-fluid2},
the resolvent system \eqref{st-str-resol} is reformulated as
%(\cite[Proposition 3.9]{MT})
%
\begin{equation}
%\begin{split}
(\lambda+\mathbb A)\left(
\begin{array}{c}
\widetilde u \\
(\eta,\omega)
\end{array}
\right)=
\left(
\begin{array}{c}
f -\ell(\kappa,\mu) \\
(\kappa,\mu)
\end{array}
\right) \qquad\mbox{in $Y_q$}.
\label{resol2}
\end{equation}

The operator $\mathbb A$ is splitted into
\[
\mathbb A=\mathbb A_0+\mathbb A_1
\]
with
\begin{equation}
\mathbb A_0\left(
\begin{array}{c}
\widetilde u \\
(\eta,\omega)
\end{array}
\right)=\left(
\begin{array}{cc}
1 & -P_\Omega \ell \\
0 & 1 
\end{array}
\right)
\left(
\begin{array}{cc}
A_\Omega & -P_\Omega\Delta \ell \\
0 & 0
\end{array}
\right)
\left(
\begin{array}{c}
\widetilde u \\
(\eta,\omega)
\end{array}
\right),
%\end{split}
\label{A0}
\end{equation}
\begin{equation}
\mathbb A_1
\left(
\begin{array}{c}
\widetilde u \\
(\eta,\omega)
\end{array}
\right)
=\left(
\begin{array}{cc}
1 & -P_\Omega \ell \\
0 & 1
\end{array}
\right)
\left(
\begin{array}{cc}
0 & 0 \\
K^{-1}Q_1 & K^{-1}Q_2
\end{array}
\right)
\left(
\begin{array}{c}
\widetilde u \\
(\eta,\omega)
\end{array}
\right).
\label{A1}
\end{equation}
%Then we see that $\mathbb A_1$ is subordinate to $\mathbb A_0$.
Let $\lambda\in\mathbb C\setminus (-\infty,0]$, then $\lambda +\mathbb A_0$ is invertible and
\begin{equation}
(\lambda +\mathbb A_0)^{-1}=\left(
\begin{array}{cc}
(\lambda +A_\Omega)^{-1} & \lambda^{-1}(\lambda +A_\Omega)^{-1}P_\Omega\Delta
\ell  \\
0 & \lambda^{-1}
\end{array}
\right)
\label{A0-inv}
\end{equation}
in ${\mathcal L}(Y_q)$. Moreover, we find
\begin{equation}
\|(\lambda+\mathbb A_0)^{-1}\|_{{\mathcal L}(Y_q)}\leq C_\varepsilon|\lambda|^{-1}
\label{A0-est}
\end{equation}
for all $\lambda\in\Sigma_\varepsilon$ with $|\lambda|\geq 1$ (say), see \eqref{para-sector},
where $\varepsilon \in (0,\pi/2)$ is fixed arbitrarily,
because of the parabolic resolvent estimate of the Stokes operator $A_\Omega$.

We will show that $\mathbb A_1$ is subordinate to $\mathbb A_0$.
To this end,
%In order to investigate the operator $\mathbb A_1$, 
let us deduce estimates of $Q_1$ and $Q_2$ given  by \eqref{Q1}--\eqref{Q2}.
By the trace estimate together with \eqref{lift-est} and \eqref{neumann-est}, we have
%for every $\gamma>0$, there is a constant $C_\gamma>0$ such that
%
\begin{equation}
%\begin{split}
|Q_1\widetilde u|_{\mathbb C^3\times \mathbb C^3}
\leq C\big(\|A_\Omega\widetilde u\|_{q,\Omega}+\|\widetilde u\|_{q,\Omega}\big)
%\end{split}
\label{Q1-est}
\end{equation}
as well as
\begin{equation*}
|Q_2(\eta,\omega)|_{\mathbb C^3\times \mathbb C^3}\leq
C|(\eta,\omega)|. %_{\mathbb C^3\times\mathbb C^3}.
\end{equation*}
Those estimates together with \eqref{lift-est} lead us to
\begin{equation*}
\begin{split}
&\quad \|\ell\left(K^{-1}Q_1\widetilde u+K^{-1}Q_2(\eta,\omega)\right)\|_{W^{1,q}(\Omega)}
+\left|K^{-1}Q_1\widetilde u+K^{-1}Q_2(\eta,\omega)\right|_{\mathbb C^3\times \mathbb C^3}   \\
&\leq C\big(\|A_\Omega\widetilde u\|_{q,\Omega}+\|\widetilde u\|_{q,\Omega}
+|(\eta,\omega)|\big).
\end{split}
\end{equation*}
Since the support of the lifting function \eqref{lift0}
is bounded, the Rellich theorem implies that,
for any sequence $\{\widetilde u_j, (\eta_j,\omega_j)\}$ being bounded in $D_q(A_\Omega)\times \mathbb C^3\times \mathbb C^3$, 
one can subtract a convergent subsequence in $L^q(\Omega)\times \mathbb C^3\times \mathbb C^3$ from 
$\left\{\ell\big(K^{-1}Q_1\widetilde u_j+K^{-1}Q_2(\eta_j,\omega_j)\big),\, K^{-1}Q_1\widetilde u_j+K^{-1}Q_2(\eta_j,\omega_j)\right\}$.
Along the same subsequence,
$\mathbb A_1\left(\begin{array}{c}
\widetilde u_j \\
(\eta_j,\omega_j)
\end{array}
\right)$
is also convergent in $Y_q$ as $P_\Omega$ is bounded on $L^q(\Omega)$.
Hence, $\mathbb A_1$ is a compact operator from $D_q(\mathbb A_0)=D_q(A_\Omega)\times \mathbb C^3\times \mathbb C^3$    %D_q(\mathbb A)$ 
(endowed with the graph norm) into $Y_q$.
One can then employ a perturbation theorem \cite[Chapter III, Lemma 2.16]{EN} to conclude that
$\mathbb A_1$ is $\mathbb A_0$-bounded and its $\mathbb A_0$-bound is zero; that is,
for every small $\delta>0$ there is a constant $C_\delta>0$ satisfying
\begin{equation*}
%\begin{split}
\quad \left\|\mathbb A_1\left(
\begin{array}{c}
\widetilde u \\
(\eta,\omega)
\end{array}
\right)
\right\|_{Y_q}  \\
\leq \delta\left\|\mathbb A_0\left(
\begin{array}{c}
\widetilde u \\
(\eta,\omega)
\end{array}
\right)
\right\|_{Y_q}+C_\delta\left\|\left(
\begin{array}{c}
\widetilde u \\
(\eta,\omega)
\end{array}
\right)
\right\|_{Y_q}
%\end{split}
\end{equation*}
where $\|\cdot\|_{Y_q}$ is obviously given by the sum of $\|\cdot\|_{q,\Omega}$ and $|(\cdot,\cdot)|_{\mathbb C^3\times \mathbb C^3}$.
%with $\gamma >0$ being arbitrary.
This combined with \eqref{A0-est} allows us to take 
$\Lambda_\varepsilon>0$ for each $\varepsilon \in (0,\pi/2)$ such that
\begin{equation}
\|(\lambda+\mathbb A)^{-1}\|_{{\mathcal L}(Y_q)}\leq C_\varepsilon|\lambda|^{-1}
\label{A-est}
\end{equation}
for all $\lambda\in \Sigma_\varepsilon$ with 
$|\lambda|\geq \Lambda_\varepsilon$,
where the resolvent is described as the Neumann series in ${\mathcal L}(Y_q)$ for such $\lambda$:
\begin{equation}
\begin{split}
&(\lambda+\mathbb A)^{-1}=(\lambda+\mathbb A_0)^{-1}
\sum_{j=0}^\infty\left[-\mathbb A_1(\lambda+\mathbb A_0)^{-1}\right]^j  \\
&\mbox{with}\quad
\sum_{j=0}^\infty 
\left\|\mathbb A_1(\lambda+\mathbb A_0)^{-1}\right\|_{{\mathcal L}(Y_q)}^j
\leq 2.
\end{split}
\label{A-series}
\end{equation}
With \eqref{A-est} at hand, we 
immediately obtain
\[
\|\widetilde u\|_{q,\Omega}+|(\eta,\omega)|\leq C_\varepsilon|\lambda|^{-1}\big(\|f\|_{q,\Omega}+|(\kappa,\mu)|\big)
\]
for \eqref{resol2},
from which together with \eqref{lift-est} we are led to
\[
\|u\|_{q,\Omega}+|(\eta,\omega)|\leq C_\varepsilon|\lambda|^{-1}\big(\|f\|_{q,\Omega}+|(\kappa,\mu)|\big)
\]
for the solution to \eqref{st-str-resol}.
On account of Proposition \ref{prop-equi}, we conclude
\[
\|U\|_{q,\mathbb R^3}\leq C_\varepsilon|\lambda|^{-1}\|F\|_{q,\mathbb R^3}
\]
for \eqref{resol1} with $C_\varepsilon>0$ independent of
$\lambda\in\Sigma_\varepsilon$ satisfying 
$|\lambda|\geq \Lambda_\varepsilon$, where $\varepsilon\in (0,\pi/2)$
is arbitrary.

%% 3.4
\subsection{Generation of the evolution operator}
\label{generation}

In this subsection we show the generation of the evolution operator by 
the family $\{L_+(t);\, t\in\mathbb R\}$
of the Oseen-structure operators \eqref{oseen-op}--\eqref{oseen-term}.
It is of parabolic type (in the sense of Tanabe and Sobolevskii \cite{Ta, Y}) with the properties \eqref{semi}--\eqref{evo-eq2}.

Let $1<q<\infty$.
We first see that, for each fixed $t\in\mathbb R$, the operator $-L_+(t)$ generates an analytic semigroup on the space $X_q(\mathbb R^3)$.
This is readily verified as follows by a simple perturbation argument.
We fix $\varepsilon\in (0,\pi/2)$.
Given $F\in X_q(\mathbb R^3)$, let us take $U=(\lambda+A)^{-1}F$ with $\lambda\in\Sigma_\varepsilon$, see \eqref{para-sector}, in \eqref{perturb-est}
and employ \eqref{para-resol} to obtain
\begin{equation}
\|B(t)(\lambda+A)^{-1}F\|_{q,\mathbb R^3}
%\leq C\|(\eta_b(t)-u_b(t))\cdot\nabla u\|_{q,\Omega}
\leq C\big(|\lambda|^{-1/2}+|\lambda|^{-1}\big)
\|U_b\|\|F\|_{q,\mathbb R^3}.
\label{neumann}
\end{equation}
%for all $F\in X_q(\mathbb R^3)$.and $\lambda\in\Sigma_\phi$.
Hence, there is a constant $c_0>0$ dependent on $\|U_b\|$ but independent of $t$
such that if 
$\lambda\in\Sigma_\varepsilon$ fulfills $|\lambda|\geq c_0$, then the 
right-hand side of \eqref{neumann}
is bounded from above by $\frac{1}{2}\|F\|_{q,\mathbb R^3}$, yielding
$\lambda\in\rho(-L_+(t))$ subject to
\[
\|(\lambda+L_+(t))^{-1}\|_{{\mathcal L}(X_q(\mathbb R^3))}\leq 2\|(\lambda+A)^{-1}\|_{{\mathcal L}(X_q(\mathbb R^3))}
\leq \frac{C}{|\lambda|}
\]
for every $t\in\mathbb R$ by the Neumann series argument. % and $\lambda\in\Sigma_\phi$ with $\|\lambda|\geq c_0$.

We next verify the regularity of $L_+(t)$ in $t$ that allows us to apply 
the theory of parabolic evolution operators,
see \cite[Chapter 5]{Ta}.
In fact, by the same computations as above 
with use of \eqref{grad}, it follows from \eqref{ass-2}--\eqref{quan} that
\begin{equation}
\begin{split}
&\quad \|(L_+(t)-L_+(s))(\lambda+L_+(\tau))^{-1}F\|_{q,\mathbb R^3}  \\
&=\big\|(B(t)-B(s))(\lambda+A)^{-1}\big[1+B(\tau)(\lambda+A)^{-1}\big]^{-1}F\big\|_{q,\mathbb R^3}  \\
&\leq C(c_0^{-1/2}+c_0^{-1})(t-s)^\theta\,
[U_b]_\theta \|F\|_{q,\mathbb R^3}
\end{split}
\label{hoelder}
\end{equation}
for all $F\in X_q(\mathbb R^3)$, $\lambda\in\Sigma_\varepsilon$ with $|\lambda|\geq c_0$, and $t,\,s,\,\tau\in\mathbb R$ with $t>s$.
As a consequence, %for every compact interval $I\subset\mathbb R$, 
the family $\{L_+(t);\, t\in\mathbb R\}$ generates an evolution operator
$\{T(t,s);\, s,t\in I,\, s\leq t\}$ on $X_q(\mathbb R^3)$ for every compact interval $I\subset \mathbb R$, 
which provides the family
$\{T(t,s);\, -\infty<s\leq t<\infty\}$ with the properties \eqref{semi}--\eqref{evo-eq2} 
%on $X_q(\mathbb R^3)$.
by uniqueness of evolution operators.
%which gives us the evolution operator $\{T(t,s);\, -\infty<s\leq t<\infty\}$.
Notice that the locally H\"older continuity in $t$ of $u_b(t)$ (with values in $L^\infty(\Omega)$)
as well as $\eta_b(t)$ is enough just for generation of the evolution operator, however, the globally H\"older continuity \eqref{ass-2}
is needed for further studies.
This point is indeed the issue of the next subsection.

%% 3.5
\subsection{Smoothing estimates of the evolution operator}
\label{smoothing-estimate}

This subsection is devoted to the $L^q$-$L^r$ smoothing estimates \eqref{decay-1}--\eqref{grad-new-form}
for all $(t,s)$ with $t-s\in (0,\tau_*]$, where $\tau_*\in (0,\infty)$
is fixed arbitrarily.
Those rates themselves are quite standard, nevertheless, the point is to show that the constants in \eqref{decay-1}--\eqref{grad-new-form} may be dependent on $\tau_*$, however,
can be taken uniformly with respect to such $t,\,s$ under the globally H\"older condition \eqref{ass-2} on $u_b(t)$ and $\eta_b(t)$.
The similar studies are found in \cite[Lemma 3.2]{Hi18}.
\begin{proposition}
Suppose \eqref{ass-1} and \eqref{ass-2}.
Let $1<q<\infty$,
$r\in [q,\infty]$ (except $r=\infty$ for \eqref{grad-new-form})
 and $\tau_*,\, \alpha_0,\, \beta_0\in (0,\infty)$.
Then there is a constant $C=C(q,r,\tau_*,\alpha_0,\beta_0,\theta)>0$ such that both \eqref{decay-1} and \eqref{grad-new-form} hold for 
all $(t,s)$ with $t-s\in (0,\tau_*]$ and $F\in X_q(\mathbb R^3)$ whenever
$\|U_b\|\leq \alpha_0$ and $[U_b]_\theta\leq\beta_0$, where $\|U_b\|$ and $[U_b]_\theta$ are given by \eqref{quan}.
Moreover, we have
\begin{equation}
\begin{split}
\lim_{t\to s}\; (t-s)^{(3/q-3/r)/2}\|T(t,s)F\|_{r,\mathbb R^3}&=0, \qquad r\in (q,\infty],  \\
\lim_{t\to s}\; (t-s)^{1/2+(3/q-3/r)/2}\|\nabla T(t,s)F\|_{r,\mathbb R^3}&=0, \qquad r\in [q,\infty), 
\end{split}
\label{sm-little}
\end{equation}
for every $F\in X_q(\mathbb R^3)$, where the convergence is uniform on each precompact set of $X_q(\mathbb R^3)$.
\label{prop-smooth}
\end{proposition}

\begin{proof}
Let us fix $\tau_*>0$. 
In the proof of construction of the evolution operator,
%due to Tanabe \cite[Chapter 5]{Ta}, 
an important step is to show that
\begin{equation}
\|L_+(t)T(t,s)\|_{{\mathcal L}(X_q(\mathbb R^3))}\leq C(t-s)^{-1}
\label{tanabe-est}
\end{equation}
as well as
\begin{equation}
\|T(t,s)\|_{{\mathcal L}(X_q(\mathbb R^3))}\leq C
\label{tanabe-0}
\end{equation}
for $t-s\leq\tau_*$.
If we look into details of deduction of \eqref{tanabe-est}--\eqref{tanabe-0},
see Tanabe \cite[Chapter 5]{Ta},
we find that both constants $C=C(q,\tau_*,\alpha_0,\beta_0,\theta)>0$ 
%{\color{blue} (..... might be independent of $\alpha_0$ ? ......)}
can be taken uniformly
in $(t,s)$ with $t-s\leq\tau_*$ 
%as well as in $[U_b]_\theta \in (0,\beta_0]$ 
on account of the globally H\"older continuity \eqref{hoelder}.
%where $\beta_0>0$ is fixed arbitrarily. This is the whole thing.
See also the discussions in \cite[Lemma 3.2]{Hi18}.
Based on \eqref{tanabe-est}--\eqref{tanabe-0},
we take $U(t)=T(t,s)F$ with $F\in X_q(\mathbb R^3)$ in $\mbox{\eqref{equi-LA}}_2$ and use
\eqref{grad} to infer
\begin{equation}
\|\nabla T(t,s)F\|_{q,\Omega}\leq C(t-s)^{-1/2}\|F\|_{q,\mathbb R^3}
\label{grad-smoothing}
\end{equation}
for $t-s\leq\tau_*$ with some constant $C=C(q,\tau_*,\alpha_0,\beta_0,\theta)>0$.
It follows from \eqref{tanabe-0}--\eqref{grad-smoothing} together with the Gagliardo-Nirenberg inequality that 
\begin{equation}
\|T(t,s)F\|_{r,\Omega}\leq C(t-s)^{-(3/q-3/r)/2}\|F\|_{q,\mathbb R^3}
\label{0th-smoothing}
\end{equation}
for $t-s\leq\tau_*$ provided $1/q-1/r< 1/3$
(actually, $1/q-1/r\leq 1/3$ if $q\neq 3$).
From the relation \eqref{X1} with $U(t)=T(t,s)F$, it is readily seen that
\begin{equation}
\|U(t)\|_{r,B}\leq C(|\eta(t)|+|\omega(t)|)\leq
C\int_B|U(y,t)|\,dy\leq C\|U(t)\|_{q,\mathbb R^3}\leq C\|F\|_{q,\mathbb R^3}
\label{rigid-est1}
\end{equation}
and, from \eqref{grad-form}, that
\begin{equation}
\|\nabla U(t)\|_{r,B}\leq C|\omega(t)|\leq C\|F\|_{q,\mathbb R^3}.
\label{rigid-est2}
\end{equation}
Estimates \eqref{0th-smoothing} and \eqref{rigid-est1} imply \eqref{decay-1}
for $t-s\leq\tau_*$ if $1<q\leq r\leq \infty \;(q\neq\infty)$ and $1/q-1/r< 1/3$,
but the latter restriction can be eventually removed by using the semigroup property \eqref{semi}. % (at most three times)
Then \eqref{grad-smoothing} together with \eqref{decay-1} for $t-s\leq\tau_*$ leads us to
\[
\|\nabla T(t,s)F\|_{r,\Omega}\leq C(t-s)^{-1/2-(3/q-3/r)/2}\|F\|_{q,\mathbb R^3}
\]
for $t-s\leq\tau_*$, which along with \eqref{rigid-est2} concludes
\eqref{grad-new-form} for such $t,\, s$.

It remains to show \eqref{sm-little}.
If $F\in {\mathcal E}(\mathbb R^3)\subset D_q(A),\, 1<q<\infty$,
then it follows from \eqref{grad}, \eqref{equi-LA} and the boundedness
\[
\|L_+(t)T(t,s)(k+L_+(s))^{-1}\|_{{\mathcal L}(X_q(\mathbb R^3))}\leq C
\]
near $t=s$ (\cite[Chapter 5, Theorem 2.1]{Ta}),
where $k>0$ is fixed large enough, that
\[
\|\nabla T(t,s)F\|_{q,\Omega}\leq C\big(\|AF\|_{q,\mathbb R^3}+\|F\|_{q,\mathbb R^3}\big)
\]
near $t=s$.
This combined with \eqref{rigid-est2} (with $r=q$)
implies that
\[
\lim_{t\to s}\; (t-s)^{1/2}\|\nabla T(t,s)F\|_{q,\mathbb R^3}=0
\]
for every $F\in {\mathcal E}(\mathbb R^3)$ and, therefore,
every $F\in X_q(\mathbb R^3)$ since ${\mathcal E}(\mathbb R^3)$ is dense in $X_q(\mathbb R^3)$,
see Proposition \ref{prop-decom}.
The other behaviors in \eqref{sm-little} are verified easier.
The proof is complete.
\end{proof}

We next derive the H\"older estimate of the evolution operator.
This plays a role to verify that a mild solution becomes a strong one to the nonlinear initial value problem in section \ref{stability}.
The argument is similar to
the one for the autonomous case based on the theory of analytic semigroups,
but not completely the same. % as the one for the autonomous case based on the theory of analytic semigroups.
It can be found, for instance, in the paper \cite[Lemma 2.8]{Te83} by Teramoto who used the fractional powers of generators, however,
in order to justify the argument there, one has to deduce estimate of 
$\|(k+L_+(t))^\alpha (k+L_+(s))^{-\alpha}\|_{{\mathcal L}(X_q(\mathbb R^3))}$
independent of $(t,s)$, where $k>0$ is fixed large enough, as pointed out by Farwig and Tsuda \cite[Lemma 3.6]{FT22}.
In the latter literature, the authors discuss the desired estimate of the fractional powers
by use of bounded ${\mathcal H}^\infty$-calculus of generators uniformly in $t$.
%Since it is not the main issue of this paper, we do not intend to furnish the best result, but 
Instead, in this paper, we take easier way
to deduce the following result, %less optimal result (regarding a restriction \eqref{restr} on the summability exponents)
which is enough for later use in section \ref{stability}.
\begin{proposition}
Suppose \eqref{ass-1} and \eqref{ass-2}.
Let $j\in\{0,\,1\}$, $1<q<\infty$, $r\in [q,\infty]$
(except $r=\infty$ for $j=1$) and $\tau_*,\alpha_0,\beta_0\in (0,\infty)$.
Assume that $q$ and $r$ satisfy
\begin{equation}
\frac{1}{q}-\frac{1}{r}<\frac{2-j}{3}\qquad (j=0,\,1).
\label{restr}
\end{equation}
Given $\mu$ satisfying
\begin{equation}
0<\mu<1-\frac{j}{2}-\frac{3}{2}\left(\frac{1}{q}-\frac{1}{r}\right),
\label{degree-hoel}
\end{equation}
set
\[
\kappa=\max \left\{\frac{j}{2}+\frac{3}{2}\left(\frac{1}{q}-\frac{1}{r}\right)+\mu,\;\frac{1}{2}\right\}.
\]
Then there is a constant $C=C(\mu,q,r,\tau_*,\alpha_0,\beta_0,\theta)>0$ such that
\begin{equation}
\|\nabla^j T(t+h,s)F-\nabla^j T(t,s)F\|_{r,\mathbb R^3}
\leq C(t-s)^{-\kappa}h^\mu\|F\|_{q,\mathbb R^3}
\label{hoel-est}
\end{equation}
for all $(t,s)$ with $t-s\in (0,\tau_*]$, $h\in (0,1)$ and $F\in X_q(\mathbb R^3)$, whenever $\|U_b\|\leq\alpha_0$ and
$[U_b]_\theta\leq\beta_0$, where $\|U_b\|$ and $[U_b]_\theta$ are given by \eqref{quan}.
\label{hoelder-evo}
\end{proposition}

\begin{proof}
Set $U(t)=T(t,s)F$, that satisfies the equation
\[
U(t)=e^{-(t-s)A}F-V(t), \qquad V(t)=\int_s^t e^{-(t-\tau)A}B(\tau)U(\tau)\,d\tau,
\]
in terms of the fluid-structure semigroup $e^{-tA}$.
By $L^q$-$L^r$ estimates of the semigroup $e^{-tA}$ due to \cite{EMT} we get
%From the standard estimate
%
\begin{equation}
\|\nabla^j \big(e^{-(t+h-s)A}-e^{-(t-s)A}\big)F\|_{r,\mathbb R^3}
\leq C(t-s)^{-j/2-(3/q-3/r)/2-\mu}h^\mu\|F\|_{q,\mathbb R^3}
\label{hoelder-stokes}
\end{equation}
for every $\mu\in (0,1)$.
%From this together with
Proposition \ref{prop-smooth} %\eqref{perturb-est}
implies that
\[
\|B(t)U(t)\|_{q,\mathbb R^3}\leq C\|U_b\|(t-s)^{-1/2}\|F\|_{q,\mathbb R^3}
\]
which together with \eqref{hoelder-stokes} leads to
\begin{equation*}
\begin{split}
&\quad\|\nabla^j V(t+h)-\nabla^j V(t)\|_{r,\mathbb R^3}  \\
&\leq C\|U_b\| %(t-s)^{-1/2}h^\mu 
\|F\|_{q,\mathbb R^3}
\Big[(t-s)^{1/2-j/2-(3/q-3/r)/2-\mu}\,h^\mu+(t-s)^{-1/2}\,h^{1-j/2-(3/q-3/r)/2}
\Big]
\end{split}
\end{equation*}
with $\mu$ satisfying \eqref{degree-hoel}.
In view of this and \eqref{hoelder-stokes}, we conclude \eqref{hoel-est}.
\end{proof}

% 3.6
\subsection{Estimate of the pressure}
\label{sm-pressure}

In this subsection we study a smoothing rate near the initial time
of the pressure associated with $T(t,s)F$.
As in \cite{Hi20}, this must be an important issue at the final stage of the proof of decay estimates of
$T(t,s)F$ because of the non-autonomous character.
Indeed, this circumstance differs from that for the autonomous case %of analytic semigroups 
\cite{EMT} in which several advantages of analytic semigroups are used.
For our purpose, we need to discuss the domains of fractional powers of the Stokes-structure operator $A$ through
the behavior of the resolvent of the other operator $\mathbb A$.
We start with the following lemma.
\begin{lemma}
Suppose that $\mathbb A$ is the operator given by \eqref{other-A}.
Let $1<q<\infty$ and $\vartheta\in (0,1/2q)$.
Then there is a constant $C=C(q,\vartheta)>0$ such that
\begin{equation}
\|\mathbb A(\lambda+\mathbb A)^{-1}G\|_{Y_q}\leq C\lambda^{-\vartheta}\big(\|g\|_{W^{1,q}(\Omega)}+\|G\|_{Y_q}\big) %+\|\nabla g\|_{q,\Omega}\big)
\label{behavior-other}
\end{equation}
for all $\lambda\geq 1$ and $G=(g,\kappa,\mu)\in Y_q=L^q_\sigma(\Omega)\times\mathbb C^3\times\mathbb C^3$ with 
$g\in W^{1,q}(\Omega)$.
\label{beha-other}
\end{lemma}

\begin{proof}
Since we are going to discuss the behavior of the resolvent for $\lambda\to\infty$ only along the real half line,
we may fix, for instance, $\varepsilon =\pi/4$ and assume that $\lambda\geq \Lambda_{\pi/4}$.
Then we know the representation \eqref{A-series} of $(\lambda+\mathbb A)^{-1}$, which leads to
\[
\mathbb A(\lambda+\mathbb A)^{-1}G=\mathbb A(\lambda+\mathbb A_0)^{-1}(G+H(\lambda))
\]
with
\begin{equation}
\begin{split}
H(\lambda)=
\left(
\begin{array}{c}
h(\lambda) \\
\big(\alpha(\lambda),\beta(\lambda)\big)
\end{array}
\right)
&:=\sum_{j=1}^\infty \big[-\mathbb A_1(\lambda+\mathbb A_0)^{-1}\big]^j\,G  \\
&=-\mathbb A_1(\lambda+\mathbb A_0)^{-1}(G+H(\lambda))
%\sum_{j=0}^\infty\big[\cdots\big]^jG.
\end{split}
\label{H-repre}
\end{equation}
and
\begin{equation}
\|G+H(\lambda)\|_{Y_q}\leq \sum_{j=0}^\infty\|\mathbb A_1(\lambda+\mathbb A_0)^{-1}\|_{{\mathcal L}(Y_q)}^j\|G\|_{Y_q}
\leq 2\|G\|_{Y_q}.
\label{GplusH}
\end{equation}
In view of \eqref{other-A} and \eqref{A0-inv} we find
\begin{equation*}
\begin{split}
\|\mathbb A(\lambda+\mathbb A_0)^{-1}G\|_{Y_q}
&\leq \|A_\Omega w(\lambda)\|_{q,\Omega}
+\lambda^{-1}\|P_\Omega\Delta \ell(\kappa,\mu)\|_{q,\Omega}  \\
&\quad +\|P_\Omega \ell\big(K^{-1}Q_1w(\lambda)+\lambda^{-1}K^{-1}Q_2(\kappa,\mu)\big)\|_{q,\Omega}  \\
&\quad +|K^{-1}Q_1w(\lambda)|_{\mathbb C^3\times\mathbb C^3}+\lambda^{-1}|K^{-1}Q_2(\kappa,\mu)|_{\mathbb C^3\times\mathbb C^3}
\end{split}
\end{equation*}
where
\[
w(\lambda):=(\lambda+A_\Omega)^{-1}g+\lambda^{-1}(\lambda+A_\Omega)^{-1}P_\Omega\Delta \ell(\kappa,\mu).
\]
From \eqref{lift-est}, \eqref{Q1-est} and the resolvent estimate of $A_\Omega$ it follows that
\begin{equation}
\|\mathbb A(\lambda+\mathbb A_0)^{-1}G\|_{Y_q}
\leq C\|A_\Omega(\lambda+A_\Omega)^{-1}g\|_{q,\Omega}+C\lambda^{-1}\big(\|g\|_{q,\Omega}+|(\kappa,\mu)|\big).
\label{G-pre-est}
\end{equation}
We here recall the fact that, except the vanishing normal trace, the space
$D_q(A_\Omega^\vartheta)$ does not involve any boundary condition
(Noll and Saal \cite[Subsection 2.3]{NS}, %Fujiwara \cite[Section 2, Theorem 5]{Fu68}, 
Giga and Sohr \cite{GiSo89})
and, thereby,
\begin{equation*}
g\in L^q_\sigma(\Omega)\cap W^{1,q}(\Omega)\subset %D_q(A_\Omega^\theta)
L^q_\sigma(\Omega)\cap H_q^{2\vartheta}(\Omega)
=[L^q_\sigma(\Omega), D_q(A_\Omega)]_\vartheta
=D_q(A_\Omega^\vartheta) %L^q_\sigma(\Omega)\cap H_q^{2\theta}(\Omega)
\end{equation*}
provided $\vartheta\in (0,1/2q)$, where $[\cdot,\cdot]_\vartheta$ stands for the complex interpolation functor and
$H^{2\vartheta}_q(\Omega):=[L^q(\Omega), W^{2,q}(\Omega)]_\vartheta$ is the Bessel potential space.
For such $\vartheta$, we infer
\begin{equation*}
\begin{split}
\|A_\Omega(\lambda+A_\Omega)^{-1}g\|_{q,\Omega}
&=\|A_\Omega^{1-\vartheta}(\lambda+A_\Omega)^{-1}A_\Omega^\vartheta g\|_{q,\Omega}  \\
&\leq C\|A_\Omega(\lambda+A_\Omega)^{-1}A_\Omega^\vartheta g\|_{q,\Omega}^{1-\vartheta}
\|(\lambda+A_\Omega)^{-1}A_\Omega^\vartheta g\|_{q,\Omega}^\vartheta  \\
&\leq C\lambda^{-\vartheta}\|A_\Omega^\vartheta g\|_{q,\Omega}  \\
&\leq C\lambda^{-\vartheta}\|g\|_{W^{1,q}(\Omega)}
\end{split}
\end{equation*}
which combined with \eqref{G-pre-est} implies that
\begin{equation}
\|\mathbb A(\lambda+\mathbb A_0)^{-1}G\|_{Y_q}
\leq C\lambda^{-\vartheta}\big(\|g\|_{W^{1,q}(\Omega)}+\|G\|_{Y_q}\big) %+\|\nabla g\|_{q,\Omega}\big)
\label{G-est}
\end{equation}
for all $\lambda\geq \Lambda_{\pi/4}$.

It remains to show the other estimate
\begin{equation}
\|\mathbb A(\lambda+\mathbb A_0)^{-1}H(\lambda)\|_{Y_q}\leq C\lambda^{-\vartheta}\|G\|_{Y_q}
\label{H-est}
\end{equation}
in which the additional condition $g\in W^{1,q}(\Omega)$ is not needed.
In fact, exactly by the same argument as above for deduction of \eqref{G-est},
we furnish
\begin{equation}
\begin{split}
\|\mathbb A(\lambda +\mathbb A_0)^{-1}H(\lambda)\|_{Y_q}
&\leq C\lambda^{-\vartheta}\big(\|h(\lambda)\|_{W^{1,q}(\Omega)}+\|H(\lambda)\|_{Y_q}\big)  \\
&\leq C\lambda^{-\vartheta}\big(\|h(\lambda)\|_{W^{1,q}(\Omega)}+\|G\|_{Y_q}\big).
\end{split}
\label{H-pre-est}
\end{equation}
By the representation \eqref{H-repre} of $H(\lambda)$ along with
\eqref{A1}--\eqref{A0-inv} we find
\[
\|h(\lambda)\|_{W^{1,q}(\Omega)}
=\|P_\Omega \ell\big(K^{-1}Q_1 v(\lambda)+\lambda^{-1}K^{-1}Q_2(\kappa+\alpha(\lambda),\mu+\beta(\lambda))\big)\|_{W^{1,q}(\Omega)}
\]
where
\[
v(\lambda):=(\lambda+A_\Omega)^{-1}(g+h(\lambda))+\lambda^{-1}(\lambda+A_\Omega)^{-1}P_\Omega\Delta \ell(\kappa+\alpha(\lambda),\mu+\beta(\lambda)).
\]
Since $P_\Omega$ is bounded on $W^{1,q}(\Omega)$, we
use \eqref{lift-est}, \eqref{Q1-est} and \eqref{GplusH} to observe
\[
\|h(\lambda)\|_{W^{1,q}(\Omega)}
\leq %C\|A_\Omega(\lambda+A_\Omega)^{-1}(g+h(\lambda))\|_{q,\Omega}
C\|G+H(\lambda)\|_{Y_q}
%\leq C\sum_{j=0}^\infty\left\|\mathbb A_1(\lambda+\mathbb A_0)^{-1}\right\|_{{\mathcal L}(Y_q)}^j\|G\|_{Y_q}
\leq C\|G\|_{Y_q}
\]
with some $C>0$ independent of $\lambda\geq\Lambda_{\pi/4}$.
%The other term $\|H(\lambda)\|_{Y_q}$ in \eqref{H-pre-est} is discussed similarly.
In this way, we conclude \eqref{H-est}.
The proof is complete.
\end{proof}

The behavior of the resolvent with respect to $\lambda$ obtained in Lemma \ref{beha-other} is inherited 
by the Stokes-structure operator to conclude the following lemma.
\begin{lemma}
Suppose that $A$ is the operator given by \eqref{stokes-op}.
Let $1<q<\infty$ and $\vartheta\in (0,1/2q)$.
Then there is a constant $C=C(q,\vartheta)>0$ such that
\begin{equation}
\|A(\lambda+A)^{-1}F\|_{q,\mathbb R^3}\leq C\lambda^{-\vartheta}\big(\|f\|_{W^{1,q}(\Omega)}+\|F\|_{q,\mathbb R^3}\big) %+\|\nabla f\|_{q,\Omega}\big)
\label{behavior-A}
\end{equation}
for all $\lambda\geq 1$ and $F\in X_q(\mathbb R^3)$ with $f=F|_\Omega\in W^{1,q}(\Omega)$.
\label{beha-A}
\end{lemma}
\begin{proof}
Set $(f,\kappa,\mu)=i(F)$, see \eqref{X1}, and 
\[
U=(\lambda+A)^{-1}F, \qquad (u,\eta,\omega)=i(U).
\]
Following subsection \ref{stokes-structure}, we take the lifting function $U_0=\ell(\eta,\omega)$ as in \eqref{lift0}.
Then we have
\[
A(\lambda +A)^{-1}F
=F-\lambda U=(f-\lambda u)\chi_\Omega+\big\{(\kappa-\lambda \eta)+(\mu-\lambda\omega)\times x\big\}\chi_B
\]
so that
\begin{equation}
\begin{split}
\quad\|A(\lambda+A)^{-1}F\|_{q,\mathbb R^3} 
&\leq C\|f-\lambda u\|_{q,\Omega}+C|(\kappa,\mu)-\lambda(\eta,\omega)|  \\
&\leq C\|f-\ell(\kappa,\mu)-\lambda \big(u-\ell(\eta,\omega)\big)\|_{q,\Omega}
+C|(\kappa,\mu)-\lambda (\eta,\omega)|
\end{split}
\label{pre-shift}
\end{equation}
on account of \eqref{lift-est}.
By virtue of \eqref{resol2} in the other formulation, we obtain
\[
\left(
\begin{array}{c}
\widetilde u \\
(\eta,\omega)
\end{array}
\right)
=(\lambda+\mathbb A)^{-1}
\left(
\begin{array}{c}
g \\
(\kappa,\mu)
\end{array}
\right)
\]
with
\[
\widetilde u=u-\ell(\eta,\omega), \qquad
g=f-\ell(\kappa,\mu).
\]
By the assumption $f\in W^{1,q}(\Omega)$ we have $g\in W^{1,q}(\Omega)$ with
\[
\|\nabla g\|_{q,\Omega}\leq \|\nabla f\|_{q,\Omega}+C|(\kappa,\mu)|
\]
and, therefore, we know \eqref{behavior-other} for
\[
\mathbb A(\lambda+\mathbb A)^{-1}G=G-\lambda
\left(
\begin{array}{c}
\widetilde u \\
(\eta,\omega)
\end{array}
\right), \qquad
G=\left(
\begin{array}{c}
g \\
(\kappa,\mu)
\end{array}
\right).
\]
Thus, in view of \eqref{pre-shift}, we conclude \eqref{behavior-A}.
The proof is complete.
\end{proof}

In the following lemma, we note that $\mathbb P\psi$ never fulfills the boundary condition at $\partial\Omega$ which is
%$u=\eta+\omega\times x$ at $\partial\Omega$
involved in $D_r(A)$, see \eqref{stokes-op}, even if $\psi\in C_0^\infty(\Omega)\subset L^r_R(\mathbb R^3)$
(by setting zero outside $\Omega$);
thus, $\delta>0$ must be small in order that \eqref{frac-proj} holds true.
\begin{lemma}
Let $1<r<\infty$ and $\delta\in (0,1/2r)$.
Then there is a constant $C=C(r,\delta)>0$ such that $\mathbb P \psi\in D_r(A^\delta)$ subject to
\begin{equation}
\|A^\delta\mathbb P\psi\|_{r,\mathbb R^3}\leq C \big(\|\psi\|_{W^{1,r}(\Omega)}+\|\psi\|_{r,\mathbb R^3}\big)
\label{frac-proj}
\end{equation}
for all $\psi\in L^r_R(\mathbb R^3)$ with $\psi|_{\Omega}\in W^{1,r}(\Omega)$.
\label{lem-frac-proj}
\end{lemma}
\begin{proof}
%The case $\alpha=0$ is obvious.
Given $\delta\in (0,1/2r)$, we take $\vartheta \in (\delta, 1/2r)$.
Let $F\in X_r(\mathbb R^3)$ with $f=F|_\Omega\in W^{1,r}(\Omega)$, then we have \eqref{behavior-A} with such $\vartheta$,
which implies that $F\in D_r(A^\delta)$ subject to
\begin{equation}
\|A^\delta F\|_{r,\mathbb R^3}
\leq C\big(\|f\|_{W^{1,r}(\Omega)}+\|F\|_{r,\mathbb R^3}\big) %+\|\nabla f\|_{q,\Omega}\big)
\label{pre-frac-proj}
\end{equation}
by following the argument in \cite[Theorem 2.24]{Y}.
In this literature, $0\in \rho(A)$ is assumed, however, it is not needed by
considering $1+A$ instead of $A$ as follows:
In fact, by \eqref{behavior-A} with $\vartheta\in (\delta,1/2r)$ we obtain
\[
\int_1^\infty \lambda^{-1+\delta}\|(1+A)(\lambda+1+A)^{-1}F\|_{r,\mathbb R^3}\,d\lambda
\leq C\big(\|f\|_{W^{1,r}(\Omega)}+\|F\|_{r,\mathbb R^3}\big),
\]
while the integral over the interval $(0,1)$ of the same integrand always convergent.
As a consequence,
\begin{equation*}
\begin{split}
(1+A)^{-1+\delta}F
&=\frac{\sin \pi(1-\delta)}{\pi}\int_0^\infty
\lambda^{-1+\delta}(\lambda+1+A)^{-1}F\,d\lambda  \\
&=\frac{\sin \pi (1-\delta)}{\pi}(1+A)^{-1}\int_0^\infty
\lambda^{-1+\delta}(1+A)(\lambda+1+A)^{-1}F\,d\lambda
\end{split}
\end{equation*}
belongs to $D_r(A)$, yielding $F\in D_r(A^\delta)$ along with \eqref{pre-frac-proj}.

By \eqref{grad-proj-bdd} together with \eqref{proj-bdd} we know that
the projection $\mathbb P$ is bounded from
$L^r_R(\mathbb R^3)\cap W^{1,r}(\Omega)$ to $X_r(\mathbb R^3)\cap W^{1,r}(\Omega)$. % for every integer $k\geq 0$.
Thus, \eqref{pre-frac-proj} implies \eqref{frac-proj}.
\end{proof}

We provide estimate of the pressure near the initial time, that would be of independent interest.
\begin{proposition}
Suppose \eqref{ass-1} and \eqref{ass-2}.
Let $1<q<\infty$ and $\tau_*,\, \alpha_0,\, \beta_0\in (0,\infty)$.
Then,
for every $\gamma\in\big((1+1/q)/2,\, 1\big)$,
there is a constant $C=C(\gamma,q,\tau_*,\alpha_0,\beta_0,\theta)>0$ such that
\begin{equation}
\|p(t)\|_{q,\Omega_3}+\|\partial_tT(t,s)F\|_{W^{-1,q}(\Omega_3)}\leq C(t-s)^{-\gamma}\|F\|_{q,\mathbb R^3}
\label{sm-est-pressure}
\end{equation}
for all $(t,s)$ with $t-s\in (0,\tau_*]$ and $F\in X_q(\mathbb R^3)$ whenever
$\|U_b\|\leq\alpha_0$ and $[U_b]_\theta\leq\beta_0$,
where $p(t)$ denotes the pressure associated with $T(t,s)F$ and it is singled out in such a way that
$\int_{\Omega_3}p(t)\,dx=0$ for each $t\in (s,\infty)$, while $\|U_b\|$ and $[U_b]_\theta$ are given by \eqref{quan}.
\label{prop-pressure}
\end{proposition}
\begin{proof}
Set $U(t)=T(t,s)F$ and $u(t)=U(t)|_\Omega$.
We follow the approach developed by Noll and Saal \cite[Lemma 13]{NS}.
Given $\phi\in C_0^\infty(\Omega_3)$, we take $\psi:=\mathbb B[\phi-\overline\phi]\in W^{1,q^\prime}_0(\Omega_3)$
that solves
\begin{equation}
\mbox{div $\psi$}=\phi-\overline\phi \quad\mbox{in $\Omega_3$}, \qquad \psi=0\quad \mbox{on $\partial\Omega_3$},
\label{bvp-div}
\end{equation}
where 
$\mathbb B$ denotes the Bogovskii operator for the domain $\Omega_3$ and 
$\overline\phi:=\frac{1}{|\Omega_3|}\int_{\Omega_3}\phi(x)\,dx$.
The boundary value problem \eqref{bvp-div} 
(in which $\phi-\overline{\phi}$ is replaced by $g$ with the compatibility condition)
admits many solutions, among which the Bogovskii operator
\begin{equation}
\mathbb B: W_0^{k,r}(\Omega_3)\to W_0^{k+1,r}(\Omega_3)^3, \qquad
r\in (1,\infty),\; k=0,1,2,\cdots 
\label{bog-op}
\end{equation}
specifies a particular solution discovered by Bogovskii \cite{Bog} with fine regularity properties
\begin{equation}
\|\mathbb \nabla^{k+1}\mathbb Bg\|_{r,\Omega_3}
\leq C\|\nabla^k g\|_{r,\Omega_3},  \qquad
\|\mathbb Bg\|_{r,\Omega_3}
\leq C\|g\|_{W^{1,r^\prime}(\Omega_3)^*}
\label{bog-est}
\end{equation}
with some $C>0$ (dependent on the same $k,\, r$ as above)
as well as
$\mbox{div $\mathbb Bg$}=g$ provided $\int_{\Omega_3}g\,dx=0$, see \cite{Bor-So, Ga-b, GHH} for the details.
Note that $\mathbb Bg\in C_0^\infty(\Omega_3)^3$ for every $g\in C_0^\infty(\Omega_3)$ and 
that the right-hand side of the latter estimate of \eqref{bog-est} cannot be replaced by
$\|g\|_{W^{-1,r}(\Omega_3)}$.

We may understand $\psi\in W^{1,q^\prime}(\mathbb R^3)\cap L^{q^\prime}_R(\mathbb R^3)$ 
by setting zero outside $\Omega_3$ and use \eqref{bog-est} to get
\begin{equation}
\|\psi\|_{W^{1,q^\prime}(\mathbb R^3)}=\|\psi\|_{W^{1,q^\prime}_0(\Omega_3)}
\leq C\|\phi-\overline\phi\|_{q^\prime,\Omega_3}
\leq C\|\phi\|_{q^\prime,\Omega_3}.
\label{bog-auxi}
\end{equation}
By virtue of $\int_{\Omega_3}p(t)\,dx=0$, %and $\psi=0$ outside $\Omega_3$, 
we obtain
\begin{equation*}
\begin{split}
\langle p(t),\phi\rangle_{\Omega_3}
&=\langle p(t),\phi-\overline\phi\rangle_{\Omega_3}
=\langle p(t),\mbox{div $\psi$}\rangle_{\Omega_3}
=-\langle \nabla p(t),\psi\rangle_{\Omega_3}  \\
&=\langle\partial_tu-\Delta u+(u_b-\eta_b)\cdot\nabla u, \psi\rangle_{\Omega_3}  
\end{split}
\end{equation*}
for all $\phi\in C_0^\infty(\Omega_3)$.
Since $\psi=0$ outside $\Omega_3$, we have
\[
\langle \partial_tu,\psi\rangle_{\Omega_3}=\langle\partial_tU, \psi\rangle_{\mathbb R^3,\rho}
=-\langle AU+B(t)U, \psi\rangle_{\mathbb R^3,\rho}
\]
and, thereby,
\begin{equation}
\begin{split}
\langle p(t), \phi\rangle_{\Omega_3}
&=-\langle AU,\psi\rangle_{\mathbb R^3,\rho}-\langle\Delta u,\psi\rangle_\Omega
+\langle (1-\mathbb P)\big[\{(u_b-\eta_b)\cdot\nabla u\}\chi_\Omega\big], \psi\rangle_{\mathbb R^3,\rho}  \\
&=:I+II+III
\end{split}
\label{press-dual}
\end{equation}
for all $\phi\in C_0^\infty(\Omega_3)$.
Note that the pairing over $\mathbb R^3$ should involve the constant weight $\rho$, otherwise one can use neither \eqref{op-sym}
nor \eqref{proj-sym-0}.
It follows from Proposition \ref{prop-smooth}, \eqref{ass-1} and \eqref{bog-auxi} that
\begin{equation}
|III|\leq C\|U_b\|\|\nabla u\|_{q,\Omega}\|\psi\|_{q^\prime,\mathbb R^3}
\leq C(t-s)^{-1/2}\|U_b\|\|F\|_{q,\mathbb R^3}\|\phi\|_{q^\prime,\Omega_3}
\label{press-3}
\end{equation}
and that
\begin{equation}
|II|\leq \|\nabla u\|_{q,\Omega}\|\nabla \psi\|_{q^\prime,\Omega}
\leq C(t-s)^{-1/2}\|F\|_{q,\mathbb R^3}\|\phi\|_{q^\prime,\Omega_3}
\label{press-2}
\end{equation}
for all $(t,s)$ with $t-s\leq\tau_*$ and $F\in X_q(\mathbb R^3)$.

By virtue of the momentum inequality for fractional powers with $\gamma\in (0,1)$ together with
\eqref{equi-LA} and \eqref{tanabe-est}--\eqref{tanabe-0} we find
\begin{equation}
\|A^\gamma U\|_{q,\mathbb R^3}
\leq C\|AU\|_{q,\mathbb R^3}^\gamma \|U\|_{q,\mathbb R^3}^{1-\gamma}
\leq C(t-s)^{-\gamma}\|F\|_{q,\mathbb R^3}
\label{press-1-left}
\end{equation}
for all $(t,s)$ with $t-s\leq\tau_*$ and $F\in X_q(\mathbb R^3)$.
If in particular $\gamma\in\big((1+1/q)/2,\, 1\big)$, then
Lemma \ref{lem-frac-proj} with $r=q^\prime$, $\delta=1-\gamma\in (0,1/2q^\prime)$ and \eqref{bog-auxi}
for $\psi\in W^{1,q^\prime}(\mathbb R^3)\cap L^{q^\prime}_R(\mathbb R^3)$ imply
\[
\|A^{1-\gamma}\mathbb P\psi\|_{q^\prime,\mathbb R^3}
\leq C\big(\|\psi\|_{W^{1,q^\prime}(\Omega)}+\|\psi\|_{q^\prime,\mathbb R^3}\big)
\leq C\|\phi\|_{q^\prime,\Omega_3}
\]
for all $\phi\in C_0^\infty(\Omega_3)$,
which combined with \eqref{press-1-left} yields
\begin{equation}
|I|=|\langle A^\gamma U, A^{1-\gamma}\mathbb P \psi\rangle_{\mathbb R^3,\rho}|
\leq C(t-s)^{-\gamma}\|F\|_{q,\mathbb R^3}\|\phi\|_{q^\prime,\Omega_3}
\label{press-1}
\end{equation}
for all $(t,s)$ with $t-s\leq \tau_*$ and $F\in X_q(\mathbb R^3)$ by taking into account \eqref{op-sym}
and \eqref{proj-sym-0}.
%{\color{blue} ($A_q^*=A_{q^\prime}$)}.
We collect \eqref{press-dual}--\eqref{press-2} and \eqref{press-1} to conclude \eqref{sm-est-pressure} for the pressure.

It remains to show \eqref{sm-est-pressure} for $\partial_tU(t)|_{\Omega_3}=\partial_tu(t)|_{\Omega_3}$, that readily follows from
\begin{equation*}
\begin{split}
|\langle \partial_tu, \Phi\rangle_{\Omega_3}|
&=|\langle \Delta u+(\eta_b-u_b)\cdot\nabla u-\nabla p, \Phi\rangle_{\Omega_3}|   \\
&\leq C\big(\|\nabla u\|_{q,\Omega_3}+\|U_b\|\|u\|_{q,\Omega_3}+\|p\|_{q,\Omega_3}\big)\|\nabla\Phi\|_{q^\prime,\Omega_3}
\end{split}
\end{equation*}
for all $\Phi\in C_0^\infty(\Omega_3)^3$
together with estimate for the pressure obtained above and Proposition \ref{prop-smooth}.
The proof is complete.
\end{proof}

\begin{remark}
Tolksdorf \cite[Proposition 3.4]{Tol} has made it clear that the smoothing rate %\eqref{sm-est-pressure} with
$(t-s)^{-3/4}$ of the pressure is sharp within the $L^2$ theory of the Stokes semigroup in bounded domains % in $L^2$
through the behavior of the resolvent for $|\lambda|\to\infty$.
This suggests that \eqref{sm-est-pressure} %Proposition \ref{prop-pressure}
would be almost optimal.
The same rate as in \eqref{sm-est-pressure} was discovered first by Noll and Saal \cite{NS} 
%for the Stokes operator in exterior domains 
and it was slightly improved as
$\gamma=(1+1/q)/2$ by \cite{HiShi09, Hi20} for the exterior problem with prescribed rigid motions. 
\label{rem-tolks}
\end{remark}

% 3.7
\subsection{Adjoint evolution operator and backward problem}
\label{adjoint-backward}

Let $t\in\mathbb R$ be a parameter as the final time of the problem below.
For better understanding of the initial value problem \eqref{eq-linear}
for the linearized system,
it is useful to study the backward problem for the adjoint system subject to the final conditions at $t$:
\begin{equation}
\begin{split}
&-\partial_sv=\Delta v-(\eta_b(s)-u_b(s))\cdot\nabla v+\nabla p_v, \qquad %y\in\Omega,\; s\in (-\infty,t), \\
\mbox{div $v$}=0 \quad \mbox{in $\Omega\times (-\infty,t)$}, \\ %y\in\Omega \\
&v|_{\partial\Omega}=\eta+\omega\times y, \qquad %y\in\partial\Omega \\
v\to 0 \quad \mbox{as $|y|\to\infty$}, \\
&-m\,\frac{d\eta}{ds}+\int_{\partial\Omega}\mathbb S(v,-p_v)\nu\,d\sigma =0, \\
&-J\,\frac{d\omega}{ds}+\int_{\partial\Omega}y\times \mathbb S(v,-p_v)\nu\,d\sigma=0, \\
&v(\cdot,t)=v_0, \quad \eta(t)=\eta_0, \quad \omega(t)=\omega_0,
\end{split}
\label{backward}
\end{equation}
where $v(y,s)$, $p_v(y,s)$, $\eta(s)$ and $\omega(s)$ are unknowns functions.
By using the Oseen-structure operator \eqref{oseen-op} together with the description \eqref{X1}--\eqref{X2}, the backward
problem \eqref{backward} is formulated as
\begin{equation}
-\frac{dV}{ds}+L_-(s)V=0, \quad s\in (-\infty,t);\qquad V(t)=V_0
\label{back-adj}
\end{equation}
for the monolithic velocity
$V=v\chi_\Omega+(\eta+\omega\times y)\chi_B$,
where $V_0=v_0\chi_\Omega+(\eta_0+\omega_0\times y)\chi_B$.
Conversely, once we have a solution to \eqref{back-adj},
there exists a pressure $p_v$ which together with $(v,\eta,\omega)$ solves \eqref{backward} as in subsection \ref{resol-equi}.

We begin with justification of the duality relation between the operators $L_\pm(t)$ within $X_q(\mathbb R^3)^*=X_{q^\prime}(\mathbb R^3)$,
see \eqref{dual-an}, and the dissipative structure of 
each of those operators.
From subsection \ref{generation} we know that $k+L_\pm(t)$ is bijective for large $k>0$, which combined with
\eqref{dual} below implies that
$L_\pm(t)^*=L_\mp(t)$. 
The latter property \eqref{dissi} also plays a crucial role in section \ref{est-evo}. %see Remark \ref{rem-perturb}.
\begin{lemma}
Suppose \eqref{ass-1} and \eqref{ass-2}.
Let $1<q<\infty$, % and $1/q^\prime+1/q=1$, 
then
\begin{equation}
\langle L_\pm(t)U,V\rangle_{\mathbb R^3,\rho}=\langle U,L_\mp(t)V\rangle_{\mathbb R^3,\rho}
\label{dual}
\end{equation}
for all $U\in D_q(A)$, $V\in D_{q^\prime}(A)$ and $t\in\mathbb R$, where $\langle\cdot,\cdot\rangle_{\mathbb R^3,\rho}$
is given by \eqref{pairing}.
Moreover, we have
\begin{equation}
\langle L_\pm(t)U, U\rangle_{\mathbb R^3,\rho}=2\|Du\|_{2,\Omega}^2
\label{dissi}
\end{equation}
for all $U\in D_2(A)$ and $t\in\mathbb R$, where
$u=U|_\Omega$.
\label{dual-dissi}
\end{lemma}
\begin{proof}
The following fine properties of the Stokes-structure operator $A$ is well-known, see Takahashi and Tucsnak \cite[Section 4]{TT04}:
\begin{equation}
\langle AU,V\rangle_{\mathbb R^3,\rho}=\langle U,AV\rangle_{\mathbb R^3,\rho} =2\langle Du, Dv\rangle_\Omega
\label{stokes-sym}
\end{equation}
for all $U\in D_q(A)$ and $V\in D_{q^\prime}(A)$,
where $u=U|_\Omega,\, v=V|_\Omega$; in particular,
\begin{equation}
\langle AU,U\rangle_{\mathbb R^3,\rho}=2\|Du\|_{2,\Omega}^2
\label{stokes-dissi}
\end{equation}
for all $U\in D_2(A)$.
Computation of \eqref{stokes-dissi} is already done in the latter half of 
the proof of Proposition \ref{two-st-str} of the present paper as well.
We stress that the constant weight $\rho >0$ is needed in order that
\eqref{stokes-sym} and \eqref{stokes-dissi} hold true.
It thus suffices to show that
\begin{equation}
\langle B(t)U,V\rangle_{\mathbb R^3,\rho}+\langle U,B(t)V\rangle_{\mathbb R^3,\rho}=0
\label{B-skew}
\end{equation}
for all $U$ and $V$ as above.
We see from \eqref{proj-sym-0} that
\begin{equation}
\langle B(t)U,V\rangle_{\mathbb R^3,\rho}
=\langle (u_b(t)-\eta_b(t))\cdot\nabla u, v\rangle_\Omega \qquad
\label{B-om}
\end{equation}
and the same relation for $\langle U,B(t)V\rangle_{\mathbb R^3,\rho}$.
Let us use the same cut-off function $\phi_R$ as in \eqref{para-cut}.
Since
\begin{equation}
\nu\cdot(u_b(t)-\eta_b(t))=x\cdot(\omega_b(t)\times x)=0, \qquad x\in \partial\Omega %=\partial B
\label{crucial}
\end{equation}
by \eqref{ass-1}, we find
\[
\int_{\Omega}\mbox{div $\big[(u\cdot v)(u_b(t)-\eta_b(t))\phi_R\big]$}\,dx
=\int_{\partial\Omega}(u\cdot v)\nu\cdot (u_b(t)-\eta_b(t))\,d\sigma=0
\]
for all $u=U|_\Omega$ and $v=V|_\Omega$.
By \eqref{ass-1} again it is readily seen that
\[
\lim_{R\to \infty}\int_{2R<|x|<3R}|u||v||(u_b(t)-\eta_b(t))\cdot\nabla \phi_R|\,dx= 0,
\]
which implies
\[
\int_\Omega (u_b(t)-\eta_b(t))\cdot\nabla (u\cdot v)\,dx=0.
%\mbox{div $\big[(u\cdot v)(\eta_b(t)-u_b(t))\big]$}\,dx=0.
\]
This combined with \eqref{B-om}
concludes \eqref{B-skew}, which %combined with \eqref{B-om} 
together with \eqref{stokes-sym} leads us to \eqref{dual}.
The relation \eqref{dissi} follows from \eqref{stokes-dissi} and \eqref{B-skew} with $V=U\in D_2(A)$.
The proof is complete.
\end{proof}

Several remarks are in order.
\begin{remark}
It should be emphasized that Lemma \ref{dual-dissi} is not accomplished if the drift term in %\eqref{oseen-term}
\eqref{eq-linear}/\eqref{backward} is replaced by the purely Oseen term $\eta_b\cdot\nabla u$, see \eqref{crucial}.
This is why the other drift term $u_b\cdot\nabla u$ must be additionally involved into the right linearization.
If the shape of the body is arbitrary, we see from \eqref{shape-arbit} that the corresponding drift term is given by
$(u_b-\eta_b-\omega_b\times x)\cdot\nabla u$, so that the boundary integral from this term 
vanishes as in the proof of Lemma \ref{dual-dissi}.
%Moreover, we have some knowledge about this drift term, see \cite{Hi18, Hi20}.
\label{rem-drift}
\end{remark}
\begin{remark}
If we wish to involve the term $(u-\eta)\cdot\nabla u_b$ as well into the linearization, we employ
\[
|\langle (u-\eta)\cdot\nabla u_b(t), u\rangle_\Omega|
\leq C\left(\|u_b(t)\|_{L^{3,\infty}(\Omega)}+\|u_b(t)\|_{2,\Omega}\right)\|\nabla u\|_{2,\Omega}^2
\]
with the aid of
$|\eta|\leq C\|Du\|_{2,\Omega}$, see %\eqref{eta-trace} and Remark \ref{rigid-control} below, 
for instance Galdi \cite[Lemma 4.9]{Ga02},
where $L^{3,\infty}(\Omega)$ stands for the weak-$L^3$ space.
Hence the smallness of $u_b$ in $L^2(\Omega)$, %$\|u_b\|_{2,\Omega}$, 
that is more restrictive (from the viewpoint of summability at infinity) than \eqref{ass-1},
is needed to ensure the desired energy relation in subsection \ref{energy-relation}.
Indeed this is possible under the self-propelling condition, see subsection \ref{basic-mo}, but it does not follow %that is not the case 
solely from the wake structure.
Since we want to develop the theory under less assumption \eqref{ass-1} on the basic motion, %$u_b\in L^q(\Omega)$
it is better not to put the term $(u-\eta)\cdot\nabla u_b$ in the linearization but to
treat this term together with the nonlinear term. % under less assumptions on the basic motion.
\label{rem-perturb}
\end{remark}
\begin{remark}
If the shape of the body is arbitrary, the corresponding term to the one mentioned in Remark \ref{rem-perturb}
is given by 
$(u-\eta-\omega\times x)\cdot\nabla u_b$ in view of \eqref{shape-arbit}.
In order that the desired energy relation is available even if this term is involved into the linearization,
we have to ask it to be $|x|u_b\in L^2(\Omega)$, which is too restrictive.
%   Moreover, it is difficult to treat the term $(\omega\times x)\cdot u_b$ as perturbation from the linearization.
In fact, $(\omega\times x)\cdot\nabla u_b$ is the worst term among all linear terms and thus
we do need the specific shape already in linear analysis unlike \cite{EMT}.
\label{rem-arbit}
\end{remark}
%
%\begin{remark}
%Even though $\nu\cdot u_*|_{\partial\Omega}\neq 0$, it would be possible to analyze the stability for small $\nu\cdot u_*$.
%In this case the equations for the rigid body in \eqref{eq-fixed} are slightly different and, therefore, the operator $B(t)$ 
%given by \eqref{oseen-term} consists of more terms and it is no longer skew-symmetric, nevertheless, can be controlled by the dissipation as long as $\nu\cdot u_*$ is small enough in a sense.
%This issue will be discussed elsewhere. 
%\label{nonzero-flux}
%\end{remark}

Let us consider the auxiliary initial value problem
\begin{equation}
\frac{dW}{d\tau}+L_-(t-\tau)W=0, \quad \tau\in (s,\infty);\qquad W(s)=V_0,
\label{auxi-eq}
\end{equation}
where $t\in\mathbb R$ is a parameter involved in the coefficient operator.
By exactly the same argument as in subsections \ref{generation} and \ref{smoothing-estimate}, we see that
the operator family $\{L_-(t-\tau);\,\tau\in\mathbb R\}$ generates an evolution operator 
$\{\widetilde T(\tau,s;\,t);\, -\infty<s\leq \tau<\infty\}$ on $X_q(\mathbb R^3)$ for every $q\in (1,\infty)$
and that it %is parabolic in the sense of Tanabe \cite{Ta} and 
satisfies the similar %corresponding 
smoothing estimates to \eqref{decay-1}--\eqref{grad-new-form},
in which the constants $C=C(\tau_*)$ are taken uniformly in $(\tau,s)$ with
$\tau-s\in (0,\tau_*]$ for given $\tau_*>0$ and
do not depend on $t\in\mathbb R$.
This implies %Lemma \ref{adj-smooth} 
\eqref{adj-sm} below on the evolution operator
defined by
\begin{equation}
S(t,s):=\widetilde T(t-s,0;\,t),\qquad -\infty<s\leq t<\infty,
\label{adj-def}
\end{equation}
which coincides with the adjoint of $T(t,s)$, see \eqref{dual-evo} in the next lemma.

For every $V_0\in X_q(\mathbb R^3)$, the function $V(s)=S(t,s)V_0$ is a solution to the backward problem \eqref{back-adj}, that is,
\begin{equation}
-\partial_sS(t,s)+L_-(s)S(t,s)=0, \quad s\in (-\infty,t); \qquad S(t,t)={\mathcal I}
\label{adj-evo1}
\end{equation}
in ${\mathcal L}(X_q(\mathbb R^3))$.
As we would expect, the following duality relation \eqref{dual-evo} holds true.
Note that the latter assertion \eqref{back-semi} %--\eqref{adj-evo2} do 
does not follow directly from \eqref{adj-def}.
\begin{lemma}
Let $1<q<\infty$, then
\begin{equation}
T(t,s)^*=S(t,s), \qquad S(t,s)^*=T(t,s) \qquad
\mbox{in ${\mathcal L}(X_q(\mathbb R^3))$}
\label{dual-evo}
\end{equation}
%in ${\mathcal L}(X_q(\mathbb R^3))$ 
for all $(t,s)$ with $-\infty<s\leq t<\infty$ in the sense of 
\eqref{dual-evo-sense} below.
Moreover, we have the backward semigroup property
\begin{equation}
S(\tau,s)S(t,\tau)=S(t,s)\qquad (-\infty<s\leq\tau\leq t<\infty)
\label{back-semi}
\end{equation}
in ${\mathcal L}(X_q(\mathbb R^3))$.
\label{lem-dual-evo}
\end{lemma}
\begin{proof}
%Let $F\in X_q(\mathbb R^3)$, $G\in X_{q^\prime}(\mathbb R^3)$ and 
We fix $(t,s)$ with $-\infty<s<t<\infty$.
For every $\tau\in (s,t)$, it follows from \eqref{evo-eq1}, \eqref{dual} and \eqref{adj-evo1} that
\begin{equation*}
\begin{split}
&\partial_\tau\langle T(\tau,s)F, S(t,\tau)G\rangle_{\mathbb R^3,\rho}  \\
&=\langle -L_+(\tau)T(\tau,s)F, S(t,\tau)G\rangle_{\mathbb R^3,\rho}
+\langle T(\tau,s)F, L_-(\tau)S(t,\tau)G\rangle_{\mathbb R^3,\rho}=0
\end{split}
\end{equation*}
for all $F\in X_q(\mathbb R^3)$ and $G\in X_{q^\prime}(\mathbb R^3)$, which implies
\begin{equation}
\langle T(t,s)F, G\rangle_{\mathbb R^3,\rho}
=\langle F, S(t,s)G\rangle_{\mathbb R^3,\rho}
\label{dual-evo-sense}
\end{equation}
yielding \eqref{dual-evo}.
Then the semigroup property \eqref{semi} of $T(t,s)$ leads to \eqref{back-semi},
which completes the proof.
%We also obtain \eqref{adj-evo2} from \eqref{dual} and \eqref{dual-evo}.
\end{proof}

In the following proposition we provide the estimate of the associated pressure as well as 
\begin{equation}
\|\nabla^jT(t,s)^*G\|_{r,\mathbb R^3}\leq C(t-s)^{-j/2-(3/q-3/r)/2}
\|G\|_{q,\mathbb R^3}
\label{adj-sm}
\end{equation}
near $s=t$.
Indeed \eqref{adj-sm} with $j=0$ and $q,\,r\in (1,\infty)$ follows from the duality, but the other cases are obtained via
estimates of $\widetilde T(\tau,s;t)$, see \eqref{adj-def}.
The proof of the estimates of the pressure and $\partial_sT(t,s)^*$ are exactly the same as in Proposition \ref{prop-pressure}
and may be omitted.
\begin{proposition}
Suppose \eqref{ass-1} and \eqref{ass-2}.
Let
\[ 
\begin{array}{ll}
1<q<\infty, \;\; q\leq r\leq \infty \qquad&\mbox{for $j=0$}, \\ 
1<q\leq r<\infty &\mbox{for $j=1$}. 
\end{array}   
\]
Given $\tau_*,\, \alpha_0,\, \beta_0\in (0,\infty)$, estimate
%$C=C(j,q,r,\tau_*,\alpha_0,\beta_0,\theta)>0$
\eqref{adj-sm} holds with some constant
$C=C(j,q,r,\tau_*,\alpha_0,\beta_0,\theta)>0$
for all $(t,s)$ with $t-s\in (0,\tau_*]$ and $G\in X_q(\mathbb R^3)$
whenever $\|U_b\|\leq\alpha_0$ and $[U_b]_\theta\leq\beta_0$, where $\|U_b\|$ 
and $[U_b]_\theta$ are given by \eqref{quan}.

Moreover, for every $\gamma\in\big((1+1/q)/2,\, 1\big)$, there is a constant 
$C=C(\gamma,q,\tau_*,\alpha_0,\beta_0,\theta)>0$ such that
\begin{equation}
\|p_v(s)\|_{q,\Omega_3}+\|\partial_sT(t,s)^*G\|_{W^{-1,q}(\Omega_3)}
\leq C(t-s)^{-\gamma}\|G\|_{q,\mathbb R^3}
\label{sm-adj-pressure}
\end{equation}
for all $(t,s)$ with $t-s\in (0,\tau_*]$ and $G\in L^q(\mathbb R^3)$ whenever $\|U_b\|\leq \alpha_0$ and
$[U_b]_\theta\leq\beta_0$, where $p_v(s)$ denotes the pressure associated with $v(s)=T(t,s)^*G$ and it is singled out
in such a way that $\int_{\Omega_3}p_v(s)\,dy=0$ for each $s\in (-\infty,t)$.
\label{adj-smooth}
\end{proposition}

%% new 4
\section{Decay estimates of the evolution operator}
\label{est-evo}

In this section we study the large time behavior of the Oseen-structure evolution operator.
Our argument is based on the following two ingredients:
One is the decay estimates of the related evolution operator without the rigid body, the other is a consequence from
the energy relation.
The former is discussed in subsection \ref{evo-whole}, whereas the latter is deduced in subsection \ref{energy-relation}.
Then the proof of \eqref{decay-1} (except for the $L^\infty$-estimate) is given in subsection \ref{decay-0th}.
Toward the gradient estimate (as well as the $L^\infty$-estimate), the local energy decay property is established in subsection \ref{led-est},
and successively the large time behavior near spatial infinity is studied in subsection \ref{near-infinity}.
The final subsection is devoted to completion of the linear theory.

% 4.1
\subsection{Evolution operator %in the whole space 
without rigid body}
\label{evo-whole}

In this subsection we consider the initial value problem for the fluid equation relating to
\eqref{eq-linear} in the whole space without the rigid body, that is,
\begin{equation}
\begin{split}
&\partial_tu=\Delta u+(\eta_b(t)-U_b(t))\cdot\nabla u-\nabla p, \quad %x\in\mathbb R^3,\, t\in (s,\infty) \\
\mbox{div $u$}=0, \quad (x,t)\in\mathbb R^3\times (s,\infty), \\
&u\to 0\quad\mbox{as $|x|\to\infty$}, \\
&u(\cdot,s)=\psi, 
\end{split}
\label{eq-whole}
\end{equation}
and the associated backward problem for the adjoint system
\begin{equation}
\begin{split}
&-\partial_sv=\Delta v-(\eta_b(s)-U_b(s))\cdot\nabla v+\nabla p_v, \quad %y\in\mathbb R^3,\; s\in (-\infty,t) \\
\mbox{div $v$}=0, \quad (y,s)\in\mathbb R^3\times (-\infty,t), \\
&v\to 0\quad \mbox{as $|y|\to\infty$}, \\
&v(\cdot,t)=\phi,
\end{split}
\label{adj-whole}
\end{equation}
where the initial and final velocities are taken from the space
$L^q_\sigma(\mathbb R^3)$ %=X_q(\mathbb R^3)\oplus G_q^{(2)}(\mathbb R^3)
with some $q\in (1,\infty)$, see \eqref{helm}, while
$U_b$ and $\eta_b$ are as in \eqref{background}--\eqref{ass-2}.
The family $\{L_{0,\pm}(t);\, t\in\mathbb R\}$ of the modified Oseen operators on $L^q_\sigma(\mathbb R^3)$ %, $1<q<\infty$,
is given by
\begin{equation}
\begin{split}
&D_q(L_{0,\pm}(t))=L^q_\sigma(\mathbb R^3)\cap W^{2,q}(\mathbb R^3), \\ %\qquad
&L_{0,\pm}(t)u=-\mathbb P_0\big[\Delta u\pm\big(\eta_b(t)-U_b(t)\big)\cdot\nabla u\big].
\end{split}
\label{oseen-wh}
\end{equation}
where $\mathbb P_0$ denotes the classical Fujita-Kato projection, see \eqref{FK-proj}. 
Then the same arguments as in subsections \ref{generation} and \ref{adjoint-backward} show 
that the operator families $\{L_{0,+}(t);\, t\in\mathbb R\}$ and $\{L_{0,-}(t-\tau);\, \tau\in\mathbb R\}$
generate parabolic evolution operators $\{T_0(t,s);\, -\infty<s\leq t<\infty\}$
and $\{\widetilde T_0(\tau,s;\,t);\, -\infty<s\leq\tau<\infty\}$, respectively,
on $L^q_\sigma(\mathbb R^3)$
for every $q\in (1,\infty)$ and
that the duality relation between the operators $L_{0,\pm}(t)$ as in Lemma \ref{dual-dissi} implies that
$T_0(t,s)^*=\widetilde T_0(t-s,0;\,t)$ is the solution operator to the backward problem \eqref{adj-whole}.

Let $q\in (1,\infty)$ and $\tau_*,\, \alpha_0,\, \beta_0\in (0,\infty)$.
On account of the globally H\"older condition \eqref{ass-2}, %given $\tau_*,\, \alpha_0,\, \beta_0\in (0,\infty)$,
there is a constant $C=C(q,\tau_*,\alpha_0,\beta_0,\theta)>0$ such that
\begin{equation}
\|\nabla T_0(t,s)\psi\|_{q,\mathbb R^3}\leq C(t-s)^{-1/2}\|\psi\|_{q,\mathbb R^3}
\label{grad-smooth-wh}
\end{equation}
\begin{equation}
\|\nabla T_0(t,s)^*\phi\|_{q,\mathbb R^3}\leq C(t-s)^{-1/2}\|\phi\|_{q,\mathbb R^3}
\label{grad-sm-adj-wh}
\end{equation}
\begin{equation}
\|T_0(t,s)^*\mathbb P_0\mbox{div $\Phi$}\|_{q,\mathbb R^3}\leq C(t-s)^{-1/2}
\|\mathbb P_0\|_{{\mathcal L}(L^q(\mathbb R^3))}\|\Phi\|_{q,\mathbb R^3}
\label{sm-dual-wh}
\end{equation}
for all $(t,s)$ with $t-s\in (0,\tau_*]$, $\psi,\,\phi\in L^q_\sigma(\mathbb R^3)$ and $\Phi\in L^q(\mathbb R^3)^{3\times 3}$ with
$\mbox{div $\Phi$}\in \cup_{1<\sigma<\infty}L^\sigma(\mathbb R^3)$
as long as $\|U_b\|\leq \alpha_0$ and $[U_b]_\theta\leq \beta_0$, see \eqref{quan}, 
by the same argument as in Proposition \ref{prop-smooth} and by duality as to \eqref{sm-dual-wh}. %obtaining \eqref{grad-smoothing}.
Our task here is to deduce the large time behavior of $T_0(t,s)$ and $T_0(t,s)^*$. %to \eqref{eq-whole}
%when $\|U_b\|$ is small enough.  
This is the stage in which the smallness of $\|U_b\|$ as well as the condition $q_0<3$ in \eqref{ass-1} is needed.

The idea is to regard the problem \eqref{eq-whole} as perturbation 
from the Oseen evolution operator
\begin{equation}
\big(E(t,s)f\big)(x)
%=\big(e^{t\Delta}f\big)\left(x+\int_s^t\eta_b(\sigma)\,d\sigma\right)
=\int_{\mathbb R^3}G\left(x-y+\int_s^t\eta_b(\sigma)\,d\sigma,\; t-s\right)f(y)\,dy
\label{os-repre0}
\end{equation}
which solves the non-autonomous Oseen initial problem
\begin{equation}  
\begin{split}  
&\partial_tu=\Delta u+\eta_b(t)\cdot\nabla u-\nabla p, \quad %x\in\mathbb R^3,\; t\in (s,\infty), \\ 
\mbox{div $u$}=0, \quad (x,t) \in\mathbb R^3\times (s,\infty), \\ 
&u \to 0 \quad\mbox{as $|x|\to\infty$}, \\
&u(\cdot,s)=f, 
\end{split}
\label{nonauto-oseen}
\end{equation}
provided that $f$ is a solenoidal vector field, where
\[
G(x,t)=(4\pi t)^{-3/2}e^{-|x|^2/4t}
\]
so that
\[
e^{t\Delta}f=G(t)*f
\]
with $*$ being the convolution denotes the heat semigroup.
Likewise, the problem \eqref{adj-whole} is viewed as perturbation from
\begin{equation}
\big(E(t,s)^*g\big)(y)
=\int_{\mathbb R^3}G\left(x-y+\int_s^t\eta_b(\sigma)\,d\sigma,\; t-s\right)g(x)\,dx.
\label{os-repre-adj}
\end{equation}
%It is obvious that both $E(t,s)$ and $\widetilde E(\tau,s;\,t)$

The heat semigroup 
enjoys the $L^q$-$L^r$ estimates ($q\leq r$) %those for the heat semigroup
\begin{equation}
%\begin{split}
\|e^{t\Delta}f\|_{r,\mathbb R^3}
\leq (4\pi t)^{-(3/q-3/r)/2}\|f\|_{q,\mathbb R^3}
\label{heat}
\end{equation}
\begin{equation}
%\leq t^{-(3/q-3/r)/2}\|f\|_{q,\mathbb R^3}, \\
\|\nabla e^{t\Delta}f\|_{r,\mathbb R^3}
\leq 4(2\pi t)^{-1/2-(3/q-3/r)/2}\|f\|_{q,\mathbb R^3}
%&\leq \sqrt\frac{8}{\pi}\,t^{-1/2-(3/q-3/r)/2}\|f\|_{q,\mathbb R^3},
\label{heat-grad}
\end{equation}
\begin{equation}
\|e^{t\Delta}\mathbb P_0\mbox{div $F$}\|_{r,\mathbb R^3}
\leq 4(2\pi t)^{-1/2-(3/q-3/r)/2}\|\mathbb P_0\|_{{\mathcal L}(L^r(\mathbb R^3))}\|F\|_{q,\mathbb R^3}
\label{heat-dual}
\end{equation}
%and the same thing holds for $\widetilde E(\tau,s;\,t)$ as well.
for all $t>0$, $f\in L^q(\mathbb R^3)$ and $F\in L^q(\mathbb R^3)^{3\times 3}$ with
$\mbox{div $F$}\in \cup_{1<\sigma<\infty}L^\sigma(\mathbb R^3)$.
In fact, \eqref{heat-grad} follows from 
$\|\nabla G(t)\|_{1,\mathbb R^3}=4(4\pi t)^{-1/2}$ and \eqref{heat}.
%in which $t$ is of course replaced by $t-s$ and $\tau-s$, respectively, for those evolution operators.
As described in the right-hand sides of \eqref{heat}--\eqref{heat-grad}, the constants
can be taken uniformly in $q$ and $r$ since those with the case $q=r$ give upper bounds (although the upper bound
$\sqrt{8/\pi}$ for \eqref{heat-grad} is not sharp).
This will be taken into account in the proof of the following proposition;
for instance, the specific constant $c_0$ in \eqref{c0-deter} below is independent of $q$.
%This will be used in observation of the dependence of smallness of $\|U_b\|$ on summability exponents in the following proposition.
Gradient estimate \eqref{heat-grad} implies \eqref{heat-dual} by duality and by using 
$\mathbb P_0e^{t\Delta}=e^{t\Delta}\mathbb P_0$, 
where the end-point case ($r=1,\,\infty$) is missing (although $q=1$ is allowed); 
indeed, this case can be actually covered and the additional condition
$\mbox{div $F$}\in\cup_{1<\sigma<\infty}L^\sigma(\mathbb R^3)$ is redundant
if one makes use of estimate of the kernel function of the composite operator $e^{t\Delta}\mathbb P_0\mbox{div}$
as in \cite{Mi00}, however, this is not useful here because we wish to replace $e^{t\Delta}$ by the Oseen
evolution operators \eqref{os-repre0} and \eqref{os-repre-adj}.
We also use the fact:
For every $r_1,\, r_2$ with $1<r_1<r_2<\infty$, there is a constant $C(r_1,r_2)>0$
independent of $q\in [r_1,r_2]$ such that the Riesz transform %the Fujita-Kato projection $\mathbb P_0$ 
satisfies
\begin{equation}
\sup_{r_1\leq q\leq r_2}\|{\mathcal R}\|_{{\mathcal L}(L^q(\mathbb R^3))}
\leq C(r_1,r_2),
%\qquad q\in [q_0^\prime,q_0],
\label{Riesz-bdd}
\end{equation}
which follows simply from the Riesz-Thorin theorem, and thereby, the Fujita-Kato projection
$\mathbb P_0={\mathcal I}+{\mathcal R}\otimes {\mathcal R}$ possesses the same property.
It is obvious that both evolution operators $E(t,s)$ and $E(t,s)^*$ %$\widetilde E(\tau,s;\, t)$ 
fulfill the same estimates as in
\eqref{heat}--\eqref{heat-dual} with the same constants, in the right-hand sides of which $t$ is of course replaced by $t-s$.

The following proposition provides us with %$L^q$^$L^r$ decay estimates 
\eqref{decay-wh} for $1<q\leq r\leq\infty\; (q\neq \infty)$ and
\eqref{decay-grad-wh} for $1<q\leq r<\infty$ when the basic motion $U_b$ is small enough, however,
the smallness is not uniform near the end-point $q=1$. 
The smallness of the basic motions together with the condition $q_0<3$ in \eqref{ass-1}
to develop the linear analysis is needed merely at the present stage.
In fact, the small constant $\alpha_2$ in Theorem \ref{evo-op} is determined by the following proposition. % \ref{est-wh}.
\begin{proposition}
Suppose \eqref{ass-1} and \eqref{ass-2}.
Given $\beta_0>0$, assume that $[U_b]_\theta\leq \beta_0$.
%\begin{enumerate}
%\item
Given $r_1\in (1,4/3]$, there exist constants $\alpha_2(r_1,q_0)>0$
%Let $q\in (1,\infty)$ and $r\in [q,\infty]$.
and $C=C(q,r,\alpha_2,\beta_0,\theta)>0$
such that if $\|U_b\|\leq\alpha_2$, then the evolution operator $T_0(t,s)$ and the pressure $p(t)$ associated with $u(t)=T_0(t,s)\psi$
enjoy %the following estimates with $j=0,\,1$ and $1<q\leq r\leq \infty\;(q\neq\infty)$:
\begin{equation}
\|T_0(t,s)\psi\|_{r,\mathbb R^3}\leq C(t-s)^{-(3/q-3/r)/2}\|\psi\|_{q,\mathbb R^3} 
\label{decay-wh}
\end{equation}
\begin{equation}
\|\nabla T_0(t,s)\psi\|_{r,\mathbb R^3}\leq C(t-s)^{-1/2-(3/q-3/r)/2}\|\psi\|_{q,\mathbb R^3} 
\label{decay-grad-wh}
\end{equation}
\begin{equation}
\|\nabla p(t)\|_{r,\mathbb R^3}\leq C(t-s)^{-1/2-(3/q-3/r)/2}\|\psi\|_{q,\mathbb R^3}
\label{pressure-wh}
\end{equation}
for all $(t,s)$ with $t>s$ and $\psi\in L^q_\sigma(\mathbb R^3)$ as long as
\[
\begin{array}{ll}
r_1\leq q<\infty,\quad
q\leq r\leq\infty \qquad &\mbox{for \eqref{decay-wh}}, \\
r_1\leq q\leq r<\infty &\mbox{for \eqref{decay-grad-wh}--\eqref{pressure-wh}},
\end{array}
\]
where $\|U_b\|$ is given by \eqref{quan}.

The same assertions hold true for the adjoint $T_0(t,s)^*$ and the associated pressure $p_v(s)$ to \eqref{adj-whole}
under the same smallness condition on $\|U_b\|$ %with the same $\alpha_0(q,q_0)$ 
as above.
\label{est-wh}
\end{proposition}

\begin{proof}
The solution $u(t)=T_0(t,s)\psi$ of \eqref{eq-whole} with $\psi\in L^q_\sigma(\mathbb R^3)$ obeys %is also a solution to
\begin{equation}
u(t)=E(t,s)\psi-\int_s^tE(t,\tau)\mathbb P_0(U_b\cdot\nabla u)(\tau)\,d\tau.
\label{int-wh}
\end{equation}
%Given $q\in (3/2,\infty)$, we take $q_0<3$ such that $1/q_0+1/q<1$, where we re-take $q_0$ in \eqref{ass-1} if necessary.
Let $-\infty<s<t<\infty $ and set
\[
M(t,s):=\sup_{\tau\in (s,t)}(\tau-s)^{1/2}\|\nabla u(\tau)\|_{q,\mathbb R^3}.
\]
From \eqref{grad-smooth-wh} we know
$M(t,s)<\infty$ and
\[
M(t,s)\leq C\|\psi\|_{q,\mathbb R^3} \qquad\mbox{for $t-s\leq 2$}.
\]
We are going to show that
\begin{equation}
M(t,s)\leq C\|\psi\|_{q,\mathbb R^3} \qquad\mbox{for $t-s>2$}
\label{large-M}
\end{equation}
with some constant $C>0$ independent of such $(t,s)$ when $\|U_b\|$ is small enough.
Let $t-s>2$.  
Since $u_b(t)\in L^\infty(\Omega)$, one may assume that $q_0\in [8/3,3)$ in \eqref{ass-1}.
Then we have $U_b(t)\in X_{q_0}(\mathbb R^3)\cap X_{6}(\mathbb R^3)$ with
\[
\|U_b(t)\|_{q_0,\mathbb R^3}+\|U_b(t)\|_{6,\mathbb R^3}\leq C\|U_b\|
\]
by \eqref{quan}.
%{\color{blue}
%(here, $\omega_b\in L^\infty$ is needed)}
The following argument is nowadays standard and it was traced back to Chen \cite{Ch93} in the nonlinear context
(where the case $q=3$ is important), but it works merely for $q\in (3/2,\infty)$.
We thus consider the case $q\in [2,\infty)$ %(and $q=\infty$ as well) 
so that
\begin{equation}
%\frac{1}{r_0}+\frac{1}{q_1}\leq 
\frac{1}{q}+\frac{1}{q_0}\leq \frac{7}{8}
\label{auxi-rela}
\end{equation}
on account of $q_0\in [8/3,3)$
and, given $r_1\in (1,4/3]$, the other case $q\in [r_1,2]$ will be discussed later by duality.
We use \eqref{int-wh} and apply \eqref{heat-grad} in which $e^{t\Delta}$ is
replaced by $E(t,s)$ to find
\begin{equation}
\begin{split}
\|\nabla u(t)\|_{q,\mathbb R^3}
&\leq C(t-s)^{-1/2}\|\psi\|_{q,\mathbb R^3}  \\
&\quad +C\int_s^{t-1}(t-\tau)^{-1/2-3/2q_0}\|U_b(\tau)\|_{q_0,\mathbb R^3}\|\nabla u(\tau)\|_{q,\mathbb R^3}\,d\tau  \\
&\quad +C\int_{t-1}^t(t-\tau)^{-3/4}\|U_b(\tau)\|_{6,\mathbb R^3}\|\nabla u(\tau)\|_{q,\mathbb R^3}\,d\tau  \\
&\leq C(t-s)^{-1/2}\|\psi\|_{q,\mathbb R^3}
+c_0\|U_b\| (t-s)^{-1/2}M(t,s)
\end{split}
\label{c0-deter}
\end{equation}
with some constant $c_0=c_0(q_0)>0$, which involves 
$\sup_{8/7\leq r\leq 6}\|\mathbb P_0\|_{{\mathcal L}(L^r(\mathbb R^3))}$, see \eqref{Riesz-bdd} and \eqref{auxi-rela}, 
and %which is dependent on $q_0$ via \eqref{q1} but
is independent of $q\in [2,\infty]$.
This readily follows by splitting the former integral into $\int_s^{(t+s)/2}+\int_{(t+s)/2}^{t-1}$ 
and by using $q_0<3$, see \eqref{ass-1}; in fact,
\begin{equation*}
\begin{split}
&\int_s^{(t+s)/2}\leq \|U_b\| M(t,s)\left(\frac{t-s}{2}\right)^{-1/2-3/2q_0}\int_s^{(t+s)/2}(\tau-s)^{-1/2}\,d\tau,  \\
&\int_{(t+s)/2}^{t-1}\leq \|U_b\| M(t,s)\left(\frac{t-s}{2}\right)^{-1/2}\int_1^\infty \tau^{-1/2-3/2q_0}\,d\tau.
\end{split}
\end{equation*}
As a consequence, %there is a constant $c_0=c_0(q,q_0)>0$ and $C=C(q)>0$ such that
we obtain
\[
M(t,s)\leq C\|\psi\|_{q,\mathbb R^3}+c_0\|U_b\| M(t,s)
\]
for all $(t,s)$ with $t-s>2$, which implies \eqref{large-M} and, therefore, %$\mbox{\eqref{decay-wh}}_{j=1,\,r=q}$ 
\begin{equation}
\|\nabla T_0(t,s)\psi\|_{q,\mathbb R^3}\leq C(t-s)^{-1/2}\|\psi\|_{q,\mathbb R^3}
\label{grad-kihon}
\end{equation}
for all $t>s$, $\psi\in L^q_\sigma(\mathbb R^3)$ and $q\in [2,\infty)$
when $\|U_b\|\leq 1/2c_0$.

Let $1<r_1\leq 4/3$.
We turn to the case $q\in [r_1,2]$, so that $q^\prime\in [2,r_1^\prime]$.
To this end, by use of $\mbox{div $U_b$}=0$, we consider the solution $v(s)=T_0(t,s)^*\phi$ to
\eqref{adj-whole} in the form
\[
v(s)=E(t,s)^*\phi+\int_s^t E(\tau,s)^* \mathbb P_0\,\mbox{div $(v\otimes U_b)(\tau)$}\,d\tau
\]
with the final velocity field of the specific form
$\phi=\mathbb P_0\mbox{div $\Phi$}$, where $\Phi\in C_0^\infty(\mathbb R^3)^{3\times 3}$.
Set
\[
\widetilde M(t,s):=\sup_{\tau\in (s,t)}(t-\tau)^{1/2}\|v(\tau)\|_{q^\prime,\mathbb R^3}
\]
which is finite and
\[
\widetilde M(t,s)\leq C\|\Phi\|_{q^\prime,\mathbb R^3}\qquad\mbox{for $t-s\leq 2$}
\]
on account of \eqref{sm-dual-wh}.
Let $t-s>2$.
Applying \eqref{heat-dual} in which $e^{t\Delta}$ is replaced by $E(t,s)^*$, %$\widetilde E(\tau,s;\,t)$, 
we find
\[
\|v(s)\|_{q^\prime,\mathbb R^3}\leq C(t-s)^{-1/2}\|\Phi\|_{q^\prime,\mathbb R^3}
+c_0^\prime\|U_b\|(t-s)^{-1/2}\widetilde M(t,s)
\]
with some constant $c_0^\prime=c_0^\prime(r_1,q_0)>0$, which involves
$\sup_{2\leq r\leq r_1^\prime}\|\mathbb P_0\|_{{\mathcal L}(L^r(\mathbb R^3))}$ and does not depend on $q^\prime\in [2,r_1^\prime]$,
by the same splitting of the integral %with the same exponent $q_1$, see \eqref{q1}, 
as in \eqref{c0-deter}.
Note that $c_0^\prime$ is increasing to $\infty$ when $r_1\to 1$.
We thus obtain
\begin{equation}
\|T_0(t,s)^*\mathbb P_0\mbox{div $\Phi$}\|_{q^\prime,\mathbb R^3}
\leq C(t-s)^{-1/2}\|\Phi\|_{q^\prime,\mathbb R^3}
\label{auxi-dual}
\end{equation}
for all $t>s$, %and, hence, the same estimate of $T_0(t,s)^*\mathbb P_0\mbox{div}$ for all $t>s$, 
$\Phi\in C_0^\infty(\mathbb R^3)^{3\times 3}$ and $q^\prime\in [2,r_1^\prime]$ when $\|U_b\|\leq 1/2c_0^\prime$.
By continuity the composite operator
$T_0(t,s)^*\mathbb P_0\mbox{div}$
extends to a bounded operator on $L^{q^\prime}(\mathbb R^3)^{3\times 3}$ with \eqref{auxi-dual}.
By duality, we are led to \eqref{grad-kihon} for $q\in [r_1,2]$ as well under the same smallness condition.
Set
\begin{equation}
\alpha_2(r_1,q_0):=\min\left\{\frac{1}{2c_0(q_0)},\; \frac{1}{2c_0^\prime(r_1,q_0)}\right\},
\label{small-zuerst}
\end{equation}
then we furnish \eqref{grad-kihon} for all $t>s$ and $q\in [r_1,\infty)$ when $\|U_b\|\leq \alpha_2$.
By the aforementioned dependence of $c_0^\prime$ on $r_1$,
we see that $\alpha_2$ is decreasing to zero when $r_1\to 1$.

Suppose still this smallness of $\|U_b\|$.
We then use \eqref{int-wh} and \eqref{grad-kihon} to see immediately that
\begin{equation}
\begin{split}
\|u(t)\|_{q,\mathbb R^3}
&\leq C\|\psi\|_{q,\mathbb R^3}+C\int_s^t(t-\tau)^{-1/2}\|U_b(\tau)\|_{3,\mathbb R^3}
\|\nabla u(\tau)\|_{q,\mathbb R^3}\,d\tau   \\
&\leq C(1+\|U_b\|)\|\psi\|_{q,\mathbb R^3}
\end{split}
\label{kihon-0th}
\end{equation}
yielding \eqref{decay-wh} with $r=q\in [r_1,\infty)$ for all $t>s$.
By the interpolation inequality together with %\eqref{grad-kihon} and 
\eqref{kihon-0th} as well as \eqref{grad-kihon} for $q\in [r_1,\infty)$,
we conclude \eqref{decay-wh} for $r\in [q,\infty]$. 
Then, %under the same smallness of $\|U_b\|$, 
\eqref{decay-grad-wh} for $r_1\leq q\leq r<\infty$ 
follows from \eqref{decay-wh} %$\mbox{\eqref{decay-wh}}_{j=0}$ 
and \eqref{grad-kihon} for $q\in [r_1,\infty)$ with the aid of the semigroup property.

By \eqref{eq-whole} we have
\[
-\Delta p=\mbox{div $(U_b\cdot\nabla u)$}
\]
in $\mathbb R^3$, which leads us to
\[
\nabla p=({\mathcal R}\otimes {\mathcal R})(U_b\cdot\nabla u),
\]
where ${\mathcal R}$ denotes the Riesz transform.
From %\eqref{ass-2}--
\eqref{quan} and \eqref{Riesz-bdd} it follows that
\[
\|\nabla p(t)\|_{r,\mathbb R^3}\leq C\|U_b\|\|\nabla T_0(t,s)\psi\|_{r,\mathbb R^3}
\]
%for every $r\in [r_0^\prime,r_0]$, 
which combined with \eqref{decay-grad-wh} concludes \eqref{pressure-wh}.

Finally, one can deduce the same estimates of $T_0(t,s)^*$ and
$T_0(t,s)\mathbb P_0\mbox{div}$ as in \eqref{grad-kihon}--\eqref{auxi-dual} to conclude the same result for the adjoint
$T_0(t,s)^*$.
The proof is complete.
\end{proof}

% 4.2
\subsection{Useful estimates from energy relations}
\label{energy-relation}

By the dissipative structure \eqref{dissi} along with \eqref{evo-eq1} we find
\begin{equation}
\frac{1}{2}\,\partial_t\|T(t,s)F\|_{X_2(\mathbb R^3)}^2+2\|DT(t,s)F\|_{2,\Omega}^2=0
\label{energy}
\end{equation}
for every $F\in X_2(\mathbb R^3)$.
%We see from \eqref{new-para} and \eqref{angular} 
Recall that the energy \eqref{energy-rep} is written as
\begin{equation*}
%\begin{split}
\|T(t,s)F\|_{X_2(\mathbb R^3)}^2
%=\|u(t)\|_{2,\Omega}^2+\int_B|\eta(t)|^2\rho\,dx+\int_B|\omega(t)\times x|^2\rho\,dx  \\
=\|u(t)\|_{2,\Omega}^2
+m|\eta(t)|^2+\frac{2m}{5}|\omega(t)|^2
%\end{split}
\end{equation*}
where $(u,\eta,\omega)=i(U)$ with $U=T(t,s)F$, see \eqref{X1}.  %and \eqref{X-norm0}.
The $L^2$-norm we should adopt is not the usual one
but the norm above to describe exactly the energy relation \eqref{energy}.

Likewise, we use
\eqref{adj-evo1} to observe
\begin{equation}
\frac{1}{2}\,\partial_s\|T(t,s)^*G\|_{X_2(\mathbb R^3)}^2=2\|DT(t,s)^*G\|_{2,\Omega}^2
\label{adj-energy}
\end{equation}
for every $G\in X_2(\mathbb R^3)$, where
\[
\|T(t,s)^*G\|_{X_2(\mathbb R^3)}^2=\|v(s)\|_{2,\Omega}^2 
+m|\eta_v(s)|^2+\frac{2m}{5}|\omega_v(s)|^2
\]
where $(v,\eta_v,\omega_v)=i(V)$ with $V=T(t,s)^*G$.
This subsection shows that 
the energy relations \eqref{energy}--\eqref{adj-energy} imply key inequalities \eqref{key-0}--\eqref{key-1} below,
which are very useful in the next subsection, see \eqref{J-est2} and \eqref{split-1}.
\begin{proposition}
Suppose \eqref{ass-1} and \eqref{ass-2}. %Let $R>1$.
Then there is a constant $C>0$ such that
\begin{equation}
\int_\sigma^t\|T(\tau,s)F\|_{2,\Omega_3}^2\,d\tau\leq C\|T(\sigma,s)F\|_{2,\mathbb R^3}^2
\label{key-0}
\end{equation}
\begin{equation}
\int_s^\sigma \|T(t,\tau)^*G\|_{2,\Omega_3}^2\,d\tau
\leq C\|T(t,\sigma)^*G\|_{2,\mathbb R^3}^2
\label{key-1}
\end{equation}
for all $F,\, G\in X_2(\mathbb R^3)$ and $s<\sigma<t$.
\label{prop-key-1}
\end{proposition}
\begin{proof}
The proof is easy, but 
we briefly show \eqref{key-1}.   % since \eqref{key-0} is proved similarly.
Fix $(t,s)$ with $s<t$ and set 
$V(\tau)=T(t,\tau)^*G$ for $\tau\in (s,t)$, %$(v,\eta_v,\omega_v)=i(V)$.
then we observe
\begin{equation}
\|V(\tau)\|_{2,\Omega_3}\leq C\|V(\tau)\|_{6,\mathbb R^3}
\leq C\|\nabla V(\tau)\|_{2,\mathbb R^3}.
\label{poin}
\end{equation}
Since $V$ is solenoidal, we have $\Delta V=\mbox{div $(2DV)$}$ in
$\mathbb R^3$, which together with $V|_B\in {\rm RM}$ yields
\begin{equation}
\|\nabla V(\tau)\|_{2,\mathbb R^3}^2=2\|Dv(\tau)\|_{2,\Omega}^2,
\label{grad-D}
\end{equation}
where $v=V|_{\Omega}$.
By \eqref{poin} and \eqref{grad-D} we are led to
\begin{equation}
C\|V(\tau)\|^2_{2,\Omega_3}\leq 2\|Dv(\tau)\|^2_{2,\Omega}.
\label{V-D}
\end{equation}
We now employ the energy relation \eqref{adj-energy} and \eqref{X-norm}
to conclude \eqref{key-1} for every $\sigma\in (s,t)$, where $C>0$ is dependent only on $\rho$.
The other estimate \eqref{key-0} is proved similarly.
The proof is complete.
\end{proof}

% 4.3
\subsection{Proof of \eqref{decay-1} except for the case $r=\infty$}
\label{decay-0th}

This subsection is devoted to the proof of \eqref{decay-1} for all $t>s$ when $1<q\leq r<\infty$.
This can be proved simultaneously with \eqref{adj-sm} for all $t>s$.
$L^\infty$-estimate will be studied in subsections \ref{led-est} and \ref{near-infinity}.
The essential step is to show
the uniformly boundedness of both $T(t,s)$ and $T(t,s)^*$
for every $r\in (2,\infty)$ and all $(t,s)$ with $t>s$: %large $(t-s)$:
\begin{equation}
\|T(t,s)F\|_{r,\mathbb R^3}\leq C\|F\|_{r,\mathbb R^3},
\label{unif}
\end{equation}
\begin{equation}
\|T(t,s)^*G\|_{r,\mathbb R^3}\leq C\|G\|_{r,\mathbb R^3}.
\label{unif-adj}
\end{equation}
Since we know \eqref{unif}--\eqref{unif-adj} for $t-s\leq 3$ by Propositions \ref{prop-smooth} and \ref{adj-smooth},
it suffices to derive them for $t-s>3$.
In fact, we have the following lemma.
\begin{lemma}
Assume \eqref{ass-1} and \eqref{ass-2}. 

\begin{enumerate}

\item
Suppose that, for some $r_0\in (2,\infty)$, estimate \eqref{unif} with  
$r=r_0$ holds for all $(t,s)$ with $t>s$ and $F\in {\mathcal E}(\mathbb R^3)$, see \eqref{dense-sub}.

\begin{enumerate}

\item
Let $2\leq q\leq r\leq r_0$.
Then we have
\eqref{decay-1} for all $(t,s)$ with $t>s$ and $F\in X_q(\mathbb R^3)$.

\item
Let $r_0^\prime\leq q\leq r\leq 2$.
Then we have
\eqref{adj-sm} with $j=0$ for all $(t,s)$ with $t>s$ and $G\in X_q(\mathbb R^3)$.
\end{enumerate}

\item
Suppose that, for some $r_0\in (2,\infty)$, estimate \eqref{unif-adj} with 
$r=r_0$ holds for all $(t,s)$ with $t>s$ and $G\in {\mathcal E}(\mathbb R^3)$.
\begin{enumerate}

\item
Let $2\leq q\leq r\leq r_0$.
Then we have \eqref{adj-sm} with $j=0$ for all $(t,s)$ with $t>s$ and 
$G\in X_q(\mathbb R^3)$.

\item
Let $r_0^\prime\leq q\leq r\leq 2$.
Then we have \eqref{decay-1} for all $(t,s)$ with $t>s$ and $F\in X_q(\mathbb R^3)$.
\end{enumerate}

\end{enumerate}
\label{diff-ineq}
\end{lemma}
\begin{proof}
We show (b) of the first item, from which (a) follows by duality and by \eqref{X-norm}.
The other item is proved in the same way.

The case $q=r=2$ is obvious by taking into account \eqref{X-norm}--\eqref{energy-rep} since we have 
the energy relation \eqref{adj-energy}.
Fix $q\in [r_0^\prime,2)$, then we obtain from the assumption together with \eqref{adj-energy} that
\begin{equation}
\|T(t,s)^*G\|_{q,\mathbb R^3}\leq C\|G\|_{q,\mathbb R^3}
\label{unif-adj0}
\end{equation}
for all $t>s$ and $G\in X_q(\mathbb R^3)$, in which ${\mathcal E}(\mathbb R^3)$ is dense, see Proposition \ref{prop-decom}.
From \eqref{unif-adj0} and \eqref{grad-D} with the aid of the interpolation inequality it follows that
\begin{equation*}
\|T(t,s)^*G\|_{X_2(\mathbb R^3)}
\leq C\|\nabla T(t,s)^*G\|_{2,\mathbb R^3}^\mu\|T(t,s)^*G\|_{q,\mathbb R^3}^{1-\mu} 
\leq C\|DT(t,s)^*G\|_{2,\Omega}^\mu\|G\|_{q,\mathbb R^3}^{1-\mu}
\end{equation*}
for all $G\in {\mathcal E}(\mathbb R^3)\setminus\{0\}$,
where $1/2=\mu/6+(1-\mu)/q$. 
We fix $t\in\mathbb R$ and combine the inequality above with \eqref{adj-energy} to find that $v(s)=T(t,s)^*G$ enjoys
\[
\frac{d}{ds}\|v(s)\|_{X_2(\mathbb R^3)}^2\geq \frac{C\|v(s)\|_{X_2(\mathbb R^3)}^{2/\mu}}{\|G\|_{q,\mathbb R^3}^{2(1/\mu-1)}}
\]
for $s\in (-\infty, t)$,
which implies that %\eqref{decay-1} with $r=2$.
\[
\|v(s)\|_{2,\mathbb R^3}\leq C\|v(s)\|_{X_2(\mathbb R^3)}
\leq C(t-s)^{-\frac{\mu}{2(1-\mu)}}\|G\|_{q,\mathbb R^3} 
\]
with
\[
\frac{\mu}{2(1-\mu)}=\frac{3}{2}\left(\frac{1}{q}-\frac{1}{2}\right).
\]
This together with \eqref{unif-adj0} leads us to \eqref{decay-1} for $r_0^\prime\leq q\leq r\leq 2$.
The proof is complete.
\end{proof}
\begin{proposition}
Suppose \eqref{ass-1} and \eqref{ass-2}.
%Let $r_0\in (6,\infty)$.
Given $\beta_0>0$, assume that $[U_b]_\theta\leq \beta_0$.
There is a constant 
$\alpha_1=\alpha_1(q_0)$ %\in (0, \alpha_5(r_0,q_0)]$
%{\color{blue} (... why further smallness than $\alpha_0$ ....... $\alpha_1$ is independent of $\beta_0$ .....)}
such that if $\|U_b\|\leq\alpha_1$, then %\eqref{decay-1} and 
\eqref{adj-sm} with $j=0$ as well as \eqref{decay-1} holds for all $(t,s)$ with $t>s$ and $F,\, G\in X_q(\mathbb R^3)$
%and for all $(q,r)$ with 
provided that $1< q\leq r<\infty$,
%$\alpha_5$ is the constant in Proposition \ref{est-wh}, while
where the constants $C$ in those estimates depend on 
$q,r,\alpha_1,\beta_0,\theta$. 
\label{prop-0th}
\end{proposition}

\begin{proof}
We follow the argument developed by \cite{Hi18, Hi21}. 
Set 
\begin{equation}
\alpha_1=\alpha_1(q_0):=\alpha_2(4/3,q_0),
\label{small-0th}
\end{equation}
where $\alpha_2$ is the constant given in Proposition \ref{est-wh}, see also \eqref{small-zuerst}.
%The exponent $3/2$ is not important and it is just fixed in the interval $(1,2)$.
In what follows we assume $\|U_b\|\leq \alpha_1$.
Let $2<r<\infty$. % and $t-s>3$.
Given $F\in {\mathcal E}(\mathbb R^3)$, see \eqref{dense-sub},
we set $(f,\eta_f,\omega_f)=i(F)$, see \eqref{X1}.
%\[
%u_0=U_0|_\Omega, 
%\]
For the proof of \eqref{unif}, let us take $u_0(t)=T_0(t,s)F$ as the approximation 
near spatial infinity of $U(t)=T(t,s)F$.
We fix a cut-off function $\phi\in C_0^\infty(B_3)$ as in \eqref{cut0}.
%such that $\phi=1$ in $B_2$.
Let $\mathbb B$ be the Bogovskii operator for the domain $\Omega_3$, %\setminus \overline{B_1}$, 
see \eqref{bog-op}. %\cite{Bog, Bor-So, Ga-b, GHH}.
By \eqref{decay-wh} together with \eqref{bog-est} %regularity estimates of that operator 
we see that
\begin{equation}
U_0(t):=(1-\phi)u_0(t)+\mathbb B\left[u_0(t)\cdot\nabla\phi\right] \quad
\mbox{with $u_0(t)=T_0(t,s)F$}
\label{appro-funct}
\end{equation}
belongs to $X_r(\mathbb R^3)$ and satisfies %\eqref{unif} with $F$ replaced by $U_0$.
\begin{equation}
\|U_0(t)\|_{r,\mathbb R^3}\leq C\|F\|_{r,\mathbb R^3}
\label{appro}
\end{equation}
for all $(t,s)$ with $t>s$ since
$\|U_b\|\leq \alpha_1$.
We fix the pressure $p(t)$ 
associated with $U(t)=T(t,s)F$ and also choose
the pressure $p_0(t)$ associated with $u_0(t)=T_0(t,s)F$
such that $\int_{\Omega_3}p_0(t)\,dx=0$, yielding
\begin{equation}
\|p_0(t)\|_{r,\Omega_3}\leq C\|\nabla p_0(t)\|_{r,\Omega_3}.
\label{pre-poincare}
\end{equation}

Let us define $V$ and $p_v$ by
\begin{equation}
U(t)=U_0(t)+V(t), \qquad 
p(t)=(1-\phi)p_0(t)+p_v(t)
\label{appro-perturb}
\end{equation}
and set $(v,\eta,\omega)=i(V)$, see \eqref{X1}.
%Since $U(t)|_B=V(t)|_B$,
Then the rigid motion $\eta+\omega\times x$ associated with $V(t)$ is exactly the same as the one
determined by $U(t)=T(t,s)F$ through \eqref{X1} since
$U(t)|_B=V(t)|_B$.
We see that $v$, $p_v$, $\eta$ and $\omega$ obey
\begin{equation}
\begin{split}
&\partial_tv=\Delta v+(\eta_b(t)-u_b(t))\cdot\nabla v-\nabla p_v+H(t), \qquad %x\in\Omega,\; t\in (s,\infty) \\
\mbox{div $v$}=0 \quad \mbox{in $\Omega\times (s,\infty)$},   \\
&v|_{\partial\Omega}=\eta+\omega\times x, \qquad
v\to 0 \quad \mbox{as $|x|\to\infty$}, \\
&m\frac{d\eta}{dt}+\int_{\partial\Omega}\mathbb S(v,p_v)\nu\,d\sigma=0,  \\
&J\frac{d\omega}{dt}+\int_{\partial\Omega} x\times \mathbb S(v,p_v)\nu\,d\sigma=0,  \\
&v(\cdot,s)=\phi f-\mathbb B[f\cdot\nabla\phi], \quad \eta(s)=\eta_f, \quad \omega(s)=\omega_f,
\end{split}
\label{ivp-v}
\end{equation}
where
\begin{equation}
\begin{split}
H(t)
&=-2\nabla\phi\cdot\nabla u_0(t)-\big[\Delta \phi+(\eta_b(t)-u_b(t))\cdot\nabla\phi\big]\,u_0(t)  \\ 
&\quad +(\nabla\phi)p_0(t)
+\big\{-\partial_t+\Delta+(\eta_b(t)-u_b(t))\cdot\nabla\big\}\mathbb B\left[u_0(t)\cdot\nabla\phi\right].
\end{split}
\label{cutoff-remain}
\end{equation}
By the same symbol $H(t)$ we denote its extension %$H(t)\chi_\Omega$ defined on the whole space 
on $\mathbb R^3$ by setting zero outside $\Omega_3$, then $H(t)\in L^r_R(\mathbb R^3)$ for every $r\in (1,\infty)$.
Furthermore, we find
\begin{equation}
%\|H(t)\|_{2,B_3}\leq C
\|H(t)\|_{r,\Omega_3}
\leq C(t-s)^{-1/2}(1+t-s)^{-3/2r+1/2}\|F\|_{r,\mathbb R^3}
\label{remainder}
\end{equation}
for every $r\in [4/3,\infty)$ 
(we are considering the case $r>2$)
and all $(t,s)$ with $t>s$ %and every $r\in (1,\infty)$ 
owing to Proposition \ref{est-wh} together with \eqref{bog-est} since $\|U_b\|\leq\alpha_1$. 
In fact, it follows from \eqref{pressure-wh} and \eqref{pre-poincare} that
\begin{equation}
\|p_0(t)\|_{r,\Omega_3}\leq\left\{
\begin{array}{ll}
\|\nabla p_0(t)\|_{r,\mathbb R^3}\leq C(t-s)^{-1/2}\|F\|_{r,\mathbb R^3}\quad
& (t-s\leq 1),  \\
\|\nabla p_0(t)\|_{\max\{r,3\},\mathbb R^3}\leq C(t-s)^{-3/2r}\|F\|_{r,\mathbb R^3} & (t-s>1),
\end{array}
\right.
\label{remain-pres}
\end{equation}
%where $r_1\in (r,\infty)$ is taken as large as one wishes, 
and that
\begin{equation*}
\begin{split}
&\quad \|\partial_t\mathbb B\left[u_0(t)\cdot\nabla\phi\right]\|_{r,\Omega_3} \\
&=\|\mathbb B\left[\{\Delta u_0+(\eta_b-U_b)\cdot\nabla u_0-\nabla p_0\}\cdot\nabla\phi\right]\|_{r,\Omega_3}  \\
&\leq C\|\{\Delta u_0+(\eta_b-U_b)\cdot\nabla u_0-\nabla p_0\}\cdot\nabla\phi\|_{W^{1,r^\prime}(\Omega_3)^*}  \\
&\leq C\|\nabla u_0(t)\|_{r,\Omega_3}+C\|p_0(t)\|_{r,\Omega_3}
\end{split}
\end{equation*}
to which one can apply \eqref{decay-grad-wh} %--\eqref{pressure-wh} where $u_0(t)=T_0(t,s)F$ with the exponent $\max\{r,3\}$ as in
and \eqref{remain-pres}.
The other terms are harmless.
In this way, we obtain \eqref{remainder}.
Note that the smallness $\|U_b\|\leq\alpha_1$ is needed only for large $(t-s)$, see Proposition \ref{est-wh}; indeed, we have used
\eqref{decay-wh} with $r=\infty$ and \eqref{decay-grad-wh}--\eqref{pressure-wh} with $r$ replaced by $\max\{r,3\}$.

%Following the monolithic approach, 
We formulate the problem \eqref{ivp-v} as
\[
\frac{dV}{dt}+L_+(t)V=\mathbb PH(t), \quad t\in (s,\infty); \qquad V(s)=\widetilde F
\]
where
\begin{equation}
\widetilde F=\big(\phi f-\mathbb B[f\cdot\nabla\phi]\big)\chi_\Omega+(\eta_f+\omega_f\times x)\chi_B
\label{til-IC}
\end{equation}
whose support is contained in $B_3$ %compact in $\mathbb R^3$.
and which belongs to $X_q(\mathbb R^3)$ for every $q\in (1,\infty)$.
By use of the evolution operator $T(t,s)$, we convert the problem above into
\begin{equation}
V(t)=T(t,s)\widetilde F+\int_s^tT(t,\tau)\mathbb PH(\tau)\,d\tau.
\label{V-int}
\end{equation}
To make full use of the advantage that $H(\tau)$ is %and $V_0$ are 
compactly supported,
it is better to deal with the weak form
\[
\langle V(t),\psi\rangle_{\mathbb R^3,\rho}=\langle\widetilde F, T(t,s)^*\psi\rangle_{\mathbb R^3,\rho}
+\int_s^t\langle H(\tau), T(t,\tau)^*\psi\rangle_{\Omega_3}\,d\tau
\]
for $\psi\in {\mathcal E}(\mathbb R^3)$, where we have used
\[
\langle T(t,\tau)\mathbb P H(\tau), \psi\rangle_{\mathbb R^3,\rho}
=\langle H(\tau), T(t,\tau)^*\psi\rangle_{\mathbb R^3,\rho}
=\langle H(\tau), T(t,\tau)^*\psi\rangle_{\Omega_3}
\]
on account of \eqref{proj-sym-0} and \eqref{dual-evo-sense},
and employ the duality argument.
Let $t-s>3$ and recall that $2<r<\infty$.
%As in \cite[Section 4]{Hi18} and \cite[Section 2]{Hi21}, 
It is readily seen from Proposition \ref{adj-smooth}, \eqref{adj-energy} and \eqref{remainder} that
\begin{equation}
\begin{split}
&\quad |\langle\widetilde F, T(t,s)^*\psi\rangle_{\mathbb R^3,\rho}|+
\left|\left(\int_s^{s+1}+\int_{t-1}^t\right)\langle H(\tau), T(t,\tau)^*\psi\rangle_{\Omega_3}\,d\tau\right|  \\
&\leq \|\widetilde F\|_{X_2(\mathbb R^3)}\|T(t,s)^*\psi\|_{X_2(\mathbb R^3)}
+\int_s^{s+1}\|H(\tau)\|_{2,\Omega_3}\|T(t,\tau)^*\psi\|_{2,\Omega_3}\,d\tau  \\
&\quad +\int_{t-1}^t\|H(\tau)\|_{r,\Omega_3}\|T(t,\tau)^*\psi\|_{r^\prime,\Omega_3}\,d\tau  \\
&\leq C\left(\|\widetilde F\|_{r,\Omega_3}%\|T(t,t-1)^*\psi\|_{2,\mathbb R^3}
+\int_s^{s+1}\|H(\tau)\|_{r,\Omega_3}\,d\tau\right) \|T(t,t-1)^*\psi\|_{X_2(\mathbb R^3)}  \\
&\quad +C\|F\|_{r,\mathbb R^3}\|\psi\|_{r^\prime,\mathbb R^3}\int_{t-1}^t (\tau-s)^{-3/2r}\,d\tau  \\
&\leq C\|F\|_{r,\mathbb R^3}\|\psi\|_{r^\prime,\mathbb R^3}  \\
&\leq C\|F\|_{r,\mathbb R^3}\|\psi\|_{X_{r^\prime}(\mathbb R^3)}.
\end{split}
\label{easy-part}
\end{equation}
%for all $(t,s)$ with $t-s>3$ and $r\in (2,\infty)$.

Our main task is thus to discuss the term
\[
J:=\int_{s+1}^{t-1}\langle H(\tau), T(t,\tau)^*\psi\rangle_{\Omega_3}\,d\tau,
\]
for which we have
\begin{equation}
%\begin{split}
|J|\leq
\int_{s+1}^{t-1}\|H(\tau)\|_{2,\Omega_3}\|T(t,\tau)^*\psi\|_{2,\Omega_3}\,d\tau.
%\end{split}
\label{J-est}
\end{equation}
We make use of \eqref{key-1} with $\sigma=t-1$ %and $R=3$ 
along with \eqref{remainder} for $r>2$ to find
\begin{equation}
%\begin{split}
|J|\leq C\left(\int_{s+1}^{t-1}(\tau-s)^{-3/r}\,d\tau\right)^{1/2}\|F\|_{r,\mathbb R^3}\|T(t,t-1)^*\psi\|_{2,\mathbb R^3}
%\end{split}
\label{J-est2}
\end{equation}
which combined with \eqref{adj-sm} with $\tau_*=1$ yields 
\[
|J|\leq C\|F\|_{r,\mathbb R^3}\|\psi\|_{r^\prime,\mathbb R^3}
\leq C\|F\|_{r,\mathbb R^3}\|\psi\|_{X_{r^\prime}(\mathbb R^3)}
\]
for $t-s>3$ and $\psi\in {\mathcal E}(\mathbb R^3)$ provided $2<r<3$.
%In view of \eqref{appro-perturb},
This together with \eqref{easy-part} and \eqref{appro} %and \eqref{appro-perturb} 
imply \eqref{unif} for $t-s>3$ and, therefore, for all $(t,s)$ with $t>s$ since we already know \eqref{unif}
for $t-s\leq 3$ from Proposition \ref{prop-smooth}.
Hence, by virtue of Lemma \ref{diff-ineq} we obtain \eqref{adj-sm} with $j=0$ for all $(t,s)$ with $t>s$ and $G\in X_q(\mathbb R^3)$
provided $3/2< q\leq r\leq 2$. 

With this at hand, we proceed to the next step in which the case $r\in (3,6)$ %{\color{blue} (... notation $r_0$....)}
is discussed.
Given such $r$, we set
\begin{equation}
\mu=1-\frac{3}{r}, \qquad
\frac{1}{q}-\frac{1}{2}=\frac{\mu}{3},
\label{auxi-rate}
\end{equation}
where $\mu$ comes from the growth rate
$(t-s-1)^\mu$ of the integral of \eqref{J-est2}, which we intend to
overcome by use of the decay property of the adjoint evolution operator.
Then we have $\mu\in (0,1/2)$ and $q\in (3/2,2)$.
We use \eqref{key-1} with $\sigma=(s+t)/2$, apply the result obtained in the previous step 
and recall \eqref{adj-sm} with $\tau_*=1$ to furnish
\begin{equation}
\begin{split}
\int_{s+1}^{(s+t)/2}\|T(t,\tau)^*\psi\|_{2,\Omega_3}^2\,d\tau
&\leq C\|T(t,(s+t)/2)^*\psi\|_{2,\mathbb R^3}^2  \\
&\leq C(t-s-2)^{-\mu}\|T(t,t-1)^*\psi\|_{q,\mathbb R^3}^2  \\
&\leq C(t-s-2)^{-\mu}\|\psi\|_{r^\prime,\mathbb R^3}^2
\end{split}
\label{split-1}
\end{equation}
for $t-s>2$.
Following \eqref{adj-def}, we set $W(t-\tau):=T(t,\tau)^*\psi$ 
and $\sigma:=(t-s-2)/2$.
We then rewrite \eqref{split-1} as
\[
\int_{1+\sigma}^{1+2\sigma}\|W(\tau)\|_{2,\Omega_3}^2\,d\tau
\leq c_1\,\sigma^{-\mu}\|\psi\|_{r^\prime,\mathbb R^3}^2
\]
with some $c_1>0$ for all $\sigma>0$, from which
one can deduce the following optimal growth estimate by dyadic splitting method developed in \cite[Lemma 3.4]{Hi18}:
\begin{equation}
\begin{split}
\int_{(s+t)/2}^{t-1}\|T(t,\tau)^*\psi\|_{2,\Omega_3}\,d\tau
&=\int_1^{1+\sigma}\|W(\tau)\|_{2,\Omega_3}\,d\tau  \\
&=\sum_{j=0}^\infty \int_{1+\sigma/2^{j+1}}^{1+\sigma/2^j}\|W(\tau)\|_{2,\Omega_3}\,d\tau  \\
&\leq \sqrt{c_1}\,\sigma^{(1-\mu)/2}\|\psi\|_{r^\prime,\mathbb R^3}\sum_{j=0}^\infty 2^{-(1-\mu)(j+1)/2}  \\
&=C(t-s-2)^{(1-\mu)/2}\|\psi\|_{r^\prime,\mathbb R^3}
\end{split}
\label{split-2}
\end{equation}
for $t-s>2$.
We now split \eqref{J-est} into two parts as below
which should be comparable with each other
and then use \eqref{split-1}--\eqref{split-2} 
together with \eqref{remainder} to find
\begin{equation*}
\begin{split}
|J|&\leq \int_{s+1}^{(s+t)/2}+\int_{(s+t)/2}^{t-1}  \\
&\leq C(t-s)^{(1-3/r)/2}\|F\|_{r,\mathbb R^3}\left(\int_{s+1}^{(s+t)/2}\|T(t,\tau)^*\psi\|_{2,\Omega_3}^2\,d\tau\right)^{1/2}  \\
&\quad +C(t-s)^{-3/2r}\|F\|_{r,\mathbb R^3}\int_{(s+t)/2}^{t-1}\|T(t,\tau)^*\psi\|_{2,\Omega_3}\,d\tau  \\
&\leq C\|F\|_{r,\mathbb R^3}\|\psi\|_{r^\prime,\mathbb R^3}
\end{split}
\end{equation*}
for $t-s>3$, yielding \eqref{unif} for all $t>s$ provided $3<r<6$.
Then Lemma \ref{diff-ineq} concludes \eqref{adj-sm} with $j=0$ 
for all $(t,s)$ with $t>s$ provided $6/5< q\leq r\leq 2$.

At the final stage, given $r\in (6,\infty)$, we take the same $\mu$ and $q$ as in \eqref{auxi-rate}, then we observe 
$\mu\in (1/2,1)$ and $q\in (6/5,3/2)$.
The same argument as above with better decay property of $T(t,s)^*$ obtained in the previous step
%with use of \eqref{split-1}--\eqref{split-2} 
implies \eqref{unif} for all $(t,s)$ with $t>s$ and, thereby,
\eqref{decay-1} for such $(t,s)$ provided $2\leq q\leq r<\infty$ as well as
\eqref{adj-sm} with $j=0$ for the same $(t,s)$ %all $(t,s)$ with $t>s$ 
provided $1< q\leq r\leq 2$.

The opposite case, that is, \eqref{decay-1} for $1< q\leq r\leq 2$ and \eqref{adj-sm} with $j=0$ for 
$2\leq q\leq r<\infty$
can be discussed with the aid of \eqref{key-0}
under the same condition $\|U_b\|\leq \alpha_1$ 
in the similar fashion, see also the last part of \cite[section 4]{Hi18}.
Finally, the remaining case $1< q<2<r<\infty$ for both estimates is obvious because of the semigroup properties
\eqref{semi} and \eqref{back-semi}.
The proof is complete.
\end{proof}
\begin{remark}
It would be remarkable that Proposition \ref{prop-0th} is established
under the smallness of $\|U_b\|$ uniformly in $(q,r)$
regardless of the circumstances in Proposition \ref{est-wh}.
This is because one needs \eqref{decay-wh}--\eqref{pressure-wh}
only with the case $q>2$.
The other remark concerns the constants in
\eqref{decay-1} and \eqref{adj-sm} with $j=0$.
Let $\|U_b\|\leq \alpha_0$, then,
according to Proposition \ref{prop-smooth}, the constant $C$ in \eqref{decay-1} near $t=s$ depends on $\alpha_0$.
This is why the constant $C$ in Proposition \ref{prop-0th} is also dependent on $\alpha_1$. It is also the case in Propsitions \ref{led-prop} and \ref{prop-near-inf} below.
\label{const-dep}
\end{remark}

%% 4.4
\subsection{Local energy decay estimates}
\label{led-est}

For the proof of \eqref{grad-new-form}, it suffices to show that
\begin{equation}
\|\nabla T(t,s)F\|_{q,\mathbb R^3}\leq C(t-s)^{-\min\{1/2,\, 3/2q\}}\|F\|_{q,\mathbb R^3}
\label{grad-decay}
\end{equation}
for all $(t,s)$ with $t-s>2$ and $F\in X_q(\mathbb R^3)$ on account of the semigroup property, \eqref{decay-1} with $r<\infty$ and
Proposition \ref{prop-smooth}.
%Succesively, estimate \eqref{decay-1} with $r=\infty$ is also discussed.
As in \cite{I,KShi,EShi05,HiShi09,Hi20,Hi21,EMT},
let us split \eqref{grad-decay} into estimates of $\|\nabla T(t,s)F\|_{q,B_3}$ and $\|\nabla T(t,s)F\|_{q,\mathbb R^3\setminus B_3}$.
The former is given by the following proposition and it is called the local energy decay property,
whereas the latter is studied in the next subsection.
To discuss the latter, the local energy decay of $\partial_tT(t,s)F$ is also needed, see \eqref{led-deri} below.
The following proposition gives us \eqref{led}--\eqref{led-deri} for $1<q<\infty$ when $\|U_b\|$ is small enough,
however, the smallness is not uniform near $q=1$; in fact, this circumstance arises from Proposition \ref{est-wh}.
\begin{proposition}
Suppose \eqref{ass-1} and \eqref{ass-2}.
Given $r_1\in (1,4/3]$ and
$\beta_0\in (0,\infty)$, assume that $\|U_b\|\leq \alpha_2\in (0,\alpha_1]$ and $[U_b]_\theta\leq\beta_0$,
%Let $r_1\in (1,3/2]$ and suppose $\|U_b\|\leq \alpha_5(r_1,q_0)$, 
where $\alpha_2=\alpha_2(r_1,q_0)$ and $\alpha_1=\alpha_1(q_0)=\alpha_2(4/3,q_0)$, see \eqref{small-0th},
 are respectively the constants given in Propositions \ref{est-wh}
and \ref{prop-0th}, while
$\|U_b\|$ and $[U_b]_\theta$ are given by \eqref{quan}.
Then there is a constant % $\alpha_2=\alpha_2(q,q_0)\in (0,\alpha_0(q,q_0)]$ 
$C=C(q,\alpha_2,\beta_0,\theta)>0$ such that
\begin{equation}
\|T(t,s)F\|_{W^{1,q}(B_3)}\leq C(t-s)^{-3/2q}\|F\|_{q,\mathbb R^3}
\label{led}
\end{equation}
\begin{equation}
\|\partial_tT(t,s)F\|_{W^{-1,q}(\Omega_3)}+\|p(t)\|_{q,\Omega_3}
\leq C(t-s)^{-3/2q}\|F\|_{q,\mathbb R^3}
\label{led-deri}
\end{equation}
for all $(t,s)$ with $t-s>2$ and $F\in X_q(\mathbb R^3)$ 
as long as $r_1\leq q<\infty$,
where $p(t)$ denotes the pressure associated with $T(t,s)F$ and it is singled out in such a way that 
$\int_{\Omega_3}p(t)\,dx=0$ for each $t\in (s,\infty)$.

Under less smallness condition $\|U_b\|\leq \alpha_1$ than described above, that is, the same one as in Proposition \ref{prop-0th},
there is a constant $C=C(q,\alpha_1,\beta_0,\theta)>0$ such that
\begin{equation}
\|T(t,s)F\|_{\infty,B_3}\leq C(t-s)^{-3/2q}\|F\|_{q,\mathbb R^3}
\label{led-sup}
\end{equation}
for all $(t,s)$ with $t-s>2$ and $F\in X_q(\mathbb R^3)$ as long as $1<q<\infty$.

The same assertions hold true for the adjoint $T(t,s)^*$ and the associated pressure $p_v(s)$ to \eqref{backward}
under the same smallness conditions 
on $\|U_b\|$ as above.
\label{led-prop}
\end{proposition}
\begin{proof}
Given $q\in (1,\infty)$, let us take $r_0$ so large that
\[
\max\left\{\frac{2q}{q-1},\; q,\, 6\right\}<r_0<\infty,
\]
which yields
\begin{equation}
q>r_0^\prime, \qquad
\kappa:=\frac{3}{2}\left(1-\frac{2}{r_0}\right)>\max\left\{\frac{3}{2q},\; 1\right\}.
\label{parti-expo}
\end{equation}
Let $\|U_b\|\leq \alpha_1(q_0)$, then we have \eqref{decay-1} except for the case $r=\infty$ by
Proposition \ref{prop-0th}. 
If 
$H\in L^q_R(\mathbb R^3)$ satisfies $H(x)=0$ a.e. $\mathbb R^3\setminus B_3$,
then it follows from 
Proposition \ref{prop-smooth} 
and \eqref{tanabe-est} with $\tau_*=1$ as well as Proposition \ref{prop-0th} that
\begin{equation*}
\begin{split}
\quad \|T(t,s)\mathbb P H\|_{W^{1,r_0}(\mathbb R^3)} +\|L_+(t)T(t,s)\mathbb P H\|_{r_0,\mathbb R^3} 
&\leq C\|T(t-1,s)\mathbb P H\|_{r_0,\mathbb R^3}  \\
&\leq C(t-s-1)^{-(3/r_0^\prime-3/r_0)/2}\|\mathbb P H\|_{r_0^\prime,\mathbb R^3}  \\
&\leq C(t-s)^{-\kappa}\|H\|_{q,\mathbb R^3}
\end{split}
\end{equation*}
for $t-s>2$. 
This along with Propositions \ref{prop-smooth} and \ref{prop-pressure} implies that
\begin{equation}
\|T(t,s)\mathbb P H\|_{W^{1,q}(B_3)}\leq C(t-s)^{-1/2}(1+t-s)^{-\kappa+1/2}\|H\|_{q,\mathbb R^3}
\label{1st-led}
\end{equation}
and that
\begin{equation}
\|\partial_tT(t,s)\mathbb P H\|_{W^{-1,q}(\Omega_3)}
\leq C(t-s)^{-\gamma}(1+t-s)^{-\kappa+\gamma}\|H\|_{q,\mathbb R^3}
\label{1st-led-deri}
\end{equation}
for all $(t,s)$ with $t>s$, where $\gamma\in \big((1+1/q)/2,\,1\big)$ is fixed arbitrarily.

We fix $r_1\in (1,4/3]$ arbitrarily, and let $q\in [r_1,\infty)$.
Given $F\in {\mathcal E}(\mathbb R^3)$, which is dense in $X_q(\mathbb R^3)$ (see Proposition \ref{prop-decom}),
the function \eqref{appro-funct} and its temporal derivative $\partial_t U_0(t)$
enjoy the desired estimates as in \eqref{led}--\eqref{led-deri} 
due to \eqref{decay-wh}--\eqref{pressure-wh} and \eqref{bog-est}
(together with the equation \eqref{eq-whole} for $\partial_tT_0(t,s)F$) under the condition
\begin{equation}
\|U_b\|\leq \alpha_2(r_1,q_0)\leq \alpha_2(4/3,q_0)=\alpha_1(q_0),
\label{small-led}
\end{equation}
see \eqref{small-0th} as well as \eqref{small-zuerst}.
Our task is thus to estimate $V(t)$ defined by \eqref{appro-perturb}.
We use the integral equation \eqref{V-int} and its temporal derivative
\begin{equation}
\partial_tV(t)=\partial_tT(t,s)\widetilde F+\mathbb P H(t)+\int_s^t \partial_tT(t,\tau)\mathbb P H(\tau)\,d\tau,
\label{V-int-deri}
\end{equation}
in $W^{-1,q}(\Omega_3)$, where
we can apply \eqref{1st-led}--\eqref{1st-led-deri} to $\widetilde F$ given by \eqref{til-IC}
and also to $H(\tau)$ given by \eqref{cutoff-remain}
since they vanish outside $B_3$.
In view of the relation \eqref{parti-expo}, we see at once that $T(t,s)\widetilde F$ and $\partial_tT(t,s)\widetilde F$ fulfill the desired estimates;
moreover, so does the second term $\mathbb P H(t)$ of \eqref{V-int-deri} already by \eqref{remainder},
which holds for $r\in [r_1,\infty)$ under \eqref{small-led}.
From \eqref{1st-led}--\eqref{1st-led-deri} together with
\eqref{remainder} it follows that
\begin{equation*}
\begin{split}
&\quad \int_s^t\|T(t,\tau)\mathbb P H(\tau)\|_{W^{1,q}(B_3)}\,d\tau  \\
&\leq C\|F\|_{q,\mathbb R^3}\left(\int_s^{(s+t)/2}+\int_{(s+t)/2}^t\right)  \\
&\qquad (t-\tau)^{-1/2}(1+t-\tau)^{-\kappa+1/2}
(\tau-s)^{-1/2}(1+\tau-s)^{-3/2q+1/2}\,d\tau  \\
&\leq C(t-s)^{-3/2q}\|F\|_{q,\mathbb R^3}
\end{split}
\end{equation*}
and, similarly, that
\begin{equation*}
\begin{split}
&\quad \int_s^t \|\partial_t T(t,\tau)\mathbb P H(\tau)\|_{W^{-1,q}(\Omega_3)}\,d\tau \\
&\leq C\|F\|_{q,\mathbb R^3}\left(\int_s^{(s+t)/2}+\int_{(s+t)/2}^t\right)  \\
&\qquad (t-\tau)^{-\gamma}(1+t-\tau)^{-\kappa+\gamma} (\tau-s)^{-1/2}(1+\tau-s)^{-3/2q+1/2}\,d\tau  \\
&\leq C(t-s)^{-3/2q}\|F\|_{q,\mathbb R^3}
\end{split}
\end{equation*}
for all $(t,s)$ with $t-s>2$, which conclude \eqref{led-deri} for $\partial_tT(t,s)F$ as well as \eqref{led}
%by taking into account \eqref{parti-expo} 
provided \eqref{small-led} is satisfied.

Since $\int_{\Omega_3}p(t)\,dx=0$, we see that
\[
\|p(t)\|_{q,\Omega_3}\leq C\|\nabla p(t)\|_{W^{-1,q}(\Omega_3)}
\]
from which together with the first equation of \eqref{eq-linear}, \eqref{led}, \eqref{led-deri} for $\partial_tT(t,s)F$
and \eqref{ass-1}, we obtain \eqref{led-deri} for the pressure under the same smallness \eqref{small-led}.

Let $\|U_b\|\leq \alpha_1(q_0)=\alpha_2(4/3,q_0)$, see \eqref{small-0th}, then
\eqref{led} with $q>3$ implies \eqref{led-sup} for the same $q$, which combined with Proposition \ref{prop-0th} gives
\eqref{led-sup} for the other case $q\in (1,3]$ by the semigroup property \eqref{semi}. 
Finally, the argument for the adjoint $T(t,s)^*$ works essentially in the same manner.
The proof is complete.
\end{proof}
Propositions \ref{prop-smooth}, \ref{prop-pressure} and \ref{led-prop} immediately lead us to the following corollary, 
that plays an important role in the next subsection.
\begin{corollary}
Assume the same conditions as in the first half of Proposition \ref{led-prop}. 
Then, for every $\gamma\in \big((1+1/q)/2,\,1\big)$, there are constants $C_1=C_1(\gamma,q,\alpha_2,\beta_0,\theta)>0$ 
and $C_2=C_2(q,\alpha_2,\beta_0,\theta)>0$ such that
\begin{equation}
\|\partial_tT(t,s)F\|_{W^{-1,q}(\Omega_3)}+\|p(t)\|_{q,\Omega_3}\leq C_1(t-s)^{-\gamma}(1+t-s)^{-3/2q+\gamma}\|F\|_{q,\mathbb R^3}
\label{pressure-glo}
\end{equation}
\begin{equation}
\|T(t,s)F\|_{W^{1,q}(B_3)}\leq C_2(t-s)^{-1/2}(1+t-s)^{-3/2q+1/2}\|F\|_{q,\mathbb R^3}
\label{led-glo}
\end{equation}
for all $(t,s)$ with $t>s$ and $F\in X_q(\mathbb R^3)$ as long as $r_1\leq q<\infty$, 
where the associated pressure $p(t)$ is chosen as in Proposition \ref{led-prop}.

The same assertions hold true for the adjoint $T(t,s)^*$ and the associated pressure $p_v(s)$ to \eqref{backward} under the same
condition $\|U_b\|\leq\alpha_2(r_1,q_0)$.
\label{led-cor}
\end{corollary}

%% 4.5
\subsection{Large time behavior near spatial infinity}
\label{near-infinity}

In this subsection we deduce the decay property of %the evolution operator 
$\nabla T(t,s)$ near spatial infinity under the same conditions as in Proposition \ref{led-prop}
to complete the proof of \eqref{grad-decay}.
\begin{proposition}
Suppose \eqref{ass-1} and \eqref{ass-2}.
Given $r_1\in (1,4/3]$ and $\beta_0\in (0,\infty)$,
%Let $1<q<\infty$ and $\beta_0>0$.
assume that $\|U_b\|\leq \alpha_2\in (0,\alpha_1]$ and $[U_b]_\theta\leq\beta_0$, where $\alpha_2=\alpha_2(r_1,q_0)$ and
$\alpha_1=\alpha_1(q_0)$ are the constants given in Propositions \ref{est-wh} and \ref{prop-0th},
while $\|U_b\|$ and $[U_b]_\theta$ are given by \eqref{quan}.
Then there is a constant %$\alpha_3(q,q_0)\in (0,\alpha_2(q,q_0)]$ and 
$C=C(q,\alpha_2,\beta_0,\theta)>0$
such that %if $\|U_b\|\leq\alpha_3$, 
\begin{equation}
\|\nabla T(t,s)F\|_{q,\mathbb R^3\setminus B_3}\leq C(t-s)^{-\min\{1/2,\, 3/2q\}}\|F\|_{q,\mathbb R^3}
\label{final-grad}
\end{equation}
for all $(t,s)$ with $t-s>2$ and $F\in X_q(\mathbb R^3)$ as long as $r_1\leq q<\infty$.

Under the same condition $\|U_b\|\leq \alpha_1$ %merely, that is, the same one 
as for \eqref{led-sup}, %in Proposition \ref{prop-0th},
there is a constant $C=C(q,\alpha_1,\beta_0,\theta)>0$ such that
%$\alpha_6(q_0)\in (0,\alpha_1(q_0)]$ with $\alpha_1$ being the constant in Proposition \ref{prop-0th}
%
\begin{equation}
\|T(t,s)F\|_{\infty,\mathbb R^3\setminus B_3}\leq C(t-s)^{-3/2q}\|F\|_{q,\mathbb R^3}
\label{final-sup}
\end{equation}
for all $(t,s)$ with $t-s>2$ and $F\in X_q(\mathbb R^3)$ as long as $1<q<\infty$.

The same assertions hold true for the adjoint $T(t,s)^*$ under the same smallness conditions on $\|U_b\|$ as above.
\label{prop-near-inf}
\end{proposition}

\begin{proof}
Given $F\in {\mathcal E}(\mathbb R^3)$, see \eqref{dense-sub},
we set $U(t)=T(t,s)F$, $u(t)=U(t)|_\Omega$ and take the associated pressure $p(t)$
such that $\int_{\Omega_3}p(t)\,dx=0$.
Using the same cut-off function $\phi$ and the Bogovskii operator $\mathbb B$, see \eqref{cut0} and \eqref{bog-op},
as in the proof of Propositions
\ref{prop-0th} and \ref{led-prop}, we consider
\[
v(t)=(1-\phi)u(t)+\mathbb B[u(t)\cdot\nabla\phi], \qquad
p_v(t)=(1-\phi)p(t).
\]
Then $v(t)$ obeys
\begin{equation}
v(t)=T_0(t,s)\widetilde F+\int_s^tT_0(t,\tau)\mathbb P_0K(\tau)\,d\tau
\label{int-final}
\end{equation}
in terms of the evolution operator $T_0(t,s)$ without the rigid body studied in subsection \ref{evo-whole},
where
\[
\widetilde F=(1-\phi)F+\mathbb B[F\cdot\nabla \phi]\in C^\infty_{0,\sigma}(\mathbb R^3)
\]
and
\begin{equation*}
\begin{split}
K(t)
&=2\nabla\phi\cdot\nabla u(t)+\left[\Delta\phi+\big(\eta_b(t)-u_b(t)\big)\cdot\nabla\phi\right]u(t)  \\
&\quad -(\nabla\phi)p(t)+\left\{ \partial_t-\Delta-\big(\eta_b(t)-u_b(t)\big)\cdot\nabla \right\} \mathbb B[u(t)\cdot\nabla\phi].
\end{split}
\end{equation*}
%which vanishes outside $B_3$.
By the same symbol $K(t)$ we denote its extension on $\mathbb R^3$ by setting zero outside $\Omega_3$.
Given $r_1\in (1,4/3]$, we assume the smallness \eqref{small-led} on $\|U_b\|$
as in Proposition \ref{led-prop} and Corollary \ref{led-cor},
then it follows from \eqref{pressure-glo}--\eqref{led-glo} and \eqref{ass-1} together with \eqref{bog-est}
%regularity estimates of the Bogovskii operator \cite{Bog, Bor-So, Ga-b, GHH} 
that
\begin{equation}
\|K(t)\|_{r,\mathbb R^3}\leq C(t-s)^{-\gamma}(1+t-s)^{-3/2q+\gamma}\|F\|_{q,\mathbb R^3}
\label{remainder2}
\end{equation}
for all $(t,s)$ with $t>s$ and $r\in (1,q]$ as long as $q\in [r_1,\infty)$, where $\gamma\in \big((1+1/q)/2,\,1\big)$ is fixed arbitrarily.

Let $t-s>2$.
For the proof of \eqref{final-grad}--\eqref{final-sup}, our task is to deduce the desired eatimates of
$\|\nabla v(t)\|_{q,\mathbb R^3}$ and $\|v(t)\|_{\infty,\mathbb R^3}$ by using \eqref{int-final}.
It is obvious that the term $T_0(t,s)\widetilde F$ fulfills those thanks to \eqref{decay-wh}--\eqref{decay-grad-wh}
provided $q\geq r_1$.
%as long as $\|U_b\|\leq \alpha_0$, that is already met since $\alpha_2\leq\alpha_0$.
By \eqref{remainder2} along with \eqref{decay-grad-wh} (for $q\geq r_1$) and by choosing, for instance,
\[
\left\{
\begin{array}{ll}
r=q \qquad & \mbox{if $q\in [r_1,3/2)$}, \\
r=4/3 & \mbox{if $q\in [3/2,\infty)$},
\end{array}
\right.
\]
(note that $r\in [r_1,q]$ is required to apply Proposition \ref{est-wh}),
%under the condition \eqref{grad-small}, 
we find
\begin{equation}
\begin{split}
&\quad \int_s^t\|\nabla T_0(t,\tau)\mathbb P_0K(\tau)\|_{q,\mathbb R^3}\,d\tau  \\
&\leq C\|F\|_{q,\mathbb R^3}\left(\int_s^{(s+t)/2}+\int_{(s+t)/2}^t\right) \\
&\qquad (t-\tau)^{-1/2}(1+t-\tau)^{-(3/r-3/q)/2}(\tau-s)^{-\gamma}(1+\tau-s)^{-3/2q+\gamma}\,d\tau  \\
&\leq C(t-s)^{-\min\{1/2,\, 3/2q\}}\|F\|_{q,\mathbb R^3}
\end{split}
\label{final-duha}
\end{equation}
which concludes \eqref{final-grad} under the condition $\|U_b\|\leq \alpha_2(r_1,q_0)$. 

For the $L^\infty$-estimate, the integrand of \eqref{final-duha} is replaced by
\[
(t-\tau)^{-3/2q}(1+t-\tau)^{-(3/r-3/q)/2}(\tau-s)^{-\gamma}(1+\tau-s)^{-3/2q+\gamma}.
\]
Let $3/2<q<\infty$.
We then take $r=4/3$ %$r\in (1,3/2)$, assume \eqref{final-small} for such $r$ 
and compute the integral in the 
same way as above %by use of \eqref{decay-wh} with $r=\infty$ 
to deduce \eqref{final-sup} provided
\begin{equation*}
\|U_b\|\leq \alpha_1(q_0)=\alpha_2(4/3,q_0), %\leq \alpha_5(3/2,q_0)=\alpha_1(q_0),
%\label{inf-small}
\end{equation*}
which ensures \eqref{remainder2} with $q\in (3/2,\infty)$ and allows us to apply \eqref{decay-wh} with $r=\infty$
and $q=4/3$.
In order to derive this for the other case $q\in (1,3/2]$ as well,
%on account of the semigroup property \eqref{semi}, %for the other case $q\in (1,3/2]$ as well,
we have only to combine \eqref{final-sup} for $q=2$ (say) obtained above with \eqref{decay-1} for $q\in (1,3/2]$
under the condition %\eqref{inf-small} 
$\|U_b\|\leq \alpha_1(q_0)$
by taking into account the semigroup property \eqref{semi}.
The adjoint $T(t,s)^*$ is discussed similarly.
The proof is complete.
\end{proof}

%% 4.6
\subsection{Proof of Theorem \ref{evo-op}}
\label{proof-1}

We collect Proposition \ref{prop-smooth} (case $r=\infty$), Proposition \ref{prop-0th}, \eqref{led-sup} and 
\eqref{final-sup} %\ref{led-prop} and \ref{prop-near-inf}
to furnish \eqref{decay-1} provided that $\|U_b\|\leq\alpha_1(q_0)$. % see \eqref{inf-small}.

Let $r_1\in (1,4/3]$ and suppose %\eqref{grad-small}.
$\|U_b\|\leq \alpha_2(r_1,q_0)\leq\alpha_1(q_0)$.
Then it follows from \eqref{led} and \eqref{final-grad} along with Proposition \ref{prop-smooth} that
\begin{equation}
\|\nabla T(t,s)F\|_{q,\mathbb R^3}\leq C(t-s)^{-1/2}(1+t-s)^{\max\{(1-3/q)/2,\,0\}}\|F\|_{q,\mathbb R^3}
\label{grad-qq}
\end{equation}
for all $(t,s)$ with $t>s$ and $F\in X_q(\mathbb R^3)$ as long as $r_1\leq q<\infty$.
%Since \eqref{grad-small} implies \eqref{inf-small},
One may combine \eqref{grad-qq} with \eqref{decay-1} to conclude
\eqref{grad-new-form} for $r\in [r_1,\infty)$ and $q\in (1,r]$.

%As mentioned in Remark \ref{rem-grad}, 
With the same estimates for the adjoint $T(t,s)^*$ under the same smallness of $\|U_b\|$ at hand, let us show \eqref{compo}.
Let $1< q<\infty$ and $\phi\in {\mathcal E}(\mathbb R^3)$. %$1<r^\prime\leq q^\prime\leq 3$.
Then, in view of \eqref{dual-evo-sense} together with \eqref{proj-sym-0}, see also \eqref{other-pair}, we have
%\eqref{grad-new-form} for $T(t,s)^*$ to
%
\begin{equation*}
\begin{split}
&\quad\big|\langle T(t,s)\mathbb P\mbox{div $F$},\; \phi\rangle_{\mathbb R^3,\rho}\big|  \\
&=\left|-\langle F,\; \nabla T(t,s)^*\phi\rangle_{\mathbb R^3,\rho}+(1-\rho)\int_{\partial\Omega}(F\nu)\cdot\big(T(t,s)^*\phi\big)\,d\sigma\right|  \\
&\leq \|F\|_{q,(\mathbb R^3,\rho)}\|\nabla T(t,s)^*\phi\|_{q^\prime,(\mathbb R^3,\rho)}  \\
&\leq C\|F\|_{q,\mathbb R^3}\|\nabla T(t,s)^*\phi\|_{q^\prime,\mathbb R^3}
\end{split}
\end{equation*}
for all $F\in L^q(\mathbb R^3)^{3\times 3}$ %and $\phi\in {\mathcal E}(\mathbb R^3)$
with $F\nu=0$ at $\partial\Omega$ %(from both directions) 
as well as $\mbox{div $F$}$ %being assumed to 
belonging to %a Lebesgue space over $\mathbb R^3$ 
$L^p(\mathbb R^3)$ for some $p\in (1,\infty)$
(so that $(F\nu)|_{\partial\Omega}$
from both directions coincide with each other) and fulfilling $(\mbox{div $F$})|_B\in {\rm RM}$.
Given $r_0\in [4,\infty)$, suppose that 
\begin{equation}
\|U_b\|\leq\alpha_3(r_0,q_0):=\alpha_2(r_0^\prime,q_0)\leq \alpha_1(q_0).
\label{small-compo1}
\end{equation}
Then we apply \eqref{grad-new-form} for $T(t,s)^*$ 
to conclude \eqref{compo} by duality provided that $q\in (1,r_0]$ and $r\in [q,\infty)$.
%Since \eqref{inf-small}
This
%By \eqref{compo} with $r<\infty$ 
together with \eqref{decay-1} with $r=\infty$ %(resp. \eqref{grad-new-form}) 
leads us to %we are led to 
\eqref{compo} with $r=\infty$ as well. % (resp. \eqref{compo-grad}) as well.

Finally, we deduce \eqref{compo-grad}.
Given $r_0\in [4,\infty)$ and $r_1\in (1,4/3]$, assume that
\begin{equation}
\|U_b\|\leq \alpha_4(r_0,r_1,q_0):=\alpha_2\big(\min\{r_0^\prime,r_1\}, q_0\big).
\label{small-compo2}
\end{equation}
Then we use \eqref{grad-new-form} and \eqref{compo} to obtain \eqref{compo-grad} for
$1<q\leq r<\infty$ with $q\in (1,r_0]$ as well as $r\in [r_1,\infty)$.
The proof is complete.

%% 5
\section{Stability of the basic motion}
\label{stability}

The initial value problem \eqref{perturbed} is transformed into 
\begin{equation}
U(t)=T(t,s)U_0+\int_s^t T(t,\tau)H(U(\tau))\,d\tau
=:\overline{U}(t)+(\Lambda U)(t).
\label{int-NS}
\end{equation}
Look at \eqref{rhs}; since $(\eta-u)\cdot\nu=0$ at $\partial \Omega$, we observe
\begin{equation}
%\begin{split}
H(U)
=\mathbb P\left[\big\{\mbox{div $\big((u_b+u)\otimes (\eta-u)\big)$}\big\}\chi_\Omega\right]  \\
=\mathbb P\, %\left[
\mbox{div $\big\{(u_b+u)\otimes (\eta-u)\chi_\Omega\big\}$}. %\right].
%\end{split}
\label{HU}
\end{equation}
%We intend to use two forms of $H(U)$ appropriately.
We do need this divergence form especially for the nonlinear term $\eta\cdot\nabla u$ in the first integral of \eqref{nonlinear-est}
below (as in \cite{EHL14, EMT}),
while the other terms can be discussed anyway if we impose more assumptions on $\nabla u_b$ than \eqref{ass-3}.
For finding a solution to \eqref{int-NS} we adopt the function space
\begin{equation*}
\begin{split}
E:=\big\{U\in & C\big((s,\infty);\,W^{1,3}(\mathbb R^3)\cap %C((s,\infty);\,
L^\infty(\mathbb R^3)\big);\;  \\
&U(t)\in X_3(\mathbb R^3)\;\forall t\in (s,\infty),\;\lim_{t\to s}\|U\|_{E(t)}=0,\;\|U\|_E<\infty\big\}
\end{split}
\end{equation*}
with
\begin{equation*}
\begin{split}
\|U\|_{E(t)}
&:=\sup_{\tau\in (s,t)}
(\tau-s)^{1/2}\big(\|\nabla U(\tau)\|_{3,\mathbb R^3}+\|U(\tau)\|_{\infty,\mathbb R^3}\big)\quad \mbox{for $t\in (s,\infty)$}, \\
\|U\|_E
&:=\sup_{t\in (s,\infty)}\left(\|U\|_{E(t)}+\|U(t)\|_{3,\mathbb R^3}\right).
\end{split}
\end{equation*}
Then $E$ is a Banach space endowed with norm $\|\cdot\|_E$.
Let us remark that $U\in E$ already involves the boundary condition \eqref{hidden-bc2}
with $(u(t),\eta(t),\omega(t))=i(U(t))$ for each $t>s$, see \eqref{X1}, since $U(t)\in W^{1,3}(\mathbb R^3)$.
In view of \eqref{X1}, $U\in E$ implies that
\begin{equation}
(t-s)^{1/2}\big(\|\nabla u(t)\|_{3,\Omega}+\|u(t)\|_{\infty,\Omega}+|\eta(t)|+|\omega(t)|\big)\leq C\|U\|_{E(t)}\to 0
\label{equi-E0}
\end{equation}
as $t\to s$ and that
\begin{equation}
%\sup_{t>s}\left\{
(t-s)^{1/2}\big(\|\nabla u(t)\|_{3,\Omega}+\|u(t)\|_{\infty,\Omega}\big)+(1+t-s)^{1/2}\big(|\eta(t)|+|\omega(t)|\big)
+\|u(t)\|_{3,\Omega}
%\right\}
\leq C\|U\|_E
\label{equi-E}
\end{equation}
for all $(t,s)$ with $t>s$. 

By $\Lambda U$ we denote the Duhamel term in \eqref{int-NS}:
\[
(\Lambda U)(t):=\int_s^t T(t,\tau)H(U(\tau))\,d\tau.
\]
Then we see the following lemma.
\begin{lemma}
Suppose \eqref{ass-1}--\eqref{ass-2} and \eqref{ass-3}.
If $\|U_b\|\leq \alpha_1$ with $\alpha_1=\alpha_1(q_0)$ being the constant given in Proposition \ref{prop-0th}, 
see \eqref{small-0th}, then
%for every $U\in E$ 
we have 
$\Lambda U\in E$ as well as
\begin{equation}
\lim_{t\to s}\|(\Lambda U)(t)\|_{3,\mathbb R^3}=0
\label{initial}
\end{equation}
for every $U\in E$ and %we have
\begin{equation}
\|\Lambda U\|_E\leq c_1\|U_b\|^\prime\|U\|_E+c_2\|U\|_E^2
\label{est-duha}
\end{equation}
\begin{equation}
\|\Lambda U-\Lambda V\|_E
\leq \big(c_1\|U_b\|^\prime+c_2\|U\|_E+c_2\|V\|_E\big)\|U-V\|_E
\label{difference}
\end{equation}
for all $U,\, V\in E$
with some constants $c_1=c_1(q_0,\alpha_1,\beta_0,\theta)>0$ and $c_2=c_2(\alpha_1,\beta_0,\theta)$, 
where $\|U_b\|$ and $\|U_b\|^\prime$ are the constants given by \eqref{quan} and \eqref{quan2}, respectively.
Furthermore, under the condition above, %for every $U\in E$, we have 
the following additional properties hold for every $U\in E$:
(i)
Let $r\in (3,\infty)$, then
\begin{equation}
\|\nabla (\Lambda U)(t)\|_{r,\mathbb R^3}=O\big((t-s)^{-1/2}\big)
\label{add-decay}
\end{equation}
as $(t-s)\to\infty$.
%{\color{blue}(... M--20--22)}
(ii)
$\Lambda U$ is locally H\"older continuous on $(s,\infty)$ with values in $W^{1,3}(\mathbb R^3)\cap L^\infty(\mathbb R^3)$, to be precise,
\begin{equation}
\begin{split}
&\Lambda U\in C^{\theta_0}_{\rm loc}\big((s,\infty);\, X_3(\mathbb R^3)\big)\cap C^{\theta_1}_{\rm loc}\big((s,\infty);\, L^\infty(\mathbb R^3)\big), \\
&\nabla \Lambda U\in C^{\theta_1}_{\rm loc}\big((s,\infty);\, L^3(\mathbb R^3)\big),
\end{split}
\label{duha-hoel}
\end{equation}
for every $\theta_0\in (0,3/4)$ and $\theta_1\in (0,1/4)$. 
\label{lem-duha}
\end{lemma}

\begin{proof}
We are concerned only with \eqref{initial}--\eqref{est-duha} since
the other estimate \eqref{difference} is shown similarly.
Set
\begin{equation}
\begin{split}
(\Lambda_1U)(t)
&=\int_s^tT(t,\tau)\mathbb P[\{(\eta-u)\cdot\nabla u_b\}\chi_\Omega](\tau)\,d\tau  \\
&=\int_s^tT(t,\tau)\mathbb P\mbox{div $\{u_b\otimes (\eta-u)\chi_\Omega\}$}(\tau)\,d\tau, 
\end{split}
\label{duha-linear}
\end{equation}
\begin{equation}
\begin{split}
(\Lambda_2U)(t)
&=\int_s^tT(t,\tau)\mathbb P[\{(\eta-u)\cdot\nabla u\}\chi_\Omega](\tau)\,d\tau  \\
&=\int_s^tT(t,\tau)\mathbb P\mbox{div $\{u\otimes (\eta-u)\chi_\Omega\}$}(\tau)\,d\tau,
\end{split}
\label{duha-nonlinear}
\end{equation}
where \eqref{HU} is taken into account.
Let us note that both $u_b\otimes (\eta-u)\chi_\Omega$ and $u\otimes (\eta-u)\chi_\Omega$
satisfy the conditions imposed on $F$ for \eqref{compo}--\eqref{compo-grad} since
$(\eta-u)\cdot\nu=0$ at $\partial \Omega$.
In what follows we assume that
\[
\|U_b\|\leq \alpha_1(q_0)=\alpha_2(4/3,q_0)=\alpha_3(4,q_0)=\alpha_4(4,4/3,q_0),
\]
see \eqref{small-0th} and \eqref{small-compo1}--\eqref{small-compo2}.
This condition allows us to employ all the estimates obtained in Theorem \ref{evo-op} with the exponents needed below.
Let $U\in E$ and let us begin with estimate of \eqref{duha-linear}. %$\Lambda_1 U$.
%For $t-s>2$, we have
Since $u_b(t)\in L^\infty(\Omega)$, one may assume that $q_0\in (3/2,3)$ in \eqref{ass-1}.
We make use of \eqref{decay-1}--\eqref{compo-grad} along with \eqref{equi-E0}--\eqref{equi-E} to find that
\begin{equation}
\begin{split}
&\quad \|\nabla (\Lambda_1U)(t)\|_{3,\mathbb R^3}+\|(\Lambda_1U)(t)\|_{\infty,\mathbb R^3}  \\
&\leq C\int_s^{t-1}(t-\tau)^{-3/2q_0-1/2}\|u_b(\tau)\|_{q_0,\Omega}\big(|\eta(\tau)|+\|u(\tau)\|_{\infty,\Omega}\big)\,d\tau  \\
&\quad +C\int_{t-1}^t(t-\tau)^{-1/2}\|\nabla u_b(\tau)\|_{3,\Omega}\big(|\eta(\tau)|+\|u(\tau)\|_{\infty,\Omega}\big)\,d\tau  \\
&\leq C(t-s)^{-1/2}\|U_b\|^\prime\|U\|_{E(t)}
\end{split}
\label{ub-est}
\end{equation}
for $t-s>2$
by splitting the first integral further into $\int_s^{(s+t)/2}+\int_{(s+t)/2}^{t-1}$, 
where $q_0\in (3/2,3)$ allows us to obtain
sharp decay rate in \eqref{compo-grad}, and that
\begin{equation*}
\begin{split}
&\quad \|\nabla (\Lambda_1U)(t)\|_{3,\mathbb R^3}+\|(\Lambda_1U)(t)\|_{\infty,\mathbb R^3}  \\
&\leq C\int_s^t(t-\tau)^{-1/2}\|\nabla u_b(\tau)\|_{3,\Omega}\big(|\eta(\tau)|+\|u(\tau)\|_{\infty,\Omega}\big)\,d\tau  \\
&\leq C%(t-s)^{-3/2\sigma_0+1/2}
\|U_b\|^\prime\|U\|_{E(t)}
\end{split}
\end{equation*}
for $t-s\leq 2$.
Also, it is readily seen that
\begin{equation}
\begin{split}
\|(\Lambda_1U)(t)\|_{3,\mathbb R^3}
&\leq C\int_s^t(t-\tau)^{-1/2}\|u_b(\tau)\|_{3,\Omega}\big(|\eta(\tau)|+\|u(\tau)\|_{\infty,\Omega}\big)\,d\tau  \\
&\leq C\|U_b\|\|U\|_{E(t)}
\end{split}
\label{duha-1-L3}
\end{equation}
for all $t>s$. 
Collecting estimates above, we infer
\begin{equation}
\begin{split}
&\|\Lambda_1U\|_E\leq c_1\|U_b\|^\prime\|U\|_E, \\
&\lim_{t\to s}\big(\|\Lambda_1U\|_{E(t)}+\|(\Lambda_1U)(t)\|_{3,\mathbb R^3}\big)=0,
\end{split}
\label{est-duha1}
\end{equation}
for all $U\in E$ with some constant $c_1=c_1(q_0,\alpha_1,\beta_0,\theta)$.

We turn to the other integral $\Lambda_2U$ given by \eqref{duha-nonlinear}.
The computation with splitting below was not done by \cite{EMT} and this was why the decay rate of $\|\nabla u(t)\|_{3,\Omega}$
was not sharp in that literature.
From \eqref{decay-1}--\eqref{compo-grad} and \eqref{equi-E0}--\eqref{equi-E} it follows that
\begin{equation}
\begin{split}
&\quad \|\nabla (\Lambda_2U)(t)\|_{3,\mathbb R^3}+\|(\Lambda_2U)(t)\|_{\infty,\mathbb R^3}  \\
&\leq C\int_s^{(s+t)/2}(t-\tau)^{-1}\|u(\tau)\|_{3,\Omega}\big(|\eta(\tau)|+\|u(\tau)\|_{\infty,\Omega}\big)\,d\tau  \\
&\quad +C\int_{(s+t)/2}^t(t-\tau)^{-1/2}\|\nabla u(\tau)\|_{3,\Omega}\big(|\eta(\tau)|+\|u(\tau)\|_{\infty,\Omega}\big)\,d\tau  \\
&\leq C(t-s)^{-1/2}\big(\|U\|_E\|U\|_{E(t)}+\|U\|_{E(t)}^2\big)
\end{split}
\label{nonlinear-est}
\end{equation}
for all $t>s$ %$t-s>0$ 
and that
\begin{equation}
\begin{split}
\|(\Lambda_2U)(t)\|_{3,\mathbb R^3}
&\leq C\int_s^t(t-\tau)^{-1/2}\|u(\tau)\|_{3,\Omega}\big(|\eta(\tau)|+\|u(\tau)\|_{\infty,\Omega}\big)\,d\tau   \\
&\leq c_0B(1/2,1/2)\|U\|_E\|U\|_{E(t)}
\end{split}
\label{nonlinear-beta}
\end{equation}
for all $t>s$ with some constant $c_0>0$, where $B(\cdot,\cdot)$ denotes the beta function and $B(\frac{1}{2},\frac{1}{2})=\pi$.
Those estimates lead us to
\begin{equation}
\begin{split}
&\|\Lambda_2U\|_E\leq c_2\|U\|_E^2, \\
&\lim_{t\to s}\big(\|\Lambda_2U\|_{E(t)}+\|(\Lambda_2U)(t)\|_{3,\mathbb R^3}\big)=0,
\end{split}
\label{est-duha2}
\end{equation}
for all $U\in E$ with some constant $c_2=c_2(\alpha_1,\beta_0,\theta)$, 
which together with \eqref{est-duha1} completes the proof of \eqref{initial}--\eqref{est-duha}.

It remains to show the additional properties:

(i)
Look at \eqref{ub-est} and \eqref{nonlinear-est} in which $\|\nabla (\cdot)\|_{3,\mathbb R^3}$ is replaced by 
$\|\nabla (\cdot)\|_{r,\mathbb R^3}$
with $r\in (3,\infty)$, then the computations still work,
where further splitting 
$\int_{(s+t)/2}^{t-1}+\int_{t-1}^t$ is needed in \eqref{nonlinear-est}
and, in the latter integral, $(t-\tau)^{-1/2}$ must be replaced by $(t-\tau)^{-1+3/2r}$.

(ii)
Let us recall the H\"older estimate of the evolution operator obtained in Proposition \ref{hoelder-evo}.
Since the issue is merely local in time, we do not have to take care of the large time behavior and thus we have only to use
the first form of each of \eqref{duha-linear}--\eqref{duha-nonlinear}, respectively.
We employ \eqref{hoel-est} with
$(j,q,r)=(0,3,3),(0,3,\infty)$ and $(1,3,3)$ for $\Lambda_1 U$.
Concerning $\Lambda_2U$, we use \eqref{hoel-est} with
$(j,q,r)=(0,2,3),(0,2,\infty)$ and $(1,2,3)$ for $u\cdot\nabla u$, while
$(j,q,r)=(0,3,3),(0,3,\infty)$ and $(1,3,3)$ for $\eta\cdot\nabla u$. % for $\eta\cdot\nabla u$.
Note that \eqref{restr} is fulfilled for all of those $(j,q,r)$.
%{\color{blue}(... M--20--28, 35, 36 ... )}
Then the computations are the same as in the autonomous case with analytic semigroups.
The proof is complete.
\end{proof}

Set $\overline{U}(t):=T(t,s)U_0$ with $U_0\in X_3(\mathbb R^3)$,
then \eqref{sm-little} implies that 
$\|\overline{U}\|_{E(t)}\to 0$ as $t\to s$.
By taking into account \eqref{hoel-est} as well,
it is clear to see that $\overline{U}\in E$ together with
\[
\|\overline{U}\|_E\leq c_*\|U_0\|_{3,\mathbb R^3},
\]
with some constant $c_*>0$,
which follows from \eqref{decay-1}--\eqref{grad-new-form} provided $\|U_b\|\leq\alpha_1$.
With Lemma \ref{lem-duha} at hand, we easily find that %a fixed point $U$ of 
the map
\[
U\mapsto \overline{U}+\Lambda U
\]
is contractive from the closed ball $E_R=\{U\in E;\; \|U\|_E\leq R\}$
with radius %$2c_0\|U_0\|_{\mathbb R^3}$ 
\begin{equation}
R=\frac{1}{2c_2}\left(\frac{1}{2}-\sqrt{\frac{1}{4}-4c_2c_*\|U_0\|_{3,\mathbb R^3}}\right)
<4c_*\|U_0\|_{3,\mathbb R^3} %<\frac{1}{4c_2}
\label{sol-mag}
\end{equation}
into itself
provided that
\begin{equation}
\|U_b\|^\prime\leq \frac{1}{2c_1}, \qquad
\|U_0\|_{3,\mathbb R^3} <\delta:=\frac{1}{16c_2c_*}.
\label{IC-small}
\end{equation}
The smallness of the basic motion thus reads
\[
\|U_b\|^\prime\leq \alpha=\alpha(q_0,\beta_0,\theta)
:=\min\left\{\alpha_1,\,\frac{1}{2c_1}\right\},
\]
where $\alpha_1=\alpha_1(q_0)$ is the constant given in Proposition \ref{prop-0th}, see \eqref{small-0th}.
Then the fixed point $U\in E_R$ provides a solution to \eqref{int-NS} and also the initial condition
\[
\lim_{t\to s}\|U(t)-U_0\|_{3,\mathbb R^3}=0
\]
holds on account of \eqref{initial}.

Let us remark that uniqueness of solutions to \eqref{int-NS} still holds within the class $E$
rather than the ball $E_R$ with small radius 
%without smallness 
\eqref{sol-mag}
by means of standard argument as in Fujita and Kato \cite{FK}, where
the behavior
$\|U\|_{E(t)}\to 0$ for $t\to s$ plays a role.
Actually, even this behavior near the initial time is redundant for uniqueness of the solution constructed above, as pointed out by
Brezis \cite{Bre}, where the last assertion on the uniformly convergence \eqref{sm-little}
in Proposition \ref{prop-smooth} is employed.
Note, however, that uniqueness within the class $E$ is independent of the existence of solutions,
while it is not the case for the latter.

All the desired properties of the solution obtained above (except for the sharp large time behavior that will be seen below)
follow from Lemma \ref{lem-duha}
as well as several properties of the evolution operator.
By (ii) of Lemma \ref{lem-duha} together with \eqref{ass-3}, 
the term $H(U)$ given by \eqref{rhs} is locally 
H\"older continuous with values in $X_3(\mathbb R^3)$,
%the sum $X_3(\mathbb R^3)+X_{\sigma_0}(\mathbb R^3)$, 
so that the solution $U(t)$ is a strong one 
%as described in Theorem \ref{nonlinear} 
(\cite[Chapter 5, Theorem 2.3]{Ta}, \cite[Theorem 3.9]{Y})
as described in Theorem \ref{nonlinear}.

Finally, let us close the paper with verification of the sharp large time behavior, such as
$\|U(t)\|_{\infty, \mathbb R^3}=o\big((t-s)^{-1/2}\big)$,
although this is common as long as we work with underlying space in which the class of %smooth functions with compact support 
nice functions is dense.
Suppose $U_0\in X_3(\mathbb R^3)\cap X_p(\mathbb R^3)$ with some $p\in (1,3)$, then both $\overline{U}(t)$
and $(\Lambda_1U)(t)$ dacays to zero in $X_3(\mathbb R^3)$ with definite rate, say, $(t-s)^{-\gamma}$.
In fact, we have only to replace $\|u_b\|_{3,\Omega}$ by $\|u_b\|_{q_0,\Omega}$ with $q_0\in (3/2,3)$ 
in \eqref{duha-1-L3} as for $\Lambda_1U$. %{\color{blue}( ... M--20--25)}
One the other hand, it is readily seen that
\[
\|(\Lambda_2U)(t)\|_{3,\mathbb R^3}
\leq c_0B(1/2,1/2-\gamma)(t-s)^{-\gamma}\|U\|_E\sup_{\tau\in (s,t)}(\tau-s)^\gamma\|U(\tau)\|_{3,\mathbb R^3}
\]
for all $t>s$, where $c_0$ is the same constant as in \eqref{nonlinear-beta}.
Note that the constant $c_2$ in \eqref{est-duha2} should satisfy $c_2\geq \pi c_0=c_0B(\frac{1}{2},\frac{1}{2})$.
By the continuity of the beta function, one can take $\gamma>0$ so small that
$c_0B(\frac{1}{2},\frac{1}{2}-\gamma)\leq  2c_2$,
which together with \eqref{sol-mag} leads to
\[
c_0B(1/2,1/2-\gamma)\|U\|_E\leq 2c_2R<8c_2c_*\|U_0\|_{3,\mathbb R^3}<\frac{1}{2}
\]
under the condition \eqref{IC-small}.
We thus find
\[
\|U(t)\|_{3,\mathbb R^3}\leq C(t-s)^{-\gamma}\big(\|U_0\|_{3,\mathbb R^3}+\|U_0\|_{p,\mathbb R^3}\big),
\]
from which combined with the continuity of the solution map $U(s)\mapsto U$ in the sense that
\[
\sup_{t\in [s,\infty)}\|U(t)-V(t)\|_{3,\mathbb R^3}\leq C\|U(s)-V(s)\|_{3,\mathbb R^3}
\]
as well as denseness of $X_3(\mathbb R^3)\cap X_p(\mathbb R^3)$ in $X_3(\mathbb R^3)$,
we conclude that 
\begin{equation}
%$\displaystyle{
\lim_{t-s\to\infty}\|U(t)\|_{3,\mathbb R^3}=0.
\label{little-L3}
\end{equation}
Several papers (including mine) on the Navier-Stokes claim that \eqref{little-L3} is accomplished provided initial data are still smaller,
however, if we look carefully into estimates as above, then we see that further smallness than \eqref{IC-small} is not
needed.
This observation is due to Tomoki Takahashi, who is the author of \cite{ToT21}.

Once we have \eqref{little-L3}, we can deduce the other decay properties by following the argument as in %Enomoto and Shibata 
\cite{EShi05}, see also \cite{ToT21}.
Let $\tau_*>s$. Using the equation
\begin{equation}
U(t)=T(t,\tau_*)U(\tau_*)+\int_{\tau_*}^t T(t,\tau)H(U(\tau))\,d\tau
\label{further-IE}
\end{equation}
and performing exactly the same computations as in the proof of Lemma \ref{lem-duha}, we infer
\begin{equation*}
\begin{split}
&\quad (t-\tau_*)^{1/2}\big(\|U(t)\|_{\infty,\mathbb R^3}+\|\nabla U(t)\|_{3,\mathbb R^3}\big)   \\
&\leq C\|U(\tau_*)\|_{3,\mathbb R^3}
+\big(c_1\|U_b\|^\prime+c_2\|U\|_E\big)\sup_{\tau\in (\tau_*,t)}(\tau-\tau_*)^{1/2}\|U(\tau)\|_{\infty,\mathbb R^3}  \\
&\leq C\|U(\tau_*)\|_{3,\mathbb R^3}
+\frac{3}{4}\sup_{\tau\in (\tau_*,t)}(\tau-\tau_*)^{1/2}\|U(\tau)\|_{\infty,\mathbb R^3} 
\end{split}
\end{equation*}
for all $t>\tau_*$ in view of \eqref{sol-mag}--\eqref{IC-small}, where 
the constants $c_1$ and $c_2$ are the same as those in \eqref{est-duha}--\eqref{difference}.
Given $\varepsilon>0$ arbitrarily, we take $\tau_*-s$ so large that the first term of the right-hand side above is less than
$\varepsilon$, which is indeed possible because of \eqref{little-L3}.
We then have
$(t-\tau_*)^{1/2}\big(\|U(t)\|_{\infty,\mathbb R^3}+\|\nabla U(t)\|_{3,\mathbb R^3}\big)\leq 4\varepsilon$
for all $t>\tau_*$.
If $t-s>2(\tau_*-s)$, then we get
\[
\left(\frac{t-s}{2}\right)^{1/2}
\big(\|U(t)\|_{\infty,\mathbb R^3}+\|\nabla U(t)\|_{3,\mathbb R^3}\big)\leq 4\varepsilon,
\]
which concludes \eqref{main-decay} except for $\|\nabla u(t)\|_{r,\Omega}$ with $r\in (3,\infty)$.
For such $r$, finally,
we use \eqref{further-IE} with $\tau_*$ replaced by $\frac{t+s}{2}$ and compute it
as in deduction of \eqref{add-decay} to find
\begin{equation*}
\begin{split}
\|\nabla U(t)\|_{r,\mathbb R^3}
&\leq C(t-s)^{-1/2}\|U((t+s)/2)\|_{3,\mathbb R^3}  \\
&\quad +C(t-s)^{-1/2}\big(\|U_b\|^\prime+\|U\|_E\big)\sup_{\tau>(t+s)/2}(\tau -s)^{1/2}\|U(\tau)\|_{\infty,\mathbb R^3}
\end{split}
\end{equation*}
for $t-s>2$.
Hence, \eqref{main-decay} with $q=\infty$ and \eqref{little-L3} yield \eqref{main-decay} with $r\in (3,\infty)$.
The proof of Theorem \ref{nonlinear} is complete.
\bigskip

\noindent
{\bf Acknowledgments}.
I am grateful to Professor Giovanni P. Galdi for stimulating discussion about the issue of the paper.
This work is partially supported by the Grant-in-aid for Scientific Research 22K03372 from JSPS.
\bigskip

\noindent
{\bf Declarations}

\noindent
{\bf Conflict of interest.}
The author states that there is no conflict of interest.

%% ref

\end{document}